\documentclass[11pt,a4paper,leqno]{article}
\usepackage{a4wide}
\usepackage{latexsym}
\usepackage{amsmath}
\usepackage{amssymb}
\newtheorem{defin}{Definition}
\newtheorem{lemma}{Lemma}
\newtheorem{prop}{Proposition}
\newtheorem{theo}{Theorem}

\newenvironment{proof}{\medskip\par\noindent{\bf Proof}}{\hfill $\Box$
\medskip\par}
\begin{document}
\title{On parametric Gevrey asymptotics for some nonlinear initial value Cauchy problems}
\author{{\bf A. Lastra\footnote{The author is partially supported by the project MTM2012-31439 of Ministerio de Ciencia e
Innovacion, Spain}, S. Malek\footnote{The author is partially supported by the french ANR-10-JCJC 0105 project and the PHC
Polonium 2013 project No. 28217SG.}}\\
University of Alcal\'{a}, Departamento de F\'{i}sica y Matem\'{a}ticas,\\
Ap. de Correos 20, E-28871 Alcal\'{a} de Henares (Madrid), Spain,\\
University of Lille 1, Laboratoire Paul Painlev\'e,\\
59655 Villeneuve d'Ascq cedex, France,\\
{\tt alberto.lastra@uah.es}\\
{\tt Stephane.Malek@math.univ-lille1.fr }}
\date{March, 7 2014}
\maketitle
\thispagestyle{empty}
{ \small \begin{center}
{\bf Abstract}
\end{center}
We study a nonlinear initial value Cauchy problem depending upon a complex perturbation parameter $\epsilon$ with vanishing initial
data at complex time $t=0$ and whose coefficients depend analytically on $(\epsilon,t)$ near the origin in
$\mathbb{C}^{2}$ and are bounded holomorphic on some horizontal strip in $\mathbb{C}$ w.r.t the space variable. This problem
is assumed to be non-Kowalevskian in time $t$, therefore analytic solutions at $t=0$ cannot be expected in general. Nevertheless,
we are able to construct a family of
actual holomorphic solutions defined on a common bounded open sector with vertex at 0 in time and on the given strip above in space,
when the complex parameter $\epsilon$ belongs to a suitably chosen set of open bounded sectors whose union form a covering of some
neighborhood $\Omega$ of 0 in $\mathbb{C}^{\ast}$. These solutions are achieved by means of Laplace and Fourier inverse
transforms of some common $\epsilon-$depending function on
$\mathbb{C} \times \mathbb{R}$, analytic near the origin and with exponential growth on some unbounded sectors
with appropriate bisecting directions in the first variable and exponential decay in the second, when the perturbation
parameter belongs to $\Omega$. Moreover, these solutions satisfy the remarkable property that the difference
between any two of them is exponentially flat for some
integer order w.r.t $\epsilon$. With the help of the classical Ramis-Sibuya theorem, we obtain the existence of a formal series
(generally divergent) in $\epsilon$ which is the common Gevrey asymptotic expansion of the built up actual solutions considered
above.\medskip

}
\bigskip \bigskip

\section{Introduction}

In this paper, we consider a family of parameter depending nonlinear initial value Cauchy problems of the form
\begin{multline}
Q(\partial_{z})(\partial_{t}u(t,z,\epsilon)) = (Q_{1}(\partial_{z})u(t,z,\epsilon))(Q_{2}(\partial_{z})u(t,z,\epsilon))
+ \sum_{l=1}^{D} \epsilon^{\Delta_{l}}t^{d_l}\partial_{t}^{\delta_l}R_{l}(\partial_{z})u(t,z,\epsilon)\\
+ c_{0}(t,z,\epsilon)R_{0}(\partial_{z})u(t,z,\epsilon) + f(t,z,\epsilon) \label{ICP_main_intro}
\end{multline}
for given vanishing initial data $u(0,z,\epsilon) \equiv 0$, where $D \geq 2$, $\Delta_{l},d_{l},\delta_{l}$, $1 \leq l \leq D$ are integers
which satisfy the inequalities (\ref{assum_delta_l}), (\ref{assum_d_delta_Delta}) and (\ref{constraints_k_Borel_equation_for_u_p}).
Besides, $Q(X),Q_{1}(X),Q_{2}(X),R_{l}(X)$, $0 \leq l \leq D$ are polynomials submitted to the constraints
(\ref{assum_deg_Q_R}) and which possess the crucial analytic property (\ref{assum_Q_R_D}). The coefficient $c_{0}(t,z,\epsilon)$ and the
forcing term $f(t,z,\epsilon)$ are bounded holomorphic functions on a product $D(0,r) \times H_{\beta} \times D(0,\epsilon_{0})$,
where $D(0,r)$ (resp. $D(0,\epsilon_{0})$) is a disc centered at 0 with small radius $r>0$ (resp. $\epsilon_{0}>0$) and
$H_{\beta} \subset \mathbb{C}$ is some strip of width $\beta>0$ (see (\ref{strip_H_beta})).
In order to avoid cumbersome statements and to improve the readability of the computations, we have restricted our study to a
quadratic nonlinearity and monomial coefficients in $t$ in front of the derivatives with respect to $t$ and $z$ but the method
described here can also be extended to higher order nonlinearities, with polynomial coefficients w.r.t $t$ in the linear part on the
right handside of the equation (\ref{ICP_main_intro}).

This work can be seen as a continuation of the study described in \cite{ma2} where the second author has studied
nonlinear integro-differential initial values problems with the shape
\begin{equation}
R(\partial_{z})P(\partial_{t},\partial_{z})Y(t,z) = \int_{0}^{t} b(t-s,z)\partial_{z}^{s_0}Y(s,z) ds + \int_{0}^{t}
\partial_{z}^{s_1}Y(t-s,z)\partial_{z}^{s_2}Y(s,z) ds \label{int_diff_intro}
\end{equation}
where $R(X) \in \mathbb{C}[X]$, $P(T,X) \in \mathbb{C}[T,X]$ and $s_{0},s_{1},s_{2} \geq 0$ are non negative integers.
The coefficient $b(t,z) = \sum_{k \in I} b_{k}(z)t^{k}$ is a polynomial in $t$
and its coefficients $b_{k}(z)$ are Fourier inverse transform of some function $\mathfrak{b}_{k}(m)$ belonging to a Banach space
$E_{(\beta,\mu)}$ of exponentially decreasing continuous functions on $\mathbb{R}$ (see Definition 2) and define
bounded holomorphic functions on any strip $H_{\beta'}$, $0<\beta'<\beta$. The initial conditions are defined by
$Y(0,z) = Y_{0}(z)$, $(\partial_{t}^{j}Y)(0,z) \equiv 0$, for all $1 \leq j \leq \mathrm{deg}_{T}P(T,X)-1$,
where $Y_{0}$ is also assumed to be the Fourier inverse transform of some $\mathfrak{Y}_{0}(m)$ belonging to $E_{(\beta,\mu)}$.
We focused on the case when the degree of $R(X)P(T,X)$ with
respect to $T$ is smaller than its degree in $X$. In that case the classical Cauchy-Kowalevski theorem (see \cite{fol}) cannot be applied
and the unique formal power series solution $\hat{Y}(t,z) = \sum_{l \geq 0} Y_{l}(z)t^{l}$, with coefficients belonging to
the Banach space of bounded holomorphic functions on $H_{\beta'}$ equipped with the sup norm, is in general divergent.
Nevertheless, under suitable constraints on the roots of the polynomial $T \mapsto P(T^2,im)$ and for sufficiently small data
$||\mathfrak{b}_{k}||_{(\beta,\mu)}$, $||\mathfrak{Y}_{0}||_{(\beta,\mu)}$, one can construct by means of classical
Borel-Laplace procedure and Fourier inverse transform an actual holomorphic solution $Y(t,z)$ on
$\mathbb{C}_{+} \times H_{\beta'}$ of (\ref{int_diff_intro})
for the given initial data ($\mathbb{C}_{+}$ denotes the set of complex numbers $t$ such that $\mathrm{Re}(t) > 0$), which
possess the formal series $\hat{Y}$ as Gevrey asymptotic expansion of order 1 as $t$ tends to 0 (see Definition 6).

Compared to the work \cite{ma2}, the problem (\ref{ICP_main_intro}) now involves an additional complex parameter $\epsilon$.
Provided that $\delta_{D}+\mathrm{deg}(R_{D}) > \mathrm{deg}(Q)+1$ holds, the problem (\ref{ICP_main_intro}) is singularly
perturbed in the parameter $\epsilon$ and belongs to a class of so-called PDEs with irregular singularity at $t=0$ in the sense
of \cite{man2}. In the paper \cite{ma1}, the second author has already considered a similar problem of the form
\begin{equation}
 \epsilon t^{2} \partial_{t}\partial_{z}^{S}X_{p}(t,z,\epsilon) = F(t,z,\epsilon,\partial_{t},\partial_{z})X_{p}(t,z,\epsilon) +
P(t,z,\epsilon,X_{p}(t,z,\epsilon)) \label{SPCP_intro}
\end{equation}
for given initial data
\begin{equation}
(\partial_{z}^{j}X_{p})(t,0,\epsilon)=\phi_{j,p}(t,\epsilon) \ \ , \ \ 0 \leq p \leq \varsigma-1, 0 \leq j \leq S-1, \label{SPCP_init_cond_intro}
\end{equation}
where $S,\varsigma \geq 2$ are some positive integers, $F$ is some differential operator with polynomial coefficients and
$P$ a polynomial. The initial data $\phi_{j,p}(t,\epsilon)$
were assumed to be holomorphic on products $\mathcal{T} \times \mathcal{E}_{p} \subset \mathbb{C}^{2}$ for some sector
$\mathcal{T}$ centered at 0 and where $\underline{\mathcal{E}} = \{ \mathcal{E}_{p} \}_{0 \leq p \leq \varsigma-1}$
denotes a family of open bounded sectors with aperture larger than $\pi$ which form a so-called good covering in
$\mathbb{C}^{\ast}$ (see Definition 4). 
Under convenient assumptions on the shape of the equation (\ref{SPCP_intro}) and on the initial data (\ref{SPCP_init_cond_intro}),
the existence of a formal series $\hat{X}(\epsilon) = \sum_{k \geq 0} h_{k} \epsilon^{k}/k!$ solution of (\ref{SPCP_intro}) is
established with coefficients $h_{k}$
belonging to the Banach space $\mathbb{F}$ of bounded holomorphic functions on $\mathcal{T} \times D(0,\delta)$
(for some $\delta>0$ small enough) equipped with the sup norm. This formal series $\hat{X}(\epsilon)$ is the
Gevrey asymptotic expansion of order 1 of actual holomorphic solutions $X_{p}$ of (\ref{SPCP_intro}), (\ref{SPCP_init_cond_intro}) on
$\mathcal{E}_{p}$ as $\mathbb{F}-$valued functions, for all $0 \leq p \leq \varsigma-1$ (see Definition 6).

In this work we address the same queries as in \cite{ma1}, \cite{ma2}, namely our main purpose is the construction
of actual holomorphic solutions $u_{p}(t,z,\epsilon)$ to the problem (\ref{ICP_main_intro}) on domains
$\mathcal{T} \times H_{\beta'} \times \mathcal{E}_{p}$ using some Borel-Laplace procedure and Fourier inverse transform
and the analysis of their asymptotic expansions as $\epsilon$ tends to 0. More specifically, we can present our main statements
as follows.\medskip

\noindent {\bf Main results} {\it Let $k \geq 1$ be the integer such that $d_{D} = (\delta_{D}-1)(k+1)$ holds. We choose a family
$\underline{\mathcal{E}} = \{ \mathcal{E}_{p} \}_{0 \leq p \leq \varsigma-1}$ of sectors with aperture slightly larger than $\pi/k$
which defines a good covering of $\mathbb{C}^{\ast}$ and a set of adequate directions $\mathfrak{d}_{p} \in \mathbb{R}$,
$0 \leq p \leq \varsigma-1$ for which the constraints (\ref{root_cond_1_in_defin}), (\ref{root_cond_2_in_defin}) hold. We also take
an open bounded sector $\mathcal{T}$ centered at 0 such that for every $0 \leq p \leq \varsigma-1$, the product
$\epsilon t$ belongs to a sector with direction $\mathfrak{d}_{p}$ and aperture slightly larger than $\pi/k$, for all
$\epsilon \in \mathcal{E}_{p}$,
all $t \in \mathcal{T}$. We make the assumption that the coefficient $c_{0}(t,z,\epsilon)$ and the forcing term $f(t,z,\epsilon)$ can
be written as convergent series of the special form
$$ c_{0}(t,z,\epsilon) = \sum_{n \geq 0} c_{0,n}(z,\epsilon) (\epsilon t)^{n} \ \ , \ \
f(t,z,\epsilon) = \sum_{n \geq 1} f_{n}(z,\epsilon) (\epsilon t)^{n},$$
on a domain $D(0,r) \times H_{\beta'} \times D(0,\epsilon_{0})$ (where $H_{\beta'}$ is a strip of width $\beta'$)
such that $\mathcal{T} \subset D(0,r)$, $\cup_{0 \leq p \leq \varsigma-1} \mathcal{E}_{p} \subset D(0,\epsilon_{0})$ and
$0 < \beta' < \beta$ are given positive real numbers.
The coefficients $c_{0,0}(z,\epsilon)$, $c_{0,n}(z,\epsilon)$ and $f_{n}(z,\epsilon)$, $n \geq 1$, are supposed to be inverse Fourier
transform of functions $m \mapsto C_{0,0}(m,\epsilon)$, $m \mapsto C_{0,n}(m,\epsilon)$ and $m \mapsto F_{n}(m,\epsilon)$ that
belong to the Banach space $E_{(\beta,\mu)}$ for some $\mu > \max( \mathrm{deg}(Q_{1})+1, \mathrm{deg}(Q_{2})+1 )$
and that depend holomorphically on $\epsilon$ in $D(0,\epsilon_{0})$.

Our first result stated in Theorem 1 claims that if the norm
$||C_{0,0}(m,\epsilon)||_{(\beta,\mu)}$ and the radius $\epsilon_{0}$ are chosen small enough, then
we can construct a family of holomorphic bounded functions $u_{p}(t,z,\epsilon)$, $0 \leq p \leq \varsigma-1$, defined on the
products $\mathcal{T} \times H_{\beta'} \times \mathcal{E}_{p}$, which solves the problem
(\ref{ICP_main_intro}) with vanishing initial data $u_{p}(0,z,\epsilon) \equiv 0$ and
which can be written as Laplace-Fourier transform
$$ u_{p}(t,z,\epsilon) = \frac{k}{(2\pi)^{1/2}}\int_{-\infty}^{+\infty}
\int_{L_{\gamma_{p}}}
\omega_{k}^{\mathfrak{d}_p}(u,m,\epsilon) e^{-(\frac{u}{\epsilon t})^{k}} e^{izm} \frac{du}{u} dm.
$$
where the inner integration is made along some halfline $L_{\gamma_{p}} \subset S_{\mathfrak{d}_{p}}$ and
$S_{\mathfrak{d}_p}$ is an unbounded sector
with bisecting direction $\mathfrak{d}_{p}$, small aperture and where $\omega_{k}^{\mathfrak{d}_p}(u,m,\epsilon)$ denotes a
function with at most exponential growth of order $k$ in $u/\epsilon$ and exponential decay in $m \in \mathbb{R}$ which
satisfies the estimates
(\ref{|omega_k_d_p|<})}.

{\it Our second main result, described in Theorem 2, asserts that the functions $u_{p}$, $0 \leq p \leq \varsigma-1$,
turn out to be the $k-$sums (in the sense of Definition 6) on $\mathcal{E}_{p}$ of a common formal power series
$$ \hat{u}(t,z,\epsilon) = \sum_{m \geq 0} h_{m}(t,z) \frac{\epsilon^{m}}{m!} \in \mathbb{F}[[\epsilon]]$$
where $\mathbb{F}$ is the Banach space of bounded holomorphic functions on $\mathcal{T} \times H_{\beta'}$ equipped with the
sup norm.}\medskip

It is worth remarking that when $\mathrm{deg}(Q)+1 > \delta_{D} + \mathrm{deg}(R_{D})$, the equation
(\ref{ICP_main_intro}) is not singularly perturbed in $\epsilon$ and possess no irregular singularity at $t=0$. However, the
asymptotic expansion $\hat{u}$ of $u_{p}$ as $\epsilon$ tends to 0 on $\mathcal{E}_{p}$ remains divergent in general. The
reason for this phenomenon to appear relies on the way one constructs the actual solutions $u_{p}$ as Laplace transforms of
order $k$ in the new variable
$\epsilon t$ and from the fact that for any fixed $\epsilon \in D(0,\epsilon_{0}) \setminus \{ 0 \}$, the problem
(\ref{ICP_main_intro}) is not Kowalevskian with respect to $t$ at 0 (meaning that formal series solutions
$\hat{v}(t,z,\epsilon) = \sum_{n \geq 1} v_{n}(z,\epsilon)t^{n}$, with coefficients $z \mapsto v_{n}(z,\epsilon)$ bounded holomorphic
on $H_{\beta'}$, are in general divergent, as a consequence of Propositions 8 and 9) as it was already the case in our
previous paper \cite{ma2}.\medskip
 
The Cauchy problem (\ref{ICP_main_intro}) we consider here comes within the new trend of research concerning Borel-Laplace
summability procedures applied to partial differential differential equation going back to the seminal work of D. Lutz, M. Miyake
and R. Sch\"{a}fke on the linear complex heat equation, see \cite{lumisc}. We quote below some important results in this field not
pretending to be exhaustive. This construction of Borel-Laplace $k-$summable or even
multi-summable formal series solutions has been extended to general linear PDEs in two complex variables with constant
coefficients by W. Balser in \cite{ba3} and \cite{ba4} provided that their initial data are analytic functions near the origin that can be
analytically continued with exponential growth on some unbounded sectors. A similar result has been obtain for the so-called
fractional linear PDEs with non-integer derivatives by S. Michalik, see \cite{mi}. Latter on, linear complex heat like equations
with variable coefficients have been explored by several authors, see \cite{balo}, \cite{copata}, \cite{ly2}.
Recently, general linear PDEs with time dependent coefficients taking for granted that their initial data
are entire functions in $\mathbb{C}^{N}$, $N \geq 1$, have been investigated by H. Tahara and H. Yamazawa in \cite{taya}.
In the context of nonlinear PDEs, we mention the work \cite{ly1} of G. Lysik who constructed summable formal solutions of the
one dimensional Burgers equations with the help of the so-called Cole-Hopf transform. We also point out that
O. Costin and S. Tanveer have constructed summable formal series in time variable to the celebrated 3D Navier Stokes
equations in \cite{cota2}. We also refer to the work of S. Ouchi who constructed multisummable formal solutions to
nonlinear PDEs which come from perturbations of ordinary differential equations, see \cite{ou}. We also mention the fact that,
these last years, a lot of attention has been payed to singularly perturbed PDEs in the complex domain partly motived by a
conjecture of B. Dubrovin which concerns the question of universal behaviour of generic solutions near gradient
catastrophe of singularly Hamiltonian perturbations of first order hyperbolic equations, see \cite{du}. In this active direction,
we refer namely to the works of B. Dubrovin and M. Elaeva who investigated the case of generalized Burgers equations in \cite{duel} and
of T. Claeys and T. Grava in \cite{clgr} who solved the problem for KdV equations. We indicate the recent important studies of
T. Koike on Garnier systems, \cite{ko1}, \cite{ko2} and of S. Hirose on the reduction of general singularly perturbed holonomic
systems in two complex variables to Pearcy systems normal forms, \cite{hi}.\medskip

In the sequel, we explain our principal intermediate key results and the arguments needed in their proofs. 
In a first part, we depart from an auxiliary parameter depending initial value differential and convolution equation which is singular
in its perturbation parameter $\epsilon$ at 0, see (\ref{SCP}). This equation is formally constructed by making the change of
variable $T=\epsilon t$ in the equation (\ref{ICP_main_intro}) (as done in our previous works \cite{ma1}, \cite{lamasa1}) and
by taking the Fourier transform with respect to the variable $z$. We construct a formal power series solution
$\hat{U}(T,m,\epsilon) = \sum_{n \geq 1} U_{n}(m,\epsilon)T^{n}$ of (\ref{SCP}) whose coefficients
$m \mapsto U_{n}(m,\epsilon)$ depend holomorphically on $\epsilon \in \mathbb{C}^{\ast}$ near the origin and belong
to a Banach space $E_{(\beta,\mu)}$ of continuous function with exponential decay on $\mathbb{R}$ introduced in
the paper \cite{cota2} by O. Costin and S. Tanveer. This series turns out to be in general divergent as we will see below.

In the next step, we follow the strategy developped recently by H. Tahara and H. Yamazawa in \cite{taya}, namely we multiply
each hand side of (\ref{SCP}) by the power $T^{k}$ which transforms it into an equation (\ref{SCP_irregular}) which involves only
differential operators in $T$ of irregular type at $T=0$ of the form $T^{\beta}\partial_{T}$ with $\beta \geq k+1$
due to our assumption (\ref{assum_dl_delta_l_Delta_l}) on the shape of the equation (\ref{SCP}).

Then, we apply a formal Borel transform of order $k$ (defined as a slightly modified version of the classical Borel transform of
order $k$ from the reference book \cite{ba}), that we call $m_{k}-$Borel transform in Definition 3, to the formal series
$\hat{U}$ with respect to $T$, denoted
$$ \omega_{k}(\tau,m,\epsilon) = \sum_{n \geq 1} U_{n}(m,\epsilon) \frac{\tau^k}{\Gamma(\frac{n}{k})}. $$
From the commutation rules of the $m_{k}-$Borel transform with respect to the weighted convolution product $\star$
of formal series (introduced in Proposition 5) and the differential operators $T^{\beta}\partial_{T}$ for $\beta \geq k+1$ described in
Proposition 6, we get that $\omega_{k}(\tau,m,\epsilon)$ formally solves
a convolution equation in both variables $\tau$ and $m$, see (\ref{k_Borel_equation}).

Under some size constraint on the $E_{(\beta,\mu)}-$norm of the constant term $C_{0,0}$ of one
coefficient of the equation (\ref{k_Borel_equation}) and for all $\epsilon \in \mathbb{C}^{\ast}$ close enough to 0,
we show that $\omega_{k}(\tau,m,\epsilon)$ is actually convergent for $\tau$
on some fixed neighborhood of 0 and can be extended to a holomorphic functions $\omega_{k}^{d}(\tau,m,\epsilon)$
on unbounded sectors $S_{d}$ centered at zero with bisecting direction $d$ and tiny aperture provided that $S_{d}$ stays away from
the roots of some polynomial $P_{m}(\tau)$, for all $m \in \mathbb{R}$. Besides, the function
$\omega_{k}^{d}(\tau,m,\epsilon)$ satisfies estimates of the form : there exist constants $\nu>0$ and $\varpi_{d}>0$ with
$$
|\omega_{k}^{d}(\tau,m,\epsilon)| \leq \varpi_{d}(1+|m|)^{-\mu}e^{-\beta|m|}
\frac{ |\frac{\tau}{\epsilon}|}{1 + |\frac{\tau}{\epsilon}|^{2k}} \exp( \nu |\frac{\tau}{\epsilon}|^{k})
$$
for all $\tau \in S_{d}$, $m \in \mathbb{R}$, all $\epsilon \in \mathbb{C}^{\ast}$ near the origin (see Proposition 9). The proof
rests on a fixed point argument in some Banach space of
holomorphic functions $F_{(\nu,\beta,\mu,k,\epsilon)}^{d}$ studied in Section 2. It is worth noting that the formal series
$\hat{U}(T,m,\epsilon)$ diverges since the function $\omega_{k}(\tau,m,\epsilon)$ cannot in general be extended
everywhere on $\mathbb{C}$ w.r.t $\tau$. But, as a result, we get that these series $\hat{U}$ are $m_{k}-$summable w.r.t
$T$ (see Definition 3) in all the directions $d$ chosen as above. In other words, some Laplace transform of
order $k$ of $\omega_{k}^{d}$ denoted $U^{d}(T,m,\epsilon)$ can be constructed for all $T$ belonging to a sector
$S_{d,k,h|\epsilon|}$ with bisecting direction $d$, aperture slightly larger than $\pi/k$ and radius $h|\epsilon|$ (for some $h>0$). This
function $T \mapsto U^{d}(T,m,\epsilon)$ is the unique $E_{(\beta,\mu)}-$valued map which admits $\hat{U}(T,m,\epsilon)$ as
Gevrey asymptotic expansion of order $1/k$ on $S_{d,k,h|\epsilon|}$. Moreover, $U^{d}(T,m,\epsilon)$ solves the auxiliary
problem (\ref{SCP}) with vanishing initial data $U^{d}(0,m,\epsilon)$, see Proposition 10.

In Theorem 1, we construct a family of actual bounded holomorphic solutions $u_{p}(t,z,\epsilon)$, $0 \leq p \leq \varsigma-1$
of our original problem (\ref{ICP_main_intro}) on domains of the form $\mathcal{T} \times H_{\beta'} \times \mathcal{E}_{p}$.
The sectors $\mathcal{E}_{p}$, $0 \leq p \leq \varsigma-1$ constitute a so-called good covering in $\mathbb{C}^{\ast}$
(Definition 4). The strip $H_{\beta'}$ has width $0<\beta'<\beta$ and $\mathcal{T}$ is a fixed bounded sector centered at 0
which fulfills the constraint $\epsilon t \in S_{\mathfrak{d}_{p},k}$ for all $\epsilon \in \mathcal{E}_{p}$, $t \in \mathcal{T}$,
and $S_{\mathfrak{d}_{p},k}$ is a sector of bisecting direction $\mathfrak{d}_{p}$ and
aperture slightly larger than $\pi/k$ where $\mathfrak{d}_{p}$ are suitable directions for which the unbounded sectors
$S_{\mathfrak{d}_{p}}$ with small aperture and bisecting direction $\mathfrak{d}_{p}$ satisfy the restrictions described above.
Namely, the functions $u_{p}$ are set as Fourier inverse transforms of $U^{\mathfrak{d}_{p}}$,
$$ u_{p}(t,z,\epsilon) = \mathcal{F}^{-1}(m \mapsto U^{\mathfrak{d}_{p}}(\epsilon t,m,\epsilon) )(z) $$
where the definition of $\mathcal{F}^{-1}$ is pointed out in Proposition 7. In addition to that, one can prove that the difference of any
two neighboring functions $u_{p+1}(t,z,\epsilon)-u_{p}(t,z,\epsilon)$ tends to zero as $\epsilon \rightarrow 0$ on
$\mathcal{E}_{p} \cap \mathcal{E}_{p+1}$ faster than a function with exponential decay of order $k$, uniformly w.r.t.
$t \in \mathcal{T}$ and $z \in H_{\beta'}$, see (\ref{exp_small_difference_u_p}).

The last section of the paper is devoted to deal with this latter growth information in order to show the existence of a common
asymptotic expansion
$\hat{u}(t,z,\epsilon) = \sum_{m \geq 0} h_{m}(t,z) \epsilon^{m}/m!$ of Gevrey order $1/k$
for all the functions $u_{p}(t,z,\epsilon)$ as $\epsilon$ tends to 0 on $\mathcal{E}_{p}$, uniformly w.r.t.
$t \in \mathcal{T}$ and $z \in H_{\beta'}$, see Theorem 2. The key tool in proving the result is the classical
Ramis-Sibuya theorem (Theorem (RS)).\bigskip

\noindent The layout of this work reads as follows.\\
In Section 2, we define some weighted parameter depending Banach spaces of continuous functions on
$\mathbb{C} \times \mathbb{R}$ with exponential growth on sectors w.r.t the first variable and exponential decay on $\mathbb{R}$
w.r.t the second one. We study the continuity properties of several kind of linear and nonlinear integral operators acting on these spaces 
that will be useful in Section 4.\\
In Section 3, we give a definition of $k-$summability (that we call $m_{k}-$summability) which is a minor modification of the classical one
given in the textbook \cite{ba} and which is appropriate for the problem we have to deal with. We also give conditions for the set of
$m_{k}-$sums of formal series to be a differential algebra. This fact will be important in the next section where we construct
actual solutions
of the auxiliary equation (\ref{SCP}). We provide explicit commutation formulas for the $m_{k}-$Borel transform w.r.t products
and differential operators of irregular type.\\
In Section 4, we introduce an auxiliary differential and convolution problem (\ref{SCP}) for which we construct a formal solution.
We show that the $m_{k}-$Borel transform of this formal solution satisfies a convolution problem (\ref{k_Borel_equation}).
Under suitable assumptions, we can solve uniquely this latter problem in the Banach spaces described in Section 2
using some fixed point theorem argument. Then, applying Laplace transform, we can give a uniquely determined actual solution
to (\ref{SCP}) having the formal solution mentioned above as Gevrey asymptotic expansion.\\
In Section 5, with the help of Section 4, we build a family of actual holomorphic solutions to our initial Cauchy
problem (\ref{ICP_main_intro}) on a full neighborhood of the origin in $\mathbb{C}^{\ast}$ w.r.t the perturbation parameter
$\epsilon$. We show that the difference of any two neighboring solutions is exponentially flat for some integer order in
$\epsilon$ (Theorem 1).\\
In Section 6, we show that the actual solutions constructed in Section 5 share a common formal series as
Gevrey asymptotic expansion as $\epsilon$ tends to 0 on sectors (Theorem 2). The result relies on the classical so-called
Ramis-Sibuya theorem.

\section{Banach spaces functions with exponential growth and decay}

We denote by $D(0,r)$ the open disc centered at $0$ with radius $r>0$ in $\mathbb{C}$ and by $\bar{D}(0,r)$ its closure. Let 
$S_{d}$ be an open unbounded sector in direction $d \in \mathbb{R}$ and $\mathcal{E}$ be an
open sector with finite radius $r_{\mathcal{E}}$, both centered at $0$ in $\mathbb{C}$. By
convention, these sectors do not contain the origin in $\mathbb{C}$.

\begin{defin} Let $\nu,\beta,\mu>0$ and $\rho>0$ be positive real numbers. Let $k \geq 1$ be an integer and
let $\epsilon \in \mathcal{E}$. We denote
$F_{(\nu,\beta,\mu,k,\epsilon)}^{d}$ the vector space of continuous functions $(\tau,m) \mapsto h(\tau,m)$ on
$(\bar{D}(0,\rho) \cup S_{d}) \times \mathbb{R}$, which are holomorphic with respect to $\tau$ on $D(0,\rho) \cup S_{d}$ and such that
$$ ||h(\tau,m)||_{(\nu,\beta,\mu,k,\epsilon)} =
\sup_{\tau \in \bar{D}(0,\rho) \cup S_{d},m \in \mathbb{R}} (1+|m|)^{\mu}
\frac{1 + |\frac{\tau}{\epsilon}|^{2k}}{|\frac{\tau}{\epsilon}|}\exp( \beta|m| - \nu|\frac{\tau}{\epsilon}|^{k} ) |h(\tau,m)|$$
is finite. One can check that the normed space
$(F_{(\nu,\beta,\mu,k,\epsilon)}^{d},||.||_{(\nu,\beta,\mu,k,\epsilon)})$ is a Banach space.
\end{defin}
{\bf Remark:} These norms are appropriate modifications of the norms defined by O. Costin and S. Tanveer in \cite{cota2} and
by the second the author in \cite{ma1} and \cite{ma2}.\medskip

Throughout the whole section, we assume $\epsilon \in \mathcal{E}$, $\mu,\beta,\nu>0$ are fixed. In the next lemma, we check
the continuity property by multiplication operation with bounded functions.

\begin{lemma} Let $(\tau,m) \mapsto a(\tau,m)$ be a bounded continuous function on
$(\bar{D}(0,\rho) \cup S_{d}) \times \mathbb{R}$, which is holomorphic with respect to $\tau$ on $D(0,\rho) \cup S_{d}$. Then, we have
\begin{equation}
|| a(\tau,m) h(\tau,m) ||_{(\nu,\beta,\mu,k,\epsilon)} \leq
\left( \sup_{\tau \in \bar{D}(0,\rho) \cup S_{d},m \in \mathbb{R}} |a(\tau,m)| \right)
||h(\tau,m)||_{(\nu,\beta,\mu,k,\epsilon)}
\end{equation}
for all $h(\tau,m) \in F_{(\nu,\beta,\mu,k,\epsilon)}^{d}$.
\end{lemma}

In the next proposition, we study the continuity property of some convolution operators acting on the latter Banach spaces.

\begin{prop} Let $\gamma_{2}>0$ be a real number. Let $k \geq 1$ be an integer such that $1/k \leq \gamma_{2} \leq 1$.
Then, there exists a constant $C_{1}>0$ (depending on $\nu,k,\gamma_{2}$) with
\begin{equation}
||\int_{0}^{\tau^k} (\tau^{k}-s)^{\gamma_2}f(s^{1/k},m) \frac{ds}{s}
||_{(\nu,\beta,\mu,k,\epsilon)} \leq C_{1}|\epsilon|^{k \gamma_{2}} ||f(\tau,m)||_{(\nu,\beta,\mu,k,\epsilon)}
\label{conv_op_continuity_1}
\end{equation}
for all $f(\tau,m) \in F_{(\nu,\beta,\mu,k,\epsilon)}^{d}$.
\end{prop}
\begin{proof} Let $f(\tau,m) \in F_{(\nu,\beta,\mu,k,\epsilon)}^{d}$. For any $\tau \in \bar{D}(0,\rho) \cup S_{d}$, the segment
$[0,\tau^{k}]$ is such that the map $s \in [0,\tau^{k}] \rightarrow f(s^{1/k},m)$ is well defined, provided that $m \in \mathbb{R}$.
By definition, we have that
\begin{multline}
|| \int_{0}^{\tau^k} (\tau^{k}-s)^{\gamma_2}f(s^{1/k},m) \frac{ds}{s}
||_{(\nu,\beta,\mu,k,\epsilon)} \\ = \sup_{\tau \in \bar{D}(0,\rho) \cup S_{d},m \in \mathbb{R}} (1+|m|)^{\mu}
\frac{1 + |\frac{\tau}{\epsilon}|^{2k}}{|\frac{\tau}{\epsilon}|}\exp( \beta|m| - \nu|\frac{\tau}{\epsilon}|^{k} )\\
\times | \int_{0}^{\tau^k} \{ (1+|m|)^{\mu} e^{\beta |m|}
\exp( -\nu |s|/|\epsilon|^{k} )
\frac{ 1 + \frac{|s|^2}{|\epsilon|^{2k}}}{ \frac{|s|^{1/k}}{|\epsilon|}}f(s^{1/k},m) \}\\
\times \mathcal{A}(\tau,s,m,\epsilon) ds|
\end{multline}
where
$$
\mathcal{A}(\tau,s,m,\epsilon) = \frac{1}{(1+|m|)^{\mu}} e^{-\beta |m|}
\frac{\exp( \nu |s|/|\epsilon|^{k})}{1 + \frac{|s|^2}{|\epsilon|^{2k}}} \frac{|s|^{1/k}}{|\epsilon|}
(\tau^{k}-s)^{\gamma_2}\frac{1}{s}
$$
Therefore,
\begin{equation}
|| \int_{0}^{\tau^k} (\tau^{k}-s)^{\gamma_2}f(s^{1/k},m) \frac{ds}{s}
||_{(\nu,\beta,\mu,k,\epsilon)} \leq C_{1}(\epsilon) ||f(\tau,m)||_{(\nu,\beta,\mu,k,\epsilon)} \label{continuity_op}
\end{equation}
where
$$
 C_{1}(\epsilon) = \sup_{\tau \in \bar{D}(0,\rho) \cup S_{d}}
\frac{1 + |\frac{\tau}{\epsilon}|^{2k}}{|\frac{\tau}{\epsilon}|} \exp( -\nu |\frac{\tau}{\epsilon}|^{k} )\\
\times \int_{0}^{|\tau|^k}
\frac{\exp( \nu h/|\epsilon|^{k})}{1+\frac{h^2}{|\epsilon|^{2k}}}
\frac{h^{\frac{1}{k}-1}}{|\epsilon|} (|\tau|^{k}-h)^{\gamma_2} dh
$$
Making the change of variable $h=|\epsilon|^{k}h'$ in the integral inside $C_{1}(\epsilon)$ yields
\begin{multline}
C_{1}(\epsilon) = |\epsilon|^{k \gamma_{2}}
\sup_{\tau \in \bar{D}(0,\rho) \cup S_{d}}
\frac{1 + |\frac{\tau}{\epsilon}|^{2k}}{|\frac{\tau}{\epsilon}|}
\exp( -\nu |\frac{\tau}{\epsilon}|^{k} )\\
\times \int_{0}^{|\frac{\tau}{\epsilon}|^k} \frac{ \exp( \nu h' ) }{ 1 + h'^{2}} (h')^{\frac{1}{k}-1}
(|\frac{\tau}{\epsilon}|^{k} - h')^{\gamma_{2}} dh'
\leq |\epsilon|^{k \gamma_{2}} \sup_{x \geq 0} A(x) \label{C1_chg_var}
\end{multline}
where
$$ A(x) = \frac{1 + x^2}{x^{1/k}} \exp( -\nu x )
\int_{0}^{x} \frac{\exp( \nu h )}{1 + h^2} h^{\frac{1}{k}-1} (x-h)^{\gamma_2} dh
$$
For any $x>0$, we have $A(x) \leq \tilde{A}(x)$, where
$$ \tilde{A}(x) = (1+x^2)x^{\gamma_{2} - \frac{1}{k}}
\exp( - \nu x )\int_{0}^{x} \frac{\exp( \nu h )}{1 + h^2} h^{\frac{1}{k}-1} dh.
$$
Using L'Hospital rule, we know that
\begin{multline*}
\lim_{x \rightarrow +\infty} \tilde{A}(x) =
\lim_{x \rightarrow +\infty}
\frac{ \exp( \nu x )x^{\frac{1}{k}-1}/(1+x^2) }{ \partial_{x}(
\frac{ \exp( \nu x)}{(1+x^2)x^{\gamma_{2}-\frac{1}{k}}}) } \\
= \lim_{x \rightarrow +\infty} \frac{ (1+x^2)x^{2(\gamma_{2}-\frac{1}{k})}x^{\frac{1}{k}-1} }{\nu
(1+x^2)x^{\gamma_{2}-\frac{1}{k}} - (2x^{\gamma_{2}-\frac{1}{k}+1} +
(\gamma_{2}-\frac{1}{k})x^{\gamma_{2}-\frac{1}{k}-1}(1+x^2)) }
\end{multline*}
and this latter limit is finite if $\gamma_{2} \leq 1$ holds. Hence, we deduce that there exists a constant $\tilde{A}>0$
such that
\begin{equation}
\sup_{x \geq 0} \tilde{A}(x) \leq \tilde{A}
\label{tilde_A_1_bounded}
\end{equation} 
Gathering the estimates (\ref{continuity_op}), (\ref{C1_chg_var}), (\ref{tilde_A_1_bounded}), we see that
(\ref{conv_op_continuity_1}) holds.
\end{proof}

\begin{prop} Let $\gamma_{1} \geq 0$ and $\chi_{2}>-1$ be real numbers. Let $\nu_{2} \geq 0$ be an integer.
We consider a holomorphic function $a_{\gamma_{1},k}(\tau)$ on $D(0,\rho) \cup S_{d}$, continuous on
$\bar{D}(0,\rho) \cup S_{d}$, such that
$$ |a_{\gamma_{1},k}(\tau)| \leq \frac{1}{(1+|\tau|^{k})^{\gamma_1}} $$
for all $\tau \in \bar{D}(0,\rho) \cup S_{d}$.\medskip

\noindent {\bf i)} Assume that $\chi_{2} \geq 0$.\medskip

\noindent If $\nu_{2} + \chi_{2} - \gamma_{1} \leq 0$, then there exists a constant $C_{2.1}>0$ (depending on
$\nu,\nu_{2},\chi_{2},\gamma_{1}$) such that
\begin{multline}
|| a_{\gamma_{1},k}(\tau) \int_{0}^{\tau^k} (\tau^{k}-s)^{\chi_2}s^{\nu_{2}}f(s^{1/k},m) ds ||_{(\nu,\beta,\mu,k,\epsilon)} \\
\leq C_{2.1}|\epsilon|^{k(1+\nu_{2}+\chi_{2}-\gamma_{1})}
||f(\tau,m)||_{(\nu,\beta,\mu,k,\epsilon)}
\label{conv_op_prod_s_continuity_1}
\end{multline}
for all $f(\tau,m) \in F_{(\nu,\beta,\mu,k,\epsilon)}^{d}$.\medskip

\noindent {\bf ii)} Assume that $\chi_{2} = \frac{\chi}{k} - 1$ for some real number $\chi \geq 1$.\medskip

\noindent If $\nu_{2} + \frac{1}{k} - \gamma_{1} \leq 0$, then there exists a constant $C_{2.2}>0$
(depending $\chi,k,\nu,\gamma_{1},\nu_{2}$) on such that
\begin{multline}
|| a_{\gamma_{1},k}(\tau) \int_{0}^{\tau^k} (\tau^{k}-s)^{\chi_2}s^{\nu_{2}}f(s^{1/k},m) ds ||_{(\nu,\beta,\mu,k,\epsilon)} \\
\leq C_{2.2}|\epsilon|^{k(1+\nu_{2}+\chi_{2}-\gamma_{1})}
||f(\tau,m)||_{(\nu,\beta,\mu,k,\epsilon)}
\label{conv_op_prod_s_continuity_2}
\end{multline}
for all $f(\tau,m) \in F_{(\nu,\beta,\mu,k,\epsilon)}^{d}$.
\end{prop}
\begin{proof} In the first part of the proof, let us assume that {\bf i)} holds. Let
$f(\tau,m) \in F_{(\nu,\beta,\mu,k,\epsilon)}^{d}$. By definition, we have
\begin{multline}
|| a_{\gamma_{1},k}(\tau) \int_{0}^{\tau^k} (\tau^{k}-s)^{\chi_2}s^{\nu_2}f(s^{1/k},m) ds
||_{(\nu,\beta,\mu,k,\epsilon)} \\ = \sup_{\tau \in \bar{D}(0,\rho) \cup S_{d},m \in \mathbb{R}} (1+|m|)^{\mu}
\frac{1 + |\frac{\tau}{\epsilon}|^{2k}}{|\frac{\tau}{\epsilon}|}\exp( \beta|m| - \nu|\frac{\tau}{\epsilon}|^{k} )\\
\times | a_{\gamma_{1},k}(\tau) \int_{0}^{\tau^k} \{ (1+|m|)^{\mu} e^{\beta |m|}
\exp( -\nu |s|/|\epsilon|^{k} )
\frac{ 1 + \frac{|s|^2}{|\epsilon|^{2k}}}{ \frac{|s|^{1/k}}{|\epsilon|}}f(s^{1/k},m) \}\\
\times \mathcal{B}(\tau,s,m,\epsilon) ds|
\end{multline}
where
$$
\mathcal{B}(\tau,s,m,\epsilon) = \frac{1}{(1+|m|)^{\mu}} e^{-\beta |m|}
\frac{\exp( \nu |s|/|\epsilon|^{k})}{1 + \frac{|s|^2}{|\epsilon|^{2k}}} \frac{|s|^{1/k}}{|\epsilon|}
(\tau^{k}-s)^{\chi_2}s^{\nu_2}.
$$
Therefore,
\begin{equation}
|| a_{\gamma_{1},k}(\tau) \int_{0}^{\tau^k} (\tau^{k}-s)^{\chi_2}s^{\nu_2}f(s^{1/k},m) ds
||_{(\nu,\beta,\mu,k,\epsilon)} \leq C_{2}(\epsilon) ||f(\tau,m)||_{(\nu,\beta,\mu,k,\epsilon)} \label{continuity_op_prod_s}
\end{equation}
where
\begin{multline*}
 C_{2}(\epsilon) = \sup_{\tau \in \bar{D}(0,\rho) \cup S_{d}}
\frac{1 + |\frac{\tau}{\epsilon}|^{2k}}{|\frac{\tau}{\epsilon}|} \exp( -\nu |\frac{\tau}{\epsilon}|^{k} )\\
\times \frac{1}{(1+|\tau|^{k})^{\gamma_1}} \int_{0}^{|\tau|^k}
\frac{\exp( \nu h/|\epsilon|^{k})}{1+\frac{h^2}{|\epsilon|^{2k}}}
\frac{h^{\frac{1}{k}}}{|\epsilon|} (|\tau|^{k}-h)^{\chi_2} h^{\nu_2} dh
\end{multline*}
Making the change of variable $h=|\epsilon|^{k}h'$ in the integral inside $C_{2}(\epsilon)$ yields
\begin{multline}
C_{2}(\epsilon) = |\epsilon|^{k(1+\nu_{2}+\chi_{2})}
\sup_{\tau \in \bar{D}(0,\rho) \cup S_{d}}
\frac{1 + |\frac{\tau}{\epsilon}|^{2k}}{|\frac{\tau}{\epsilon}|}
\exp( -\nu |\frac{\tau}{\epsilon}|^{k} )\\
\times \frac{1}{(1+|\epsilon|^{k}|\frac{\tau}{\epsilon}|^{k})^{\gamma_1}}
\int_{0}^{|\frac{\tau}{\epsilon}|^k} \frac{ \exp( \nu h' ) }{ 1 + h'^{2}} (h')^{\frac{1}{k}}
(|\frac{\tau}{\epsilon}|^{k} - h')^{\chi_{2}} h'^{\nu_2} dh'\\
\leq |\epsilon|^{k (1 + \nu_{2} + \chi_{2})} \sup_{x \geq 0} B(x,\epsilon) \label{C2_chg_var}
\end{multline}
where
$$ B(x,\epsilon) = \frac{1 + x^2}{x^{1/k}} \exp( -\nu x ) \frac{1}{(1+|\epsilon|^{k}x)^{\gamma_1}}
\int_{0}^{x} \frac{\exp( \nu h )}{1 + h^2} h^{\frac{1}{k} + \nu_{2}} (x-h)^{\chi_2} dh.
$$
For any $x>0$, we get that $B(x,\epsilon) \leq \tilde{B}(x,\epsilon)$, where
$$
\tilde{B}(x,\epsilon) = \frac{(1+x^2)x^{\chi_{2}}}{(1+|\epsilon|^{k}x)^{\gamma_1}}
\exp(-\nu x) \int_{0}^{x} \frac{\exp( \nu h )}{1 + h^2} h^{\nu_2} dh
$$
Let $x_{0}>0$. From the inequality $1+|\epsilon|^{k}x \geq 1$, for all $x \in [0,x_{0}]$ and
$\epsilon \in \mathcal{E}$, there exists a constant $\tilde{B}>0$ such that
\begin{equation}
\sup_{x \in [0,x_{0}],\epsilon \in \mathcal{E}} \tilde{B}(x,\epsilon) \leq \tilde{B}. \label{tilde_B_bounded_small_x}
\end{equation}
On the other hand, since $1 + |\epsilon|^{k}x \geq |\epsilon|^{k}x$ holds for all $x \geq 0$ and $\epsilon \in \mathcal{E}$, we get that
$\tilde{B}(x,\epsilon) \leq \tilde{B}_{2}(x)/|\epsilon|^{k \gamma_{1}}$ where
\begin{equation}
\tilde{B}_{2}(x) = (1+x^2)x^{\chi_{2}-\gamma_{1}} \exp( -\nu x ) \int_{0}^{x}
\frac{\exp(\nu h)}{1 + h^2} h^{\nu_{2}} dh
\end{equation}
for all $x \geq x_{0}$. By L'Hospital rule we get that
$$
\lim_{x \rightarrow +\infty} \tilde{B}_{2}(x)
=  \lim_{x \rightarrow +\infty} \frac{ (1+x^2)x^{2(\chi_{2} - \gamma_{1})}x^{\nu_{2}} }{\nu
(1+x^2)x^{\chi_{2} - \gamma_{1}} - (2x^{\chi_{2} - \gamma_{1} + 1} +
(\chi_{2} - \gamma_{1})x^{\chi_{2} - \gamma_{1} - 1}(1+x^2))}
$$
which is finite if we assume that $1 \geq (1 + \nu_{2} + \chi_{2} - \gamma_{1})$. We deduce that there exists a constant
$\tilde{B}_{2}>0$ such that
\begin{equation}
\sup_{x \geq x_{0}} \tilde{B}(x,\epsilon) \leq \frac{1}{|\epsilon|^{k\gamma_{1}}} \sup_{x \geq x_{0}} \tilde{B}_{2}(x) \leq
\frac{\tilde{B}_{2}}{|\epsilon|^{k\gamma_{1}}} \label{tilde_B_2_bounded}
\end{equation}
Bearing in mind the estimates (\ref{continuity_op_prod_s}), (\ref{C2_chg_var}), (\ref{tilde_B_bounded_small_x}) and
(\ref{tilde_B_2_bounded}), we obtain (\ref{conv_op_prod_s_continuity_1}).\medskip

In the second part of the proof, assume now that the condition {\bf ii)} holds. Let
$f(\tau,m) \in F_{(\nu,\beta,\mu,k,\epsilon)}^{d}$. By definition, we have
\begin{multline}
|| a_{\gamma_{1},k}(\tau) \int_{0}^{\tau^k} (\tau^{k}-s)^{\frac{\chi}{k}-1}s^{\nu_2}f(s^{1/k},m) ds
||_{(\nu,\beta,\mu,k,\epsilon)} \\ = \sup_{\tau \in \bar{D}(0,\rho) \cup S_{d},m \in \mathbb{R}} (1+|m|)^{\mu}
\frac{1 + |\frac{\tau}{\epsilon}|^{2k}}{|\frac{\tau}{\epsilon}|}\exp( \beta|m| - \nu|\frac{\tau}{\epsilon}|^{k} )\\
\times | a_{\gamma_{1},k}(\tau) \int_{0}^{\tau^k} \{ (1+|m|)^{\mu} e^{\beta |m|}
\exp( -\nu |s|/|\epsilon|^{k} )
\frac{ 1 + \frac{|s|^2}{|\epsilon|^{2k}}}{ \frac{|s|^{1/k}}{|\epsilon|}}f(s^{1/k},m) \}\\
\times \{ \exp( -\nu \frac{|\tau^{k}-s|}{|\epsilon|^k} )
\frac{ 1 + \frac{|\tau^{k}-s|^2}{|\epsilon|^{2k}} }{\frac{|\tau^{k}-s|^{1/k}}{|\epsilon|}} (\tau^{k}-s)^{\frac{\chi}{k}} \}
\times \mathfrak{B}(\tau,s,m,\epsilon) ds|
\end{multline}
where
\begin{multline*}
\mathfrak{B}(\tau,s,m,\epsilon) = \frac{e^{-\beta |m|}}{(1+|m|)^{\mu}} \exp(\nu \frac{|s|}{|\epsilon|^k})
\exp( \nu \frac{ |\tau^{k}-s| }{|\epsilon|^k} ) \frac{|s|^{1/k}}{|\epsilon|} \frac{ |\tau^{k}-s|^{1/k}}{|\epsilon|}\\
\times \frac{1}{1 + \frac{|s|^2}{|\epsilon|^{2k}}}\frac{1}{1 + \frac{|\tau^{k}-s|^2}{|\epsilon|^{2k}}}(\tau^{k}-s)^{-1}s^{\nu_2}.
\end{multline*}
Hence,
\begin{equation}
|| a_{\gamma_{1},k}(\tau) \int_{0}^{\tau^k} (\tau^{k}-s)^{\frac{\chi}{k}-1}s^{\nu_2}f(s^{1/k},m) ds
||_{(\nu,\beta,\mu,k,\epsilon)} \leq
C_{2.2}(\epsilon)C_{2.3}(\epsilon)||f(\tau,m)||_{(\nu,\beta,\mu,k,\epsilon)} \label{norm_a_int_polyn_f<C22C23_norm_f}
\end{equation}
where
\begin{multline}
 C_{2.2}(\epsilon) = \sup_{x \geq 0} \exp( -\nu \frac{x}{|\epsilon|^k})
\frac{1 + \frac{x^2}{|\epsilon|^{2k}}}{\frac{x^{1/k}}{|\epsilon|}} x^{\frac{\chi}{k}},\\
C_{2.3}(\epsilon) = \sup_{\tau \in \bar{D}(0,\rho) \cup S_{d}}
\frac{1 + |\frac{\tau}{\epsilon}|^{2k}}{|\frac{\tau}{\epsilon}|}\frac{1}{(1 + |\tau|^k)^{\gamma_1}}\\
\times 
\int_{0}^{|\tau|^k} \frac{h^{1/k}}{|\epsilon|}\frac{(|\tau|^{k}-h)^{1/k}}{|\epsilon|}
\frac{1}{1 + \frac{h^2}{|\epsilon|^{2k}}}\frac{1}{1 + \frac{(|\tau|^{k}-h)^2}{|\epsilon|^{2k}}}(|\tau|^{k}-h)^{-1}h^{\nu_2} dh.
\end{multline}
By using the classical estimates
\begin{equation}
\sup_{x \geq 0} x^{m_1}\exp(-m_{2}x) = (\frac{m_1}{m_2})^{m_1}e^{-m_1} \label{x_m_exp_x<}
\end{equation}
for any real numbers $m_{1} \geq 0$ and $m_{2}>0$, we get that
\begin{equation}
C_{2.2}(\epsilon) \leq |\epsilon|^{\chi}\left( (\frac{\chi-1}{k \nu})^{\frac{\chi-1}{k}}e^{-(\frac{\chi-1}{k})}
+ (\frac{2 + \frac{\chi-1}{k}}{\nu})^{2+\frac{\chi-1}{k}} e^{-(2 + \frac{\chi-1}{k})} \right). \label{C22_epsilon<}
\end{equation}
Making the change of variable $h=|\epsilon|^{k}h'$ in the integral involved in the definition of $C_{2.3}(\epsilon)$ yields
\begin{multline}
C_{2.3}(\epsilon) = \sup_{\tau \in \bar{D}(0,\rho) \cup S_{d}}
\frac{1 + |\frac{\tau}{\epsilon}|^{2k}}{|\frac{\tau}{\epsilon}|}\frac{1}{(1 + |\epsilon|^{k}|\frac{\tau}{\epsilon}|^k)^{\gamma_1}}\\
\times \int_{0}^{|\frac{\tau}{\epsilon}|^{k}} (h')^{1/k} (|\frac{\tau}{\epsilon}|^{k}-h')^{1/k} \frac{1}{1 + (h')^2}
\frac{1}{1 + (|\frac{\tau}{\epsilon}|^{k} - h')^{2}} (|\frac{\tau}{\epsilon}|^{k}-h')^{-1} |\epsilon|^{k \nu_{2}}
(h')^{\nu_2} dh' \\
\leq |\epsilon|^{k \nu_{2}} \sup_{x \geq 0} B_{2.3}(x,\epsilon) \label{C23_epsilon<}
\end{multline}
where
$$ B_{2.3}(x,\epsilon) = \frac{1 + x^2}{x^{1/k}} \frac{1}{(1 + |\epsilon|^{k}x)^{\gamma_1}}
\int_{0}^{x} \frac{1}{(1+h^2)(1+(x-h)^2)} \frac{1}{(x-h)^{1-\frac{1}{k}}} h^{\frac{1}{k}+\nu_{2}} dh. $$
For any $x>0$, we have that $B_{2.3}(x,\epsilon) \leq \tilde{B}_{2.3}(x,\epsilon)$, where
$$ \tilde{B}_{2.3}(x,\epsilon) = \frac{1+x^2}{(1 + |\epsilon|^{k}x)^{\gamma_1}}
\int_{0}^{x} \frac{1}{(1+h^2)(1 + (x-h)^2)} \frac{1}{(x-h)^{1 - \frac{1}{k}}} h^{\nu_2} dh. $$
Let $x_{0}>0$. From the inequality $1 + |\epsilon|^{k}x \geq 1$, for all $x \in [0,x_{0}]$, $\epsilon \in \mathcal{E}$, there
exists a constant $\tilde{B}_{2.3}>0$ such that
\begin{equation}
\sup_{x \in [0,x_{0}], \epsilon \in \mathcal{E}} \tilde{B}_{2.3}(x,\epsilon) \leq \tilde{B}_{2.3}. \label{tilde_B23_x_near_zero<}
\end{equation}
On the other hand, since $1 + |\epsilon|^{k}x \geq |\epsilon|^{k}x$ holds for all $x \geq 0$ and $\epsilon \in \mathcal{E}$, we get that
\begin{equation}
\tilde{B}_{2.3}(x,\epsilon) \leq \frac{ \tilde{B}_{2.4}(x) }{|\epsilon|^{k\gamma_1}}
\end{equation}
where
$$ \tilde{B}_{2.4}(x) = (1+x^2)x^{-\gamma_{1}} \int_{0}^{x} \frac{1}{(1+h^2)(1+(x-h)^2)} \frac{1}{(x-h)^{1-\frac{1}{k}}}
h^{\nu_2} dh $$
for all $x \geq x_{0}$. Now, we make the change of variable $h=xu$ in the integral inside $\tilde{B}_{2.4}(x)$. We can write
$$ \tilde{B}_{2.4}(x) = (1+x^2)x^{\nu_{2}+\frac{1}{k}-\gamma_{1}}F_{k}(x) $$
where
$$ F_{k}(x) = \int_{0}^{1} \frac{u^{\nu_2}}{(1+x^{2}u^{2})(1 + x^{2}(1-u)^2)(1-u)^{1-\frac{1}{k}}} du. $$
Using a partial fraction decomposition, we can split $F_{k} = F_{1,k}(x) + F_{2,k}(x)$, where
\begin{multline*}
F_{1,k}(x) = \frac{1}{4+x^2} \int_{0}^{1} \frac{(2u+1)u^{\nu_2}}{(1+x^{2}u^{2})(1-u)^{1 - \frac{1}{k}}} du,\\
F_{2,k}(x) = \frac{1}{4 + x^2}\int_{0}^{1} \frac{(3-2u)u^{\nu_2}}{(1 + x^{2}(1-u)^{2})(1-u)^{1 - \frac{1}{k}}} du.
\end{multline*}
In particular, we observe that there exist two constants $\mathfrak{F}_{1,k},\mathfrak{F}_{2,k}>0$ such that
\begin{equation}
F_{1,k}(x) \leq \frac{ \mathfrak{F}_{1,k} }{4 + x^{2}} \ \ , \ \ F_{2,k}(x) \leq \frac{ \mathfrak{F}_{2,k} }{4 + x^{2}}
\end{equation}
for all $x \geq x_{0}$. Hence, if one assumes that $\nu_{2} + \frac{1}{k} - \gamma_{1} \leq 0$, then we get a constant
$\tilde{B}_{2.4.1}>0$ such that
\begin{equation}
\sup_{x \geq x_{0}} \tilde{B}_{2.3}(x,\epsilon) \leq \frac{1}{|\epsilon|^{k \gamma_{1}}}
\sup_{x \geq x_{0}} \tilde{B}_{2.4}(x) \leq \frac{\tilde{B}_{2.4.1}}{|\epsilon|^{k \gamma_{1}}} \label{tilde_B23_x_large<}
\end{equation}
Finally, gathering all the estimates (\ref{norm_a_int_polyn_f<C22C23_norm_f}), (\ref{C22_epsilon<}),
(\ref{C23_epsilon<}), (\ref{tilde_B23_x_near_zero<}), (\ref{tilde_B23_x_large<}), we get (\ref{conv_op_prod_s_continuity_2}).
\end{proof}

\begin{prop} Let $k \geq 1$ be an integer. Let $Q_{1}(X),Q_{2}(X),R(X) \in \mathbb{C}[X]$ such that
\begin{equation}
\mathrm{deg}(R) \geq \mathrm{deg}(Q_{1}) \ \ , \ \ \mathrm{deg}(R) \geq \mathrm{deg}(Q_{2}) \ \ , \ \ R(im) \neq 0
\label{R>Q1_R>Q2_R_nonzero}
\end{equation}
for all $m \in \mathbb{R}$. Assume that $\mu > \max( \mathrm{deg}(Q_{1})+1, \mathrm{deg}(Q_{2})+1 )$. Let
$m \mapsto b(m)$ be a continuous function on $\mathbb{R}$ such that
$$
|b(m)| \leq \frac{1}{|R(im)|}
$$
for all $m \in \mathbb{R}$. Then, there exists a constant $C_{3}>0$ (depending on $Q_{1},Q_{2},R,\mu,k,\nu$) such that
\begin{multline}
|| b(m) \int_{0}^{\tau^{k}} (\tau^{k}-s)^{\frac{1}{k}} ( \int_{0}^{s} \int_{-\infty}^{+\infty} Q_{1}(i(m-m_{1}))
f((s-x)^{1/k},m-m_{1}) \\
\times Q_{2}(im_{1}) g(x^{1/k},m_{1}) \frac{1}{(s-x)x} dx dm_{1} ) ds ||_{(\nu,\beta,\mu,k,\epsilon)} \\
\leq C_{3}|\epsilon| ||f(\tau,m)||_{(\nu,\beta,\mu,k,\epsilon)} ||g(\tau,m)||_{(\nu,\beta,\mu,k,\epsilon)}
\label{norm_conv_f_g<norm_f_times_norm_g}
\end{multline}
for all $f(\tau,m), g(\tau,m) \in F_{(\nu,\beta,\mu,k,\epsilon)}^{d}$.
\end{prop}
\begin{proof} Let $f(\tau,m), g(\tau,m) \in F_{(\nu,\beta,\mu,k,\epsilon)}^{d}$. For any $\tau \in \bar{D}(0,\rho) \cup S_{d}$, the
segment $[0,\tau^{k}]$ is such that for any $s \in [0,\tau^{k}]$, any $x \in [0,s]$, the expressions $f((s-x)^{1/k},m-m_{1})$
and $g(x^{1/k},m_{1})$ are well defined, provided that $m,m_{1} \in \mathbb{R}$. By definition, we can write
\begin{multline*}
|| b(m) \int_{0}^{\tau^{k}} (\tau^{k}-s)^{\frac{1}{k}} ( \int_{0}^{s} \int_{-\infty}^{+\infty} Q_{1}(i(m-m_{1}))
f((s-x)^{1/k},m-m_{1}) \\
\times Q_{2}(im_{1}) g(x^{1/k},m_{1}) \frac{1}{(s-x)x} dx dm_{1} ) ds ||_{(\nu,\beta,\mu,k,\epsilon)}\\
= \sup_{\tau \in \bar{D}(0,\rho) \cup S_{d},m \in \mathbb{R}} (1+|m|)^{\mu}
\frac{1 + |\frac{\tau}{\epsilon}|^{2k}}{|\frac{\tau}{\epsilon}|}\exp( \beta|m| - \nu|\frac{\tau}{\epsilon}|^{k} )\\
\times | \int_{0}^{\tau^{k}} (\tau^{k}-s)^{1/k} ( \int_{0}^{s}
\int_{-\infty}^{+\infty} \{ (1+|m-m_{1}|)^{\mu} e^{\beta|m-m_{1}|}
\frac{1 + \frac{|s-x|^{2}}{|\epsilon|^{2k}}}{\frac{|s-x|^{1/k}}{|\epsilon|}}
\exp( -\nu |s-x|/|\epsilon|^{\kappa} )\\
\times f( (s-x)^{1/k},m-m_{1}) \} \times \{ (1 + |m_{1}|)^{\mu}e^{\beta|m_{1}|}
\frac{1 + \frac{|x|^2}{|\epsilon|^{2k}} }{ \frac{|x|^{1/k}}{|\epsilon|} } \exp( -\nu |x|/|\epsilon|^{k} )
g(x^{1/k},m_{1}) \}\\
\times \mathcal{C}(s,x,m,m_{1},\epsilon) dxdm_{1} ) ds |
\end{multline*}
where
\begin{multline*}
\mathcal{C}(s,x,m,m_{1},\epsilon) = \frac{ \exp(-\beta|m_1|) \exp(-\beta|m-m_{1}|) }{(1 + |m-m_{1}|)^{\mu}(1+|m_{1}|)^{\mu}}
b(m)Q_{1}(i(m-m_{1}))Q_{2}(im_{1})\\
\times \frac{ \frac{|s-x|^{1/k}|x|^{1/k}}{|\epsilon|^2} }{(1 + \frac{|s-x|^2}{|\epsilon|^{2k}})(1 + \frac{|x|^2}{|\epsilon|^{2k}})}
\times 
\exp( \nu |s-x|/|\epsilon|^{k} ) \exp( \nu |x|/|\epsilon|^{k} ) \frac{1}{(s-x)x}
\end{multline*}
Now, we know that there exist $\mathfrak{Q}_{1},\mathfrak{Q}_{2},\mathfrak{R}>0$ with
\begin{multline}
|Q_{1}(i(m-m_{1}))| \leq \mathfrak{Q}_{1}(1 + |m-m_{1}|)^{\mathrm{deg}(Q_{1})} \ \ , \ \
|Q_{2}(im_{1})| \leq \mathfrak{Q}_{2}(1 + |m_{1}|)^{\mathrm{deg}(Q_{2})},\\
|R(im)| \geq \mathfrak{R}(1+|m|)^{\mathrm{deg}(R)} \label{Q1_Q2_R_deg_order}
\end{multline}
for all $m,m_{1} \in \mathbb{R}$. Therefore,
\begin{multline}
|| b(m) \int_{0}^{\tau^{k}} (\tau^{k}-s)^{\frac{1}{k}} ( \int_{0}^{s} \int_{-\infty}^{+\infty} Q_{1}(i(m-m_{1}))
f((s-x)^{1/k},m-m_{1})\\
\times Q_{2}(im_{1})g(x^{1/k},m_{1}) \frac{1}{(s-x)x} dx dm_{1} ) ds ||_{(\nu,\beta,\mu,k,\epsilon)}\\
\leq C_{3}(\epsilon)
||f(\tau,m)||_{(\nu,\beta,\mu,k,\epsilon)} ||g(\tau,m)||_{(\nu,\beta,\mu,k,\epsilon)}
\end{multline}
where
\begin{multline}
C_{3}(\epsilon) = \sup_{\tau \in \bar{D}(0,\rho) \cup S_{d},m \in \mathbb{R}} (1+|m|)^{\mu}
\frac{1 + |\frac{\tau}{\epsilon}|^{2k}}{|\frac{\tau}{\epsilon}|}\exp( \beta|m| - \nu|\frac{\tau}{\epsilon}|^{k} )
\frac{1}{\mathfrak{R}(1+|m|)^{\mathrm{deg}(R)}}\\
\times \int_{0}^{|\tau|^{k}} (|\tau|^{k}-h)^{1/k} (\int_{0}^{h}
\int_{-\infty}^{+\infty} \frac{ \exp(-\beta|m_1|) \exp(-\beta|m-m_{1}|) }{(1 + |m-m_{1}|)^{\mu}(1+|m_{1}|)^{\mu}}\\
\times
\mathfrak{Q}_{1}\mathfrak{Q}_{2}(1+|m-m_{1}|)^{\mathrm{deg}(Q_{1})}(1+|m_{1}|)^{\mathrm{deg}(Q_{2})}
\frac{ \frac{(h-x)^{1/k}x^{1/k}}{|\epsilon|^2} }{(1 + \frac{(h-x)^2}{|\epsilon|^{2k}})(1 + \frac{x^2}{|\epsilon|^{2k}})}\\
\times 
\exp( \nu (h-x)/|\epsilon|^{k} ) \exp( \nu x/|\epsilon|^{k} ) \frac{1}{(h-x)x} dx dm_{1}) dh
\label{defin_C_3}
\end{multline}

Using the triangular inequality $|m| \leq |m_{1}| + |m-m_{1}|$, for all $m,m_{1} \in \mathbb{R}$, we get that
$C_{3}(\epsilon) \leq C_{3.1}C_{3.2}(\epsilon)$ where
\begin{equation}
C_{3.1} = \frac{\mathfrak{Q}_{1}\mathfrak{Q}_{2}}{\mathfrak{R}}
\sup_{m \in \mathbb{R}} (1+|m|)^{\mu - \mathrm{deg}(R)} \int_{-\infty}^{+\infty}
\frac{1}{(1+|m-m_{1}|)^{\mu - \mathrm{deg}(Q_{1})}(1+|m_{1}|)^{\mu - \mathrm{deg}(Q_{2})}} dm_{1} \label{defin_C_3.1} 
\end{equation}
which is finite whenever $\mu>\max( \mathrm{deg}(Q_{1})+1, \mathrm{deg}(Q_{2})+1)$ under the assumption
(\ref{R>Q1_R>Q2_R_nonzero}) using the same estimates as in Lemma 4 of \cite{ma2} (see also the Lemma 2.2 from \cite{cota2}),
and where
\begin{multline}
C_{3.2}(\epsilon) = \sup_{\tau \in \bar{D}(0,\rho) \cup S_{d}}
\frac{1 + |\frac{\tau}{\epsilon}|^{2k}}{|\frac{\tau}{\epsilon}|} \exp( -\nu |\frac{\tau}{\epsilon}|^{k} )\\
\times 
\int_{0}^{|\tau|^k} (|\tau|^{k}-h)^{1/k} \exp( \nu h/|\epsilon|^{\kappa} ) \int_{0}^{h}
\frac{ \frac{(h-x)^{1/k}x^{1/k}}{|\epsilon|^2} }{(1 + \frac{(h-x)^2}{|\epsilon|^{2k}})(1 + \frac{x^2}{|\epsilon|^{2k}})}
\frac{1}{(h-x)x} dx dh. \label{defin_C3.2}
\end{multline}
Making the changes of variables $h=|\epsilon|^{k}h'$ and $x=|\epsilon|^{k}x'$, we get that
\begin{multline}
\int_{0}^{|\tau|^k} (|\tau|^{k}-h)^{1/k} \exp( \nu h/|\epsilon|^{k} ) \int_{0}^{h}
\frac{ \frac{(h-x)^{1/k}x^{1/k}}{|\epsilon|^2} }{(1 + \frac{(h-x)^2}{|\epsilon|^{2k}})(1 + \frac{x^2}{|\epsilon|^{2k}})}
\frac{1}{(h-x)x} dx dh\\
= |\epsilon| \int_{0}^{|\frac{\tau}{\epsilon}|^{k}} (|\frac{\tau}{\epsilon}|^{k}-h')^{1/k}
\exp( \nu h' ) \int_{0}^{h'} \frac{1}{(1 + (h'-x')^2)(1+x'^{2})} \frac{1}{(h'-x')^{1-\frac{1}{k}}x'^{1-\frac{1}{k}}} dx'dh'
\label{change_var_int}
\end{multline}
From (\ref{defin_C3.2}) and (\ref{change_var_int}), we get that $C_{3.2}(\epsilon) \leq |\epsilon|C_{3.3}$, where
\begin{multline}
C_{3.3} = \sup_{x \geq 0} \frac{1+x^2}{x^{1/k}} \exp( -\nu x )
\int_{0}^{x} (x-h')^{1/k} \exp( \nu h' )\\
\times (\int_{0}^{h'} \frac{1}{(1 + (h'-x')^2)(1+x'^{2})} \frac{1}{(h'-x')^{1-\frac{1}{k}}x'^{1-\frac{1}{k}}} dx') dh'
\label{defin_C3.3}
\end{multline}
Again by the change of variable $x'=h'u$, for $u \in [0,1]$, we can write
\begin{multline}
\int_{0}^{h'} \frac{1}{(1 + (h'-x')^2)(1+x'^{2})} \frac{1}{(h'-x')^{1-\frac{1}{k}}x'^{1-\frac{1}{k}}} dx'\\
= \frac{1}{h'^{1 - \frac{2}{k}}} \int_{0}^{1}
\frac{1}{(1 + (h')^{2}(1-u)^{2})(1 + h'^{2}u^2)(1 - u)^{1 - \frac{1}{k}}u^{1 - \frac{1}{k}}} du = J_{k}(h')
\end{multline}
Using a partial fraction decomposition, we can split $J_{k}(h') = J_{1,k}(h') + J_{2,k}(h')$, where
\begin{multline}
J_{1,k}(h') = \frac{1}{h'^{1 - \frac{2}{k}}(h'^{2}+4)} \int_{0}^{1}
\frac{3-2u}{(1 + h'^{2}(1-u)^{2})(1-u)^{1-\frac{1}{k}}u^{1-\frac{1}{k}}} du \\
J_{2,k}(h') = \frac{1}{h'^{1 - \frac{2}{k}}(h'^{2}+4)} \int_{0}^{1}
\frac{2u+1}{(1 + h'^{2}u^{2})(1-u)^{1-\frac{1}{k}}u^{1-\frac{1}{k}}} du
\end{multline}
From now on, we assume that $k \geq 2$. By construction of $J_{1,k}(h')$ and $J_{2,k}(h')$, we see that there exists a constant
$j_{k}>0$ such that
\begin{equation}
J_{k}(h') \leq \frac{j_{k}}{h'^{1 - \frac{2}{k}}(h'^{2}+4)} \label{J_k_maj_frac}
\end{equation}
for all $h' \geq 0$. From (\ref{defin_C3.3}) and (\ref{J_k_maj_frac}), we deduce that $C_{3.3} \leq \sup_{x \geq 0} \tilde{C}_{3.3}(x)$,
where
\begin{equation}
\tilde{C}_{3.3}(x) = (1+x^2) \exp( -\nu x )\int_{0}^{x} \frac{j_{k}
\exp( \nu h' )}{h'^{1 - \frac{2}{k}}(h'^{2}+4)} dh'. \label{defin_tilde_C3.3}
\end{equation}
From L'Hospital rule, we know that
$$
\lim_{x \rightarrow +\infty} \tilde{C}_{3.3}(x) = \lim_{x \rightarrow +\infty} \frac{j_{k}}{x^{1-\frac{2}{k}}}
\frac{ \frac{(1+x^2)^2}{x^2+4}}{\nu (1+x^2) - 2x }
$$
is finite when $k \geq 2$. Therefore, we get a constant $\tilde{C}_{3.3}>0$ such that
\begin{equation}
\sup_{x \geq 0} \tilde{C}_{3.3}(x) \leq \tilde{C}_{3.3}. \label{maj_C_3.3}
\end{equation}
Taking into account the estimates for (\ref{defin_C_3}), (\ref{defin_C_3.1}), (\ref{defin_C3.2}), (\ref{defin_C3.3}),
(\ref{defin_tilde_C3.3}) and (\ref{maj_C_3.3}), we obtain the result (\ref{norm_conv_f_g<norm_f_times_norm_g}) when
$k \geq 2$.\medskip

\noindent In the remaining case $k=1$, from Corollary 4.9 of \cite{cota} one can check the existence of a constant $j_{1}>0$ such that
\begin{equation}
J_{1}(h') \leq \frac{j_{1}}{h'^{2}+1} \label{J_1_maj_frac}
\end{equation}
for all $h' \geq 0$. From (\ref{defin_C3.3}) and (\ref{J_1_maj_frac}), we deduce that $C_{3.3} \leq \sup_{x \geq 0} \tilde{C}_{3.3.1}(x)$,
where
\begin{equation}
\tilde{C}_{3.3.1}(x) = (1+x^2) \exp( -\nu x )\int_{0}^{x} \frac{j_{1}
\exp( \nu h' )}{h'^{2}+1} dh'. \label{defin_tilde_C3.3.1}
\end{equation}
From L'Hospital rule, we know that
$$
\lim_{x \rightarrow +\infty}\tilde{C}_{3.3.1}(x) =
\lim_{x \rightarrow +\infty} \frac{ j_{1}(1+x^2) }{ \nu  (1+x^2) -2x }
$$
is finite. Therefore, we get a constant $\tilde{C}_{3.3.1}>0$ such that
\begin{equation}
\sup_{x \geq 0} \tilde{C}_{3.3.1}(x) \leq \tilde{C}_{3.3.1}. \label{maj_C_3.3.1}
\end{equation}
Taking into account the estimates for (\ref{defin_C_3}), (\ref{defin_C_3.1}), (\ref{defin_C3.2}), (\ref{defin_C3.3}),
(\ref{defin_tilde_C3.3.1}) and (\ref{maj_C_3.3.1}), we obtain the result (\ref{norm_conv_f_g<norm_f_times_norm_g}) for $k=1$.
\end{proof}

\begin{defin} Let $\beta, \mu \in \mathbb{R}$. We denote by
$E_{(\beta,\mu)}$ the vector space of continuous functions $h : \mathbb{R} \rightarrow \mathbb{C}$ such that
$$ ||h(m)||_{(\beta,\mu)} = \sup_{m \in \mathbb{R}} (1+|m|)^{\mu} \exp( \beta |m|) |h(m)| $$
is finite. The space $E_{(\beta,\mu)}$ equipped with the norm $||.||_{(\beta,\mu)}$ is a Banach space.
\end{defin}

\begin{prop} Let $k \geq 1$ be an integer. Let $Q(X),R(X) \in \mathbb{C}[X]$ be polynomials such that
\begin{equation}
\mathrm{deg}(R) \geq \mathrm{deg}(Q) \ \ , \ \ R(im) \neq 0 \label{cond_R_Q}
\end{equation}
for all $m \in \mathbb{R}$. Assume that $\mu > \mathrm{deg}(Q) + 1$. Let $m \mapsto b(m)$ be a continuous function such that
$$
|b(m)| \leq \frac{1}{|R(im)|}
$$
for all $m \in \mathbb{R}$. Then, there exists a constant $C_{4}>0$ (depending on $Q,R,\mu,k,\nu$) such that
\begin{multline}
|| b(m) \int_{0}^{\tau^{k}} (\tau^{k}-s)^{\frac{1}{k}}\int_{-\infty}^{+\infty} f(m-m_{1})Q(im_{1})g(s^{1/k},m_{1})dm_{1}
\frac{ds}{s}||_{(\nu,\beta,\mu,k,\epsilon)}\\
\leq C_{4}|\epsilon| ||f(m)||_{(\beta,\mu)} ||g(\tau,m)||_{(\nu,\beta,\mu,k,\epsilon)}
\label{norm_conv_f_g<norm_f_beta_mu_times_norm_g}
\end{multline}
for all $f(m) \in E_{(\beta,\mu)}$, all $g(\tau,m) \in F_{(\nu,\beta,\mu,k,\epsilon)}^{d}$.
\end{prop}
\begin{proof} The proof follows the same lines of arguments as those of Propositions 1 and 3. Let
$f(m) \in E_{(\beta,\mu)}$, $g(\tau,m) \in F_{(\nu,\beta,\mu,k,\epsilon)}^{d}$. We can write
\begin{multline}
N_{2} := || b(m) \int_{0}^{\tau^{k}} (\tau^{k}-s)^{\frac{1}{k}}\int_{-\infty}^{+\infty} f(m-m_{1})Q(im_{1})g(s^{1/k},m_{1})dm_{1}
\frac{ds}{s}||_{(\nu,\beta,\mu,k,\epsilon)} \\
= \sup_{\tau \in \bar{D}(0,\rho) \cup S_{d},m \in \mathbb{R}} (1+|m|)^{\mu}
\frac{1 + |\frac{\tau}{\epsilon}|^{2k}}{|\frac{\tau}{\epsilon}|}\exp( \beta|m| - \nu|\frac{\tau}{\epsilon}|^{k} )\\
\times|b(m) \int_{0}^{\tau^k}\int_{-\infty}^{+\infty} \{ (1 + |m-m_{1}|)^{\mu}
\exp( \beta|m-m_{1}| )f(m-m_{1}) \} \\
\times \{ (1+|m_{1})^{\mu} \exp( \beta |m_{1}|) 
\exp( -\frac{\nu|s|}{|\epsilon|^k} ) \frac{1 + \frac{|s|^2}{|\epsilon|^{2k}}}{\frac{|s|^{1/k}}{|\epsilon|}}g(s^{1/k},m_{1}) \}
\times \mathcal{D}(\tau,s,m,m_{1},\epsilon) dm_{1}ds |
\end{multline}
where
$$
\mathcal{D}(\tau,s,m,m_{1},\epsilon) =
\frac{ Q(im_{1})e^{-\beta |m_{1}|} e ^{-\beta|m-m_{1}|} }{(1 + |m-m_{1}|)^{\mu} (1 + |m_{1}|)^{\mu}} \times
\frac{ \exp( \frac{\nu |s|}{|\epsilon|^k} ) }{1 + \frac{|s|^2}{|\epsilon|^{2k}}} \frac{ |s|^{1/k} }{|\epsilon|}
(\tau^{k}-s)^{1/k} \frac{1}{s}
$$
Again, we know that there exist constants $\mathfrak{Q},\mathfrak{R}>0$ such that
$$ |Q(im_{1})| \leq \mathfrak{Q}(1+|m_{1}|)^{\mathrm{deg}(Q)} \ \ , \ \ |R(im)| \geq \mathfrak{R}(1+|m|)^{\mathrm{deg}(R)} $$
for all $m,m_{1} \in \mathbb{R}$. By means of the triangular inequality $|m| \leq |m_{1}|+|m-m_{1}|$, we get that
\begin{equation}
N_{2} \leq C_{4.1}(\epsilon)C_{4.2}||f(m)||_{(\beta,\mu)}||g(\tau,m)||_{(\nu,\beta,\mu,k,\epsilon)} \label{norm_b_int_conv_f_Q_g<}
\end{equation}
where
$$ C_{4.1}(\epsilon) = \sup_{\tau \in \bar{D}(0,\rho) \cup S_{d}} \frac{1 + |\frac{\tau}{\epsilon}|^{2k}}{|\frac{\tau}{\epsilon}|}
\exp( - \nu|\frac{\tau}{\epsilon}|^{k} ) \int_{0}^{|\tau|^k} \frac{ \exp( \nu h/|\epsilon|^k) }{1 + \frac{h^2}{|\epsilon|^{2k}}}
\frac{ h^{\frac{1}{k}-1} }{ |\epsilon| } (|\tau|^{k} - h)^{1/k} dh
$$
and
$$
C_{4.2} = \frac{\mathfrak{Q}}{\mathfrak{R}} \sup_{m \in \mathbb{R}} (1 + |m|)^{\mu - \mathrm{deg}(R)}
\int_{-\infty}^{+\infty} \frac{1}{(1 + |m-m_{1}|)^{\mu}(1+|m_{1}|)^{\mu - \mathrm{deg}(Q)}} dm_{1}.
$$
From the estimates (\ref{C1_chg_var}) and (\ref{tilde_A_1_bounded}), we know that there exists a constant $C_{4.1}>0$ such that
\begin{equation}
C_{4.1}(\epsilon) \leq C_{4.1}|\epsilon| \label{C41_epsilon_maj}
\end{equation}
and from the estimates for (\ref{defin_C_3.1}), we know that $C_{4.2}$ is finite under the assumption (\ref{cond_R_Q}) provided that
$\mu > \mathrm{deg}(Q)+1$. Finally, gathering this latter bound estimates together with (\ref{norm_b_int_conv_f_Q_g<}) and
(\ref{C41_epsilon_maj}) yields the result (\ref{norm_conv_f_g<norm_f_beta_mu_times_norm_g}).
\end{proof}

In the next proposition, we show that $(E_{(\beta,\mu)},||.||_{(\beta,\mu)})$ is a Banach algebra for some noncommutative
product $\star$ introduced below.

\begin{prop} Let $Q_{1}(X),Q_{2}(X),R(X) \in \mathbb{C}[X]$ be polynomials such that
\begin{equation}
\mathrm{deg}(R) \geq \mathrm{deg}(Q_{1}) \ \ , \ \ \mathrm{deg}(R) \geq \mathrm{deg}(Q_{2}) \ \ , \ \ R(im) \neq 0,
\label{cond_R_Q1_Q2}
\end{equation}
for all $m \in \mathbb{R}$. Assume that $\mu > \max( \mathrm{deg}(Q_{1})+1, \mathrm{deg}(Q_{2})+1 )$. Then, there exists a
constant $C_{5}>0$ (depending on $Q_{1},Q_{2},R,\mu$) such that
\begin{multline}
|| \frac{1}{R(im)} \int_{-\infty}^{+\infty} Q_{1}(i(m-m_{1})) f(m-m_{1}) Q_{2}(im_{1})g(m_{1}) dm_{1} ||_{(\beta,\mu)}\\
\leq C_{5} ||f(m)||_{(\beta,\mu)}||g(m)||_{(\beta,\mu)}
\end{multline}
for all $f(m),g(m) \in E_{(\beta,\mu)}$. Therefore, $(E_{(\beta,\mu)},||.||_{(\beta,\mu)})$ becomes a Banach algebra for the product
$\star$ defined by
$$ f \star g (m) = \frac{1}{R(im)} \int_{-\infty}^{+\infty} Q_{1}(i(m-m_{1})) f(m-m_{1}) Q_{2}(im_{1})g(m_{1}) dm_{1}.$$
As a particular case, when $f,g \in E_{(\beta,\mu)}$ with $\beta>0$ and $\mu>1$, the classical convolution
product
$$ f \ast g (m) = \int_{-\infty}^{+\infty} f(m-m_{1})g(m_{1}) dm_{1} $$
belongs to $E_{(\beta,\mu)}$.
\end{prop}
\begin{proof} The proof is similar to the one of Proposition 3. Let $f(m),g(m) \in E_{(\beta,\mu)}$. We write
\begin{multline}
|| \frac{1}{R(im)} \int_{-\infty}^{+\infty} Q_{1}(i(m-m_{1}))f(m-m_{1})Q_{2}(im_{1})g(m_{1}) dm_{1} ||_{(\beta,\mu)}\\
= \sup_{m \in \mathbb{R}} (1+|m|)^{\mu} e^{\beta |m|} |\frac{1}{R(im)} \int_{-\infty}^{+\infty} \{ (1 + |m-m_{1}|)^{\mu}
e^{\beta|m-m_{1}|}f(m-m_{1}) \} \\
\times \{ (1+|m_{1}|)^{\mu} e^{\beta|m_{1}|}g(m_{1}) \} \times \mathcal{E}(m,m_{1}) dm_{1} |
\end{multline}
where
$$ \mathcal{E}(m,m_{1}) = \frac{ e^{-\beta|m-m_{1}|} e^{-\beta|m_{1}|} }{(1+|m-m_{1}|)^{\mu}(1+|m_{1}|)^{\mu}}
Q_{1}(i(m-m_{1}))Q_{2}(im_{1}). $$
Using the triangular inequality $|m| \leq |m_{1}|+|m-m_{1}|$ and the estimates in (\ref{Q1_Q2_R_deg_order}), we get that
\begin{multline}
|| \frac{1}{R(im)} \int_{-\infty}^{+\infty} Q_{1}(i(m-m_{1}))f(m-m_{1})Q_{2}(im_{1})g(m_{1}) dm_{1} ||_{(\beta,\mu)}\\
\leq C_{5} ||f(m)||_{(\beta,\mu)} ||g(m)||_{(\beta,\mu)}
\end{multline}
where
$$
C_{5} = \frac{\mathfrak{Q}_{1}\mathfrak{Q}_{2}}{\mathfrak{R}}
\sup_{m \in \mathbb{R}} (1+|m|)^{\mu - \mathrm{deg}(R)} \int_{-\infty}^{+\infty}
\frac{1}{(1+|m-m_{1}|)^{\mu - \mathrm{deg}(Q_{1})}(1+|m_{1}|)^{\mu - \mathrm{deg}(Q_{2})}} dm_{1}
$$
which is finite whenever $\mu > \max( \mathrm{deg}(Q_{1})+1, \mathrm{deg}(Q_{2})+1 )$ provided that
(\ref{cond_R_Q1_Q2}) holds as explained in Proposition 3 (see (\ref{defin_C_3.1})).
\end{proof}

\section{Laplace transform, asymptotic expansions and Fourier transform}

We give a definition of $k-$Borel summability of formal series with coefficients in a Banach space which is a slightly modified version of
the one given in \cite{ba}, Section 3.2, in order to fit our necessities.

\begin{defin} Let $k \geq 1$ be an integer. Let $m_{k}(n)$ be the sequence defined by
$$ m_{k}(n) = \Gamma(\frac{n}{k}) =  \int_{0}^{+\infty} t^{\frac{n}{k}-1} e^{-t} dt $$
for all $n \geq 1$. A formal series
$$\hat{X}(T) = \sum_{n=1}^{\infty}  a_{n}T^{n} \in T\mathbb{E}[[T]]$$
with coefficients in a Banach space $( \mathbb{E}, ||.||_{\mathbb{E}} )$ is said to be $m_{k}-$summable
with respect to $t$ in the direction $d \in [0,2\pi)$ if \medskip

{\bf i)} there exists $\rho \in \mathbb{R}_{+}$ such that the following formal series, called a formal $m_{k}-$Borel transform of
$\hat{X}$ 
$$ \mathcal{B}_{m_k}(\hat{X})(\tau) = \sum_{n=1}^{\infty} \frac{ a_{n} }{ \Gamma(\frac{n}{k}) } \tau^{n}
\in \tau\mathbb{E}[[\tau]],$$
is absolutely convergent for $|\tau| < \rho$. \medskip

{\bf ii)} there exists $\delta > 0$ such that the series $\mathcal{B}_{m_k}(\hat{X})(\tau)$ can be analytically continued with
respect to $\tau$ in a sector
$S_{d,\delta} = \{ \tau \in \mathbb{C}^{\ast} : |d - \mathrm{arg}(\tau) | < \delta \} $. Moreover, there exist $C >0$ and $K >0$
such that
$$ ||\mathcal{B}_{m_k}(\hat{X})(\tau)||_{\mathbb{E}}
\leq C e^{ K|\tau|^{k} } $$
for all $\tau \in S_{d, \delta}$.
\end{defin}
If this is so, the vector valued Laplace transform of $\mathcal{B}_{k}(\hat{X})(\tau)$ in the direction $d$ is defined by
$$ \mathcal{L}^{d}_{m_k}(\mathcal{B}(\hat{X}))(T) = k \int_{L_{\gamma}}
\mathcal{B}_{k}(\hat{X})(u) e^{ - ( u/T )^{k} } \frac{d u}{u},$$
along a half-line $L_{\gamma} = \mathbb{R}_{+}e^{i\gamma} \subset S_{d,\delta} \cup \{ 0 \}$, where $\gamma$ depends on
$T$ and is chosen in such a way that $\cos(k(\gamma - \mathrm{arg}(T))) \geq \delta_{1} > 0$, for some fixed $\delta_{1}$.
The function $\mathcal{L}^{d}_{m_k}(\mathcal{B}(\hat{X}))(T)$ is well defined, holomorphic and bounded in any sector
$$ S_{d,\theta,R^{1/k}} = \{ T \in \mathbb{C}^{\ast} : |T| < R^{1/k} \ \ , \ \ |d - \mathrm{arg}(T) | < \theta/2 \},$$
where $\frac{\pi}{k} < \theta < \frac{\pi}{k} + 2\delta$ and
$0 < R < \delta_{1}/K$. This function is called the $m_{k}-$sum of the formal series $\hat{X}(T)$ in the direction $d$.\medskip

\noindent We now state some elementary properties concerning the $m_{k}-$sums of formal power series.\\

\noindent 1) The function $\mathcal{L}^{d}_{m_k}(\mathcal{B}_{k}(\hat{X}))(T)$ has the formal series $\hat{X}(T)$ as
Gevrey asymptotic
expansion of order $1/k$ with respect to $t$ on $S_{d,\theta,R^{1/k}}$. This means that for all
$\frac{\pi}{k} < \theta_{1} < \theta$, there exist $C,M > 0$
such that
\begin{equation}
 ||\mathcal{L}^{d}_{m_k}(\mathcal{B}_{k}(\hat{X}))(T) - \sum_{p=1}^{n-1} a_p T^{p}||_{\mathbb{E}} \leq
CM^{n}\Gamma(1+\frac{n}{k})|T|^{n} \label{Laplace_k_Gevrey_ae}
\end{equation}
for all $n \geq 2$, all $T \in S_{d,\theta_{1},R^{1/k}}$. Moreover, from Watson's lemma (see Proposition 11 p. 75 in \cite{ba}), we get
that $\mathcal{L}^{d}_{m_k}(\mathcal{B}_{k}(\hat{X}))(T)$ is the unique holomorphic function that satisfies the estimates
(\ref{Laplace_k_Gevrey_ae}) on the sectors $S_{d,\theta_{1},R^{1/k}}$ with large aperture $\theta_{1} > \frac{\pi}{k}$.\medskip

\noindent 2) Let us assume that $( \mathbb{E}, ||.||_{\mathbb{E}} )$ also has the structure of a Banach algebra for a product $\star$.
Let $\hat{X}_{1}(T),\hat{X}_{2}(T) \in T\mathbb{E}[[T]]$ be $m_{k}-$summable formal power series in direction
$d$. Let $q_{1} \geq q_{2} \geq 1$ be integers. We assume that 
$\hat{X}_{1}(T)+\hat{X}_{2}(T)$, $\hat{X}_{1}(T) \star \hat{X}_{2}(T)$ and
$T^{q_1}\partial_{T}^{q_2}\hat{X}_{1}(T)$, which are elements of $T\mathbb{E}[[T]]$, are $m_{k}-$summable in direction $d$.
Then, the following equalities
\begin{multline}
\mathcal{L}^{d}_{m_k}(\mathcal{B}_{k}(\hat{X}_{1}))(T) +
\mathcal{L}^{d}_{m_k}(\mathcal{B}_{k}(\hat{X}_{2}))(T) =
\mathcal{L}^{d}_{m_k}(\mathcal{B}_{k}(\hat{X}_{1} + \hat{X}_{2}))(T),\\
\mathcal{L}^{d}_{m_k}(\mathcal{B}_{k}(\hat{X}_{1}))(T) \star
\mathcal{L}^{d}_{m_k}(\mathcal{B}_{k}(\hat{X}_{2}))(T) =
\mathcal{L}^{d}_{m_k}(\mathcal{B}_{k}(\hat{X}_{1} \star \hat{X}_{2}))(T)\\
T^{q_1}\partial_{T}^{q_2}\mathcal{L}^{d}_{m_k}(\mathcal{B}_{k}(\hat{X}_{1}))(T) =
\mathcal{L}^{d}_{m_k}(\mathcal{B}_{k}(T^{q_1}\partial_{T}^{q_2}\hat{X}_{1}))(T) \label{sum_prod_deriv_m_k_sum}
\end{multline}
hold for all $T \in S_{d,\theta,R^{1/k}}$. These equalities are consequence of the unicity of the function having a given
Gevrey expansion of order $1/k$ in large sectors as stated above in 1) and from the fact that the set of holomorphic functions having
Gevrey asymptotic expansion of order $1/k$ on a sector with values in the Banach algebra $\mathbb{E}$ form a
differential algebra (meaning that this set is stable with respect to the sum and product of functions and derivation in the variable
$T$) (see Theorem 18,19 and 20 in \cite{ba}).\medskip

In the next proposition, we give some identities for the $m_{k}-$Borel transform that will be useful in the sequel.
\begin{prop} Let $\hat{f}(t) = \sum_{ n \geq 1} f_{n}t^{n}$, $\hat{g}(t) = \sum_{n \geq 1} g_{n}t^{n}$ be formal series
whose coefficients $f_{n},g_{n}$ belong to some Banach space $(\mathbb{E},||.||_{\mathbb{E}})$. We assume that
$(\mathbb{E},||.||_{\mathbb{E}})$ is a Banach algebra for some product $\star$. Let $k,m \geq 1$ be integers.
The following formal identities hold.
\begin{equation}
\mathcal{B}_{m_k}(t^{k+1}\partial_{t}\hat{f}(t))(\tau) = k \tau^{k} \mathcal{B}_{m_k}(\hat{f}(t))(\tau) \label{Borel_diff}
\end{equation}
\begin{equation}
\mathcal{B}_{m_k}(t^{m}\hat{f}(t))(\tau) = \frac{\tau^{k}}{\Gamma(\frac{m}{k})}
\int_{0}^{\tau^{k}} (\tau^{k} - s)^{\frac{m}{k}-1} \mathcal{B}_{m_k}(\hat{f}(t))(s^{1/k}) \frac{ds}{s} \label{Borel_mult_monom}
\end{equation}
and
\begin{equation}
\mathcal{B}_{m_k}( \hat{f}(t) \star \hat{g}(t) )(\tau) = \tau^{k}\int_{0}^{\tau^{k}}
\mathcal{B}_{m_k}(\hat{f}(t))((\tau^{k}-s)^{1/k}) \star \mathcal{B}_{m_k}(\hat{g}(t))(s^{1/k}) \frac{1}{(\tau^{k}-s)s} ds
\label{Borel_product}
\end{equation}
\end{prop}
\begin{proof} First, we show (\ref{Borel_diff}). By definition, we have that
\begin{equation}
\mathcal{B}_{m_k}(\frac{t^{k+1}}{k}\partial_{t}\hat{f}(t))(\tau) = \sum_{n \geq 1}
\frac{ \frac{n}{k} f_{n} }{ \Gamma( \frac{n}{k} + 1) } \tau^{n+k} \label{expansion_Borel_diff} 
\end{equation}
By application of the addition formula for the Gamma function which yields
$\Gamma( \frac{n}{k} + 1) = \frac{n}{k} \Gamma( \frac{n}{k} )$ for any $n \geq 1$, we deduce
(\ref{Borel_diff}) from (\ref{expansion_Borel_diff}).

Now, we prove (\ref{Borel_mult_monom}). By definition, we can write
\begin{equation}
\mathcal{B}_{m_k}( t^{m} \hat{f}(t) )(\tau) =\frac{1}{\Gamma( \frac{m}{k} )}
\sum_{n \geq 1} \frac{f_n}{\Gamma( \frac{n}{k} )}
\frac{ \Gamma(\frac{m}{k}) \Gamma(\frac{n}{k}) }{\Gamma(\frac{m+n}{k})} \tau^{m+n}. \label{expansion_Borel_mult_monom}
\end{equation}
Using the Beta integral formula (see Appendix B in \cite{ba2}), we can write
\begin{equation}
\frac{ \Gamma( \frac{m}{k} ) \Gamma( \frac{n}{k} ) }{\Gamma( \frac{m+n}{k} )} =
\frac{\tau^{k}}{\tau^{m+n}} \int_{0}^{\tau^{k}} (\tau^{k}-s)^{\frac{m}{k} - 1} s^{\frac{n}{k}-1} ds
\label{Beta_integral_monom}
\end{equation}
for any $m,n \geq 1$. Plugging (\ref{Beta_integral_monom}) into (\ref{expansion_Borel_mult_monom}) yields (\ref{Borel_mult_monom}).

Finally, we show (\ref{Borel_product}). By definition, we have
\begin{equation}
\mathcal{B}_{m_k}( \hat{f}(t) \star \hat{g}(t) )(\tau) = \sum_{n \geq 2} (\sum_{p+q=n} \frac{f_p}{\Gamma(\frac{p}{k})} \star
\frac{g_q}{\Gamma(\frac{q}{k})} \times \frac{\Gamma(\frac{p}{k})\Gamma(\frac{q}{k})}{\Gamma(\frac{n}{k})}) \tau^{n}
\label{expansion_Borel_product}
\end{equation}
Using again the Beta integral formula, we can write
\begin{equation}
\frac{\Gamma(\frac{p}{k})\Gamma(\frac{q}{k})}{\Gamma(\frac{n}{k})} = \frac{\tau^{k}}{\tau^{n}}
\int_{0}^{\tau^k} (\tau^{k}-s)^{\frac{p}{k}-1} s^{\frac{q}{k}-1} ds \label{Beta_integral_product}
\end{equation}
when $p+q=n$ and $p,q \geq 1$. By the substitution of (\ref{Beta_integral_product}) into (\ref{expansion_Borel_product}),
we deduce (\ref{Borel_product}).
\end{proof}

In the following proposition, we recall some properties of the inverse Fourier transform
\begin{prop}
Let $f \in E_{(\beta,\mu)}$ with $\beta > 0$, $\mu > 1$. The inverse Fourier transform of $f$ is defined by
$$ \mathcal{F}^{-1}(f)(x) = \frac{1}{ (2\pi)^{1/2} } \int_{-\infty}^{+\infty} f(m) \exp( ixm ) dm $$
for all $x \in \mathbb{R}$. The function $\mathcal{F}^{-1}(f)$ extends to an analytic function on the strip
\begin{equation}
H_{\beta} = \{ z \in \mathbb{C} / |\mathrm{Im}(z)| < \beta \}. \label{strip_H_beta}
\end{equation}
Let $\phi(m) = im f(m) \in E_{(\beta,\mu - 1)}$. Then, we have
\begin{equation}
\partial_{z} \mathcal{F}^{-1}(f)(z) = \mathcal{F}^{-1}(\phi)(z) \label{dz_fourier}
\end{equation}
for all $z \in H_{\beta}$.\\
Let $g \in E_{(\beta,\mu)}$ and let $\psi(m) = f \ast g(m)$, the convolution product of $f$ and $g$, for all $m \in \mathbb{R}$.
From Proposition 5, we know that $\psi \in E_{(\beta,\mu)}$. Moreover, we have
\begin{equation}
\mathcal{F}^{-1}(f)(z)\mathcal{F}^{-1}(g)(z) = \mathcal{F}^{-1}(\psi)(z) \label{prod_fourier}
\end{equation}
for all $z \in H_{\beta}$.
\end{prop}
\begin{proof} Let $f \in E_{(\beta,\mu)}$. It is straight to check that $\mathcal{F}^{-1}(f)$ is well defined on the real line.
The fact that $\mathcal{F}^{-1}(f)$ extends to an analytic function on the strip $H_{\beta}$ follows from the
next inequality. There exists $C>0$ such that
$$ |f(m)||\exp(izm)| \leq \frac{C}{ (1 + |m|)^{\mu} } \exp( (\beta'-\beta) |m| ) $$
for all $m \in \mathbb{R}$, $z \in H_{\beta'}$, with $\beta' < \beta$. The relations (\ref{dz_fourier}),
(\ref{prod_fourier}) are classical and can be found for instance in \cite{ru}.  
\end{proof}

\section{Formal and analytic solutions of convolution initial value problems with complex parameters}

Let $k \geq 1$ and $D \geq 2$ be integers. For $1 \leq l \leq D$, let
$d_{l},\delta_{l},\Delta_{l} \geq 0$ be nonnegative integers.
We assume that 
\begin{equation}
1 = \delta_{1} \ \ , \ \ \delta_{l} < \delta_{l+1},
\end{equation}
for all $1 \leq l \leq D-1$. We make also the assumption that
\begin{equation}
d_{D} = (\delta_{D}-1)(k+1) \ \ , \ \ d_{l} > (\delta_{l}-1)(k+1) \ \ , \ \ \Delta_{D} = d_{D} - \delta_{D} + 1 \label{assum_dl_delta_l_Delta_l}
\end{equation}
for all $1 \leq l \leq D-1$. Let $Q(X),Q_{1}(X),Q_{2}(X),R_{l}(X) \in \mathbb{C}[X]$, $0 \leq l \leq D$, be polynomials such that
\begin{multline}
\mathrm{deg}(Q) \geq \mathrm{deg}(R_{D}) \geq \mathrm{deg}(R_{l}) \ \ , \ \ \mathrm{deg}(R_{D}) \geq \mathrm{deg}(Q_{1}) \ \ , \ \
\mathrm{deg}(R_{D}) \geq \mathrm{deg}(Q_{2}), \\
Q(im) \neq 0 \ \ , \ \ R_{D}(im) \neq 0
\end{multline}
for all $m \in \mathbb{R}$, all $0 \leq l \leq D-1$. We consider sequences of functions
$m \mapsto C_{0,n}(m,\epsilon)$, for all $n \geq 0$ and $m \mapsto F_{n}(m,\epsilon)$, for all $n \geq 1$,
that belong to the Banach space $E_{(\beta,\mu)}$ for some $\beta > 0$ and
$\mu > \max( \mathrm{deg}(Q_{1})+1, \mathrm{deg}(Q_{2})+1)$ and which
depend holomorphically on $\epsilon \in D(0,\epsilon_{0})$. We assume that there exist constants $K_{0},T_{0}>0$
such that
\begin{equation}
||C_{0,n}(m,\epsilon)|| _{(\beta,\mu)} \leq K_{0} (\frac{1}{T_{0}})^{n} \ \ , \ \
||F_{n}(m,\epsilon)||_{(\beta,\mu)} \leq K_{0} (\frac{1}{T_{0}})^{n} \label{norm_beta_mu_F_n}
\end{equation}
for all $n \geq 1$, for all $\epsilon \in D(0,\epsilon_{0})$. We define
$$ C_{0}(T,m,\epsilon) = \sum_{n \geq 1} C_{0,n}(m,\epsilon) T^{n} \ \ , \ \
F(T,m,\epsilon) = \sum_{n \geq 1} F_{n}(m,\epsilon) T^{n} $$
which are convergent series on $D(0,T_{0}/2)$ with values in $E_{(\beta,\mu)}$. We consider the following singular initial value problem
\begin{multline}
Q(im)(\partial_{T}U(T,m,\epsilon) ) =
\epsilon^{-1}\frac{1}{(2\pi)^{1/2}}\int_{-\infty}^{+\infty}Q_{1}(i(m-m_{1}))U(T,m-m_{1},\epsilon)\\
\times Q_{2}(im_{1})U(T,m_{1},\epsilon) dm_{1}
+ \sum_{l=1}^{D} R_{l}(im) \epsilon^{\Delta_{l} - d_{l} + \delta_{l} - 1} T^{d_{l}} \partial_{T}^{\delta_l}U(T,m,\epsilon)\\
+ \epsilon^{-1}\frac{1}{(2\pi)^{1/2}}\int_{-\infty}^{+\infty}C_{0}(T,m-m_{1},\epsilon)R_{0}(im_{1})U(T,m_{1},\epsilon) dm_{1}\\
+  \epsilon^{-1}\frac{1}{(2\pi)^{1/2}}\int_{-\infty}^{+\infty}C_{0,0}(m-m_{1},\epsilon)R_{0}(im_{1})U(T,m_{1},\epsilon) dm_{1}
+ \epsilon^{-1}F(T,m,\epsilon)
\label{SCP}
\end{multline}
for given initial data $U(0,m,\epsilon) = 0$.\medskip

\begin{prop} There exists a unique formal series
$$ \hat{U}(T,m,\epsilon) = \sum_{n \geq 1} U_{n}(m,\epsilon) T^{n} $$
solution of (\ref{SCP}) with initial data $U(0,m,\epsilon) \equiv 0$, where the coefficients
$m \mapsto U_{n}(m,\epsilon)$ belong to $E_{(\beta,\mu)}$ for $\beta>0$ and
$\mu>\max( \mathrm{deg}(Q_{1}) + 1, \mathrm{deg}(Q_{2})+1)$ given above and depend
holomorphically on $\epsilon$ in $D(0,\epsilon_{0}) \setminus \{ 0 \}$. 
\end{prop}
\begin{proof} From Proposition 5 and the conditions in the statement above, we get that the coefficients
$U_{n}(m,\epsilon)$ of $\hat{U}(T,m,\epsilon)$ are well defined, belong to $E_{(\beta,\mu)}$ for all
$\epsilon \in D(0,\epsilon_{0}) \setminus \{ 0 \}$, all $n \geq 1$ and satisfy the following recursion relation
\begin{multline}
(n+1)U_{n+1}(m,\epsilon)\\
= \frac{\epsilon^{-1}}{Q(im)}
\sum_{n_{1}+n_{2}=n,n_{1} \geq 1,n_{2} \geq 1}
\frac{1}{(2\pi)^{1/2}} \int_{-\infty}^{+\infty} Q_{1}(i(m-m_{1}))U_{n_1}(m-m_{1},\epsilon)
Q_{2}(im_{1})U_{n_2}(m_{1},\epsilon) dm_{1}\\
+ \sum_{l=1}^{D} \frac{R_{l}(im)}{Q(im)}\left( \epsilon^{\Delta_{l} - d_{l} + \delta_{l} - 1}
\Pi_{j=0}^{\delta_{l}-1} (n+\delta_{l}-d_{l}-j) \right)
U_{n+\delta_{l}-d_{l}}(m,\epsilon)\\
+ \frac{\epsilon^{-1}}{Q(im)}
\sum_{n_{1}+n_{2}=n,n_{1} \geq 1,n_{2} \geq 1} \frac{1}{(2\pi)^{1/2}} \int_{-\infty}^{+\infty}
C_{0,n_{1}}(m-m_{1},\epsilon)R_{0}(im_{1})U_{n_2}(m_{1},\epsilon) dm_{1}\\
+ \frac{\epsilon^{-1}}{(2\pi)^{1/2}Q(im)} \int_{-\infty}^{+\infty}
C_{0,0}(m-m_{1},\epsilon)R_{0}(im_{1})U_{n}(m_{1},\epsilon) dm_{1}
+ \frac{\epsilon^{-1}}{Q(im)}F_{n}(m,\epsilon)
\end{multline}
for all $n \geq \max_{1 \leq l \leq D}d_{l}$.
\end{proof}

Using the formula from \cite{taya}, p. 40, we can expand the operators
$T^{\delta_{l}(k+1)} \partial_{T}^{\delta_l}$ in the form
\begin{equation}
T^{\delta_{l}(k+1)} \partial_{T}^{\delta_l} = (T^{k+1}\partial_{T})^{\delta_l} +
\sum_{1 \leq p \leq \delta_{l}-1} A_{\delta_{l},p} T^{k(\delta_{l}-p)} (T^{k+1}\partial_{T})^{p} \label{expand_op_diff}
\end{equation}
where $A_{\delta_{l},p}$, $p=1,\ldots,\delta_{l}-1$ are real numbers.
We define integers $d_{l,k} \geq 0$ to satisfy
\begin{equation}
d_{l} + k + 1 = \delta_{l}(k+1) + d_{l,k}
\end{equation}
for all $1 \leq l \leq D$. Multiplying the equation (\ref{SCP}) by $T^{k+1}$ and using
(\ref{expand_op_diff}), we can rewrite the equation (\ref{SCP}) in the form
\begin{multline}
Q(im)( T^{k+1}\partial_{T}U(T,m,\epsilon) ) \\
= \epsilon^{-1}T^{k+1}
\frac{1}{(2\pi)^{1/2}}\int_{-\infty}^{+\infty} Q_{1}(i(m-m_{1}))U(T,m-m_{1},\epsilon) Q_{2}(im_{1})U(T,m_{1},\epsilon) dm_{1} \\
+ \sum_{l=1}^{D} R_{l}(im)\left( \epsilon^{\Delta_{l} - d_{l} + \delta_{l} - 1}T^{d_{l,k}}(T^{k+1}\partial_{T})^{\delta_l}
U(T,m,\epsilon) \right.
\\+ \sum_{1 \leq p \leq \delta_{l}-1} A_{\delta_{l},p}
\left. \epsilon^{\Delta_{l}-d_{l}+\delta_{l}-1} T^{k(\delta_{l}-p) + d_{l,k}}(T^{k+1}\partial_{T})^{p}U(T,m,\epsilon) \right)\\
+ \epsilon^{-1}T^{k+1}
\frac{1}{(2\pi)^{1/2}}\int_{-\infty}^{+\infty} C_{0}(T,m-m_{1},\epsilon) R_{0}(im_{1})U(T,m_{1},\epsilon) dm_{1}\\
+ \epsilon^{-1}T^{k+1}
\frac{1}{(2\pi)^{1/2}}\int_{-\infty}^{+\infty} C_{0,0}(m-m_{1},\epsilon) R_{0}(im_{1})U(T,m_{1},\epsilon) dm_{1}
+ \epsilon^{-1}T^{k+1}F(T,m,\epsilon)
\label{SCP_irregular}
\end{multline}

We denote
$\omega_{k}(\tau,m,\epsilon)$ the formal $m_{k}-$Borel transform of
$\hat{U}(T,m,\epsilon)$ with respect to $T$, $\varphi_{k}(\tau,m,\epsilon)$ the formal $m_{k}-$Borel transform of
$C_{0}(T,m,\epsilon)$ with respect to $T$ and $\psi_{k}(\tau,m,\epsilon)$ the formal $m_{k}-$Borel transform of
$F(T,m,\epsilon)$ with respect to $T$,
\begin{multline*}
 \omega_{k}(\tau,m,\epsilon) = \sum_{n \geq 1} U_{n}(m,\epsilon) \frac{\tau^n}{\Gamma(\frac{n}{k})} \ \ , \ \
\varphi_{k}(\tau,m,\epsilon) = \sum_{n \geq 1} C_{0,n}(m,\epsilon) \frac{\tau^n}{\Gamma(\frac{n}{k})}\\
\psi_{k}(\tau,m,\epsilon) = \sum_{n \geq 1} F_{n}(m,\epsilon) \frac{\tau^n}{\Gamma(\frac{n}{k})}
\end{multline*}
Using (\ref{norm_beta_mu_F_n}) we get that
$\varphi_{k}(\tau,m,\epsilon) \in F_{(\nu,\beta,\mu,k,\epsilon)}^{d}$ and
$\psi_{k}(\tau,m,\epsilon) \in F_{(\nu,\beta,\mu,k,\epsilon)}^{d}$, for
all $\epsilon \in D(0,\epsilon_{0}) \setminus \{ 0 \}$, any unbounded sector $S_{d}$ centered at 0 and bisecting direction
$d \in \mathbb{R}$, for some $\nu>0$. Indeed, we have that
\begin{multline}
||\varphi_{k}(\tau,m,\epsilon)||_{(\nu,\beta,\mu,k,\epsilon)} \leq \sum_{n \geq 1}
||C_{0,n}(m,\epsilon)||_{(\beta,\mu)} (\sup_{\tau \in \bar{D}(0,\rho) \cup S_{d}}
\frac{1 + |\frac{\tau}{\epsilon}|^{2k}}{|\frac{\tau}{\epsilon}|} \exp(-\nu |\frac{\tau}{\epsilon}|^{k})
\frac{|\tau|^n}{\Gamma(\frac{n}{k})}),\\
||\psi_{k}(\tau,m,\epsilon)||_{(\nu,\beta,\mu,k,\epsilon)} \leq \sum_{n \geq 1}
||F_{n}(m,\epsilon)||_{(\beta,\mu)} (\sup_{\tau \in \bar{D}(0,\rho) \cup S_{d}}
\frac{1 + |\frac{\tau}{\epsilon}|^{2k}}{|\frac{\tau}{\epsilon}|} \exp(-\nu |\frac{\tau}{\epsilon}|^{k})
\frac{|\tau|^n}{\Gamma(\frac{n}{k})}) \label{maj_norm_psi_k_1}
\end{multline}
By using the classical estimates (\ref{x_m_exp_x<}) and Stirling formula
$\Gamma(n/k) \sim (2\pi)^{1/2}(n/k)^{\frac{n}{k}-\frac{1}{2}}e^{-n/k}$ as $n$ tends to $+\infty$, we get two constants
$A_{1},A_{2}>0$ depending on $\nu,k$ such that
\begin{multline}
\sup_{\tau \in \bar{D}(0,\rho) \cup S_{d}}
\frac{1 + |\frac{\tau}{\epsilon}|^{2k}}{|\frac{\tau}{\epsilon}|} \exp(-\nu |\frac{\tau}{\epsilon}|^{k})
\frac{|\tau|^n}{\Gamma(\frac{n}{k})} = \sup_{\tau \in \bar{D}(0,\rho) \cup S_{d}}
|\epsilon|^{n} (1 + |\frac{\tau}{\epsilon}|^{2k})
|\frac{\tau}{\epsilon}|^{n-1} \frac{\exp( -\nu |\frac{\tau}{\epsilon}|^{k} )}{\Gamma(\frac{n}{k})}\\
\leq \epsilon_{0}^{n} \sup_{x \geq 0} (1+x^{2})x^{\frac{n-1}{k}}
\frac{e^{-\nu x}}{\Gamma(\frac{n}{k})} \leq \epsilon_{0}^{n}\left( (\frac{n-1}{\nu k})^{\frac{n-1}{k}}e^{-\frac{n-1}{k}} +
( \frac{n-1}{\nu k} + \frac{2}{\nu})^{\frac{n-1}{k}+2} e^{-(\frac{n-1}{k}+2)} \right)/ \Gamma(n/k)\\
\leq A_{1}\epsilon_{0}^{n}(A_{2})^{n} \label{sup_Stirling}
\end{multline}
for all $n \geq 0$, all $\epsilon \in D(0,\epsilon_{0}) \setminus \{ 0 \}$. Therefore, if $\epsilon_{0}$ fulfills
$\epsilon_{0}A_{2} < T_{0}$, we get the estimates
\begin{multline}
||\varphi_{k}(\tau,m,\epsilon)||_{(\nu,\beta,\mu,k,\epsilon)} \leq A_{1} \sum_{n \geq 1} ||C_{0,n}(m,\epsilon)||_{(\beta,\mu)}
(\epsilon_{0}A_{2})^{n} \leq \frac{A_{1}A_{2}K_{0}}{T_0} \frac{\epsilon_{0}}{ 1 - \frac{A_{2}}{T_0}\epsilon_{0} },\\
||\psi_{k}(\tau,m,\epsilon)||_{(\nu,\beta,\mu,k,\epsilon)} \leq A_{1} \sum_{n \geq 1} ||F_{n}(m,\epsilon)||_{(\beta,\mu)}
(\epsilon_{0}A_{2})^{n} \leq \frac{A_{1}A_{2}K_{0}}{T_0} \frac{\epsilon_{0}}{ 1 - \frac{A_{2}}{T_0}\epsilon_{0} }
\label{norm_F_varphi_k_psi_k_epsilon_0}
\end{multline}
for all $\epsilon \in D(0,\epsilon_{0}) \setminus \{ 0 \}$.

Using the computation rules for the formal $m_{k}-$Borel transform in Proposition 6, we deduce the following equation satisfied by
$\omega_{k}(\tau,m,\epsilon)$,
\begin{multline}
Q(im)( k \tau^{k} \omega_{k}(\tau,m,\epsilon) )
= \epsilon^{-1}
\frac{\tau^k}{\Gamma(1 + \frac{1}{k})} \int_{0}^{\tau^k}
(\tau^{k}-s)^{1/k}\\
\times \left( \frac{1}{(2\pi)^{1/2}} s\int_{0}^{s} \int_{-\infty}^{+\infty} \right.
Q_{1}(i(m-m_{1}))\omega_{k}((s-x)^{1/k},m-m_{1},\epsilon)\\
\left. \times  Q_{2}(im_{1})
\omega_{k}(x^{1/k},m_{1},\epsilon) \frac{1}{(s-x)x} dxdm_{1} \right) \frac{ds}{s}\\
+ R_{D}(im) \left( k^{\delta_D} \tau^{\delta_{D}k} \omega_{k}(\tau,m,\epsilon) \right. \\
+ \sum_{1 \leq p \leq \delta_{D}-1} A_{\delta_{D},p}
\frac{\tau^{k}}{\Gamma(\delta_{D}-p)}\int_{0}^{\tau^k} (\tau^{k}-s)^{\delta_{D}-p-1}
\left. (k^{p} s^{p} \omega_{k}(s^{1/k},m,\epsilon)) \frac{ds}{s} \right)\\
+ \sum_{l=1}^{D-1} R_{l}(im) \left( \epsilon^{\Delta_{l}-d_{l}+\delta_{l}-1} \frac{\tau^k}{\Gamma( \frac{d_{l,k}}{k} )} \right.
\int_{0}^{\tau^k} (\tau^{k}-s)^{\frac{d_{l,k}}{k}-1}(k^{\delta_l}s^{\delta_l}\omega_{k}(s^{1/k},m,\epsilon)) \frac{ds}{s}\\
+ \sum_{1 \leq p \leq \delta_{l}-1} A_{\delta_{l},p}\epsilon^{\Delta_{l}-d_{l}+\delta_{l}-1}
\frac{\tau^k}{\Gamma( \frac{d_{l,k}}{k} + \delta_{l}-p)} \int_{0}^{\tau^k}
\left. (\tau^{k}-s)^{\frac{d_{l,k}}{k}+\delta_{l}-p-1}(k^{p}s^{p}\omega_{k}(s^{1/k},m,\epsilon)) \frac{ds}{s} \right)\\
 + \epsilon^{-1}
\frac{\tau^k}{\Gamma(1 + \frac{1}{k})} \int_{0}^{\tau^k}
(\tau^{k}-s)^{1/k}\\
\times \left( \frac{1}{(2\pi)^{1/2}} s\int_{0}^{s} \int_{-\infty}^{+\infty} \right.
\left. \varphi_{k}((s-x)^{1/k},m-m_{1},\epsilon) R_{0}(im_{1}) \omega_{k}(x^{1/k},m_{1},\epsilon) \frac{1}{(s-x)x}
dxdm_{1} \right) \frac{ds}{s}\\
+ \epsilon^{-1}\frac{\tau^k}{\Gamma(1 + \frac{1}{k})} \int_{0}^{\tau^k}
(\tau^{k}-s)^{1/k} \frac{1}{(2\pi)^{1/2}} ( \int_{-\infty}^{+\infty} C_{0,0}(m-m_{1},\epsilon)R_{0}(im_{1})
\omega_{k}(s^{1/k},m_{1},\epsilon) dm_{1} )\frac{ds}{s}\\
+ \epsilon^{-1} \frac{\tau^{k}}{\Gamma(1 + \frac{1}{k})}\int_{0}^{\tau^k}
(\tau^{k}-s)^{1/k} \psi_{k}(s^{1/k},m,\epsilon) \frac{ds}{s} \label{k_Borel_equation}
\end{multline}

We make the additional assumption that there exists an unbounded sector
$$ S_{Q,R_{D}} = \{ z \in \mathbb{C} / |z| \geq r_{Q,R_{D}} \ \ , \ \ |\mathrm{arg}(z) - d_{Q,R_{D}}| \leq \eta_{Q,R_{D}} \} $$
with direction $d_{Q,R_{D}} \in \mathbb{R}$, aperture $\eta_{Q,R_{D}}>0$ for some radius $r_{Q,R_{D}}>0$ such that
\begin{equation}
\frac{Q(im)}{R_{D}(im)} \in S_{Q,R_{D}} \label{quotient_Q_RD_in_S}
\end{equation} 
for all $m \in \mathbb{R}$. We factorize the polynomial $P_{m}(\tau) = Q(im)k - R_{D}(im)k^{\delta_D}
\tau^{(\delta_{D}-1)k}$ in the form
\begin{equation}
 P_{m}(\tau) = -R_{D}(im)k^{\delta_D}\Pi_{l=0}^{(\delta_{D}-1)k-1} (\tau - q_{l}(m)) \label{factor_P_m}
\end{equation}
where
\begin{equation}
q_{l}(m) = (\frac{|Q(im)|}{|R_{D}(im)|k^{\delta_{D}-1}})^{\frac{1}{(\delta_{D}-1)k}}
\exp( \sqrt{-1}( \mathrm{arg}( \frac{Q(im)}{R_{D}(im)k^{\delta_{D}-1}}) \frac{1}{(\delta_{D}-1)k} +
\frac{2\pi l}{(\delta_{D}-1)k} ) ) \label{defin_roots}
\end{equation}
for all $0 \leq l \leq (\delta_{D}-1)k-1$, all $m \in \mathbb{R}$.

We choose an unbounded sector $S_{d}$ centered at 0, a small closed disc $\bar{D}(0,\rho)$ and we prescribe the sector
$S_{Q,R_{D}}$ in such a way that the following conditions hold.\medskip

\noindent 1) There exists a constant $M_{1}>0$ such that
\begin{equation}
|\tau - q_{l}(m)| \geq M_{1}(1 + |\tau|) \label{root_cond_1}
\end{equation}
for all $0 \leq l \leq (\delta_{D}-1)k-1$, all $m \in \mathbb{R}$, all $\tau \in S_{d} \cup \bar{D}(0,\rho)$. Indeed,
from (\ref{quotient_Q_RD_in_S}) and the explicit expression (\ref{defin_roots}) of $q_{l}(m)$, we first observe that
$|q_{l}(m)| > 2\rho$ for every $m \in \mathbb{R}$, all $0 \leq l \leq (\delta_{D}-1)k-1$ for an appropriate choice of $r_{Q,R_{D}}$
and of $\rho>0$. We also see that for all $m \in \mathbb{R}$, all $0 \leq l \leq (\delta_{D}-1)k-1$, the roots
$q_{l}(m)$ remain in a union $\mathcal{U}$ of unbounded sectors centered at 0 that do not cover a full neighborhood of
the origin in $\mathbb{C}^{\ast}$ provided that $\eta_{Q,R_{D}}$ is small enough. Therefore, one can choose an adequate
sector $S_{d}$ such that $S_{d} \cap \mathcal{U} = \emptyset$ with the property that for all
$0 \leq l \leq (\delta_{D}-1)k-1$ the quotients $q_{l}(m)/\tau$ lay outside
some small disc centered at 1 in $\mathbb{C}$ for all $\tau \in S_{d}$, all $m \in \mathbb{R}$. This yields (\ref{root_cond_1})
for some small constant $M_{1}>0$.\medskip

\noindent 2) There exists a constant $M_{2}>0$ such that
\begin{equation}
|\tau - q_{l_0}(m)| \geq M_{2}|q_{l_0}(m)| \label{root_cond_2}
\end{equation}
for some $l_{0} \in \{0,\ldots,(\delta_{D}-1)k-1 \}$, all $m \in \mathbb{R}$, all $\tau \in S_{d} \cup \bar{D}(0,\rho)$. Indeed, for the
sector $S_{d}$ and the disc $\bar{D}(0,\rho)$ chosen as above in 1), we notice that for any fixed
$0 \leq l_{0} \leq (\delta_{D}-1)k-1$, the quotient $\tau/q_{l_0}(m)$ stays outside a small disc centered at 1 in $\mathbb{C}$
for all $\tau \in S_{d} \cup \bar{D}(0,\rho)$, all $m \in \mathbb{R}$. Hence (\ref{root_cond_2}) must hold for some small
constant $M_{2}>0$.\medskip

By construction
of the roots (\ref{defin_roots}) in the factorization (\ref{factor_P_m}) and using the lower bound estimates
(\ref{root_cond_1}), (\ref{root_cond_2}), we get a constant $C_{P}>0$ such that
\begin{multline}
|P_{m}(\tau)| \geq M_{1}^{(\delta_{D}-1)k-1}M_{2}|R_{D}(im)k^{\delta_D}
|(\frac{|Q(im)|}{|R_{D}(im)|k^{\delta_{D}-1}})^{\frac{1}{(\delta_{D}-1)k}} (1+|\tau|)^{(\delta_{D}-1)k-1}\\
\geq M_{1}^{(\delta_{D}-1)k-1}M_{2}\frac{k^{\delta_D}}{(k^{\delta_{D}-1})^{\frac{1}{(\delta_{D}-1)k}}}
(r_{Q,R_{D}})^{\frac{1}{(\delta_{D}-1)k}} |R_{D}(im)| \\
\times (\min_{x \geq 0}
\frac{(1+x)^{(\delta_{D}-1)k-1}}{(1+x^{k})^{(\delta_{D}-1) - \frac{1}{k}}}) (1 + |\tau|^{k})^{(\delta_{D}-1)-\frac{1}{k}}\\
= C_{P} (r_{Q,R_{D}})^{\frac{1}{(\delta_{D}-1)k}} |R_{D}(im)| (1+|\tau|^{k})^{(\delta_{D}-1)-\frac{1}{k}} \label{low_bounds_P_m}
\end{multline}
for all $\tau \in S_{d} \cup \bar{D}(0,\rho)$, all $m \in \mathbb{R}$.

In the next proposition, we give sufficient conditions under which the equation (\ref{k_Borel_equation}) has a solution
$\omega_{k}(\tau,m,\epsilon)$ in the Banach space $F_{(\nu,\beta,\mu,k,\epsilon)}^{d}$ where $\beta,\mu$ are defined above.

\begin{prop} Under the assumption that
\begin{equation}
\delta_{D} \geq \delta_{l} + \frac{2}{k} \ \ , \ \ \Delta_{l}-d_{l}+\delta_{l} + k(\delta_{l} - \delta_{D}) + d_{l,k} \geq 0,
\label{constraints_k_Borel_equation}
\end{equation}
for all $1 \leq l \leq D-1$, there exist a radius $r_{Q,R_{D}}>0$, a constant $\varpi>0$ and constants
$\zeta_{0},\zeta_{1},\zeta_{2}>0$ (depending on $Q_{1},Q_{2},k,C_{P},\mu,\nu,\epsilon_{0},R_{l},\Delta_{l},\delta_{l},d_{l}$ for
$0 \leq l \leq D$) such that if
\begin{equation}
||C_{0,0}(m,\epsilon)||_{(\beta,\mu)} \leq \zeta_{0} \ \ , \ \
||\varphi_{k}(\tau,m,\epsilon)||_{(\nu,\beta,\mu,k,\epsilon)} \leq \zeta_{1} \ \ , \ \
||\psi_{k}(\tau,m,\epsilon)||_{(\nu,\beta,\mu,k,\epsilon)} \leq \zeta_{2} \label{norm_F_varphi_k_psi_k_small}
\end{equation}
for all $\epsilon \in D(0,\epsilon_{0}) \setminus \{ 0 \}$, the equation (\ref{k_Borel_equation}) has a unique solution
$\omega_{k}^{d}(\tau,m,\epsilon)$ in the space $F_{(\nu,\beta,\mu,k,\epsilon)}^{d}$ where $\beta,\mu>0$ are defined in
Proposition 8 which verifies $||\omega_{k}^{d}(\tau,m,\epsilon)||_{(\nu,\beta,\mu,k,\epsilon)} \leq \varpi$, for all
$\epsilon \in D(0,\epsilon_{0}) \setminus \{ 0 \}$.
\end{prop}
\begin{proof} We start the proof with a lemma which provides appropriate conditions in order to apply a fixed point theorem.
\begin{lemma} One can choose the constant $r_{Q,R_{D}}>0$, a constant $\varpi$ small enough and three constants
$\zeta_{0},\zeta_{1},\zeta_{2}>0$ (depending on $Q_{1},Q_{2},k,C_{P},\mu,\nu,\epsilon_{0},R_{l},\Delta_{l},\delta_{l},d_{l}$ for
$0 \leq l \leq D$) such that if
$$ ||C_{0,0}(m,\epsilon)||_{(\beta,\mu)} \leq \zeta_{0} \ \ , \ \
||\varphi_{k}(\tau,m,\epsilon)||_{(\nu,\beta,\mu,k,\epsilon)} \leq \zeta_{1} \ \ , \ \
||\psi_{k}(\tau,m,\epsilon)||_{(\nu,\beta,\mu,k,\epsilon)} \leq \zeta_{2}
$$
for all $\epsilon \in D(0,\epsilon_{0}) \setminus \{ 0 \}$ the map $\mathcal{H}_{\epsilon}$ defined by

\begin{multline}
\mathcal{H}_{\epsilon}(w(\tau,m)) :=
 \frac{\epsilon^{-1}}{P_{m}(\tau) \Gamma(1 + \frac{1}{k})} \int_{0}^{\tau^k}(\tau^{k}-s)^{1/k} \\
\times \left( \frac{1}{(2\pi)^{1/2}} s\int_{0}^{s} \int_{-\infty}^{+\infty} \right.
Q_{1}(i(m-m_{1}))w((s-x)^{1/k},m-m_{1}) \\
\left. \times Q_{2}(im_{1})w(x^{1/k},m_{1}) \frac{1}{(s-x)x} dxdm_{1} \right) \frac{ds}{s}\\
+ \frac{R_{D}(im)}{P_{m}(\tau)} \left\{ \sum_{1 \leq p \leq \delta_{D}-1}
\frac{A_{\delta_{D},p}}{\Gamma(\delta_{D}-p)}\int_{0}^{\tau^k} (\tau^{k}-s)^{\delta_{D}-p-1}
(k^{p} s^{p} w(s^{1/k},m)) \frac{ds}{s} \right\}\\
+ \sum_{l=1}^{D-1} \frac{R_{l}(im)}{P_{m}(\tau)}
\left\{ \frac{ \epsilon^{\Delta_{l}-d_{l}+\delta_{l}-1}}{\Gamma( \frac{d_{l,k}}{k} )}
\int_{0}^{\tau^k} (\tau^{k}-s)^{\frac{d_{l,k}}{k}-1}(k^{\delta_l}s^{\delta_l}w(s^{1/k},m)) \frac{ds}{s} \right.\\
+ \sum_{1 \leq p \leq \delta_{l}-1}
\frac{A_{\delta_{l},p}\epsilon^{\Delta_{l}-d_{l}+\delta_{l}-1}}{\Gamma( \frac{d_{l,k}}{k} + \delta_{l}-p)}
\int_{0}^{\tau^k}
\left. (\tau^{k}-s)^{\frac{d_{l,k}}{k}+\delta_{l}-p-1}(k^{p}s^{p}w(s^{1/k},m)) \frac{ds}{s} \right \}\\
+  \frac{\epsilon^{-1}}{P_{m}(\tau) \Gamma(1 + \frac{1}{k})} \int_{0}^{\tau^k}(\tau^{k}-s)^{1/k} \\
\times \left( \frac{1}{(2\pi)^{1/2}} s\int_{0}^{s} \int_{-\infty}^{+\infty}
 \varphi_{k}((s-x)^{1/k},m-m_{1},\epsilon) R_{0}(im_{1})w(x^{1/k},m_{1}) \frac{1}{(s-x)x} dxdm_{1} \right) \frac{ds}{s}\\
+ \frac{\epsilon^{-1}}{P_{m}(\tau) \Gamma(1 + \frac{1}{k})} \int_{0}^{\tau^k}(\tau^{k}-s)^{1/k}
\frac{1}{(2\pi)^{1/2}}(\int_{-\infty}^{+\infty} C_{0,0}(m-m_{1},\epsilon)R_{0}(im_{1})w(s^{1/k},m_{1}) dm_{1})\frac{ds}{s}\\
+ \frac{\epsilon^{-1}}{P_{m}(\tau)\Gamma(1 + \frac{1}{k})}\int_{0}^{\tau^k}
(\tau^{k}-s)^{1/k} \psi_{k}(s^{1/k},m,\epsilon) \frac{ds}{s}
\end{multline}
satisfy the next properties.\\
{\bf i)} The following inclusion holds
\begin{equation}
\mathcal{H}_{\epsilon}(\bar{B}(0,\varpi)) \subset \bar{B}(0,\varpi) \label{H_inclusion}
\end{equation}
where $\bar{B}(0,\varpi)$ is the closed ball of radius $\varpi>0$ centered at 0 in $F_{(\nu,\beta,\mu,k,\epsilon)}^{d}$,
for all $\epsilon \in D(0,\epsilon_{0}) \setminus \{ 0 \}$.\\
{\bf ii)} We have
\begin{equation}
|| \mathcal{H}_{\epsilon}(w_{1}) - \mathcal{H}_{\epsilon}(w_{2})||_{(\nu,\beta,\mu,k,\epsilon)}
\leq \frac{1}{2} ||w_{1} - w_{2}||_{(\nu,\beta,\mu,k,\epsilon)}
\label{H_shrink}
\end{equation}
for all $w_{1},w_{2} \in \bar{B}(0,\varpi)$, for all $\epsilon \in D(0,\epsilon_{0}) \setminus \{ 0 \}$.
\end{lemma}
\begin{proof} We first check the property (\ref{H_inclusion}). Let $\epsilon \in D(0,\epsilon_{0}) \setminus \{ 0 \}$ and
$w(\tau,m)$ be in $F_{\nu,\beta,\mu,k,\epsilon}^{d}$. We take $\zeta_{0},\zeta_{1}, \zeta_{2}, \varpi > 0$ such that
\begin{multline*}
||w(\tau,m)||_{(\nu,\beta,\mu,k,\epsilon)} \leq \varpi \ \ , \ \ ||C_{0,0}(m,\epsilon)||_{(\beta,\mu)} \leq \zeta_{0} \ \ , \ \
||\varphi_{k}(\tau,m,\epsilon)||_{(\nu,\beta,\mu,k,\epsilon)} \leq \zeta_{1},\\
||\psi_{k}(\tau,m,\epsilon)||_{(\nu,\beta,\mu,k,\epsilon)} \leq \zeta_{2},
\end{multline*}
for all $\epsilon \in D(0,\epsilon_{0}) \setminus \{ 0 \}$.

Using Lemma 1 and Proposition 3 with the lower bound estimates (\ref{low_bounds_P_m}) we get that
\begin{multline}
||\frac{\epsilon^{-1}}{P_{m}(\tau) \Gamma(1 + \frac{1}{k})} \int_{0}^{\tau^k}(\tau^{k}-s)^{1/k} \\
\times \left( \frac{1}{(2\pi)^{1/2}} s\int_{0}^{s} \int_{-\infty}^{+\infty} \right.
Q_{1}(i(m-m_{1}))w((s-x)^{1/k},m-m_{1}) \\
\left. \times Q_{2}(im_{1})w(x^{1/k},m_{1}) \frac{1}{(s-x)x} dxdm_{1} \right) \frac{ds}{s}||_{(\nu,\beta,\mu,k,\epsilon)}\\
\leq \frac{1}{\Gamma(1 + \frac{1}{k}) (2\pi)^{1/2}}
\frac{ C_{3} ||w(\tau,m)||_{(\nu,\beta,\mu,k,\epsilon)}^{2} }{ C_{P}
(r_{Q,R_{D}})^{\frac{1}{(\delta_{D}-1)k}} }\\
\leq \frac{1}{\Gamma(1 + \frac{1}{k}) (2\pi)^{1/2}}
\frac{ C_{3} \varpi^{2} }{C_{P}
(r_{Q,R_{D}})^{\frac{1}{(\delta_{D}-1)k}} } \label{fix_point_norm_estim_1}
\end{multline}
Moreover, for $0 \leq p \leq \delta_{D}-1$ and by means of Proposition 2 {\bf i)}, we deduce
\begin{multline}
||\frac{R_{D}(im)}{P_{m}(\tau)} \frac{A_{\delta_{D},p}}{\Gamma(\delta_{D}-p)}\int_{0}^{\tau^k} (\tau^{k}-s)^{\delta_{D}-p-1}
(k^{p} s^{p} w(s^{1/k},m)) \frac{ds}{s}||_{(\nu,\beta,\mu,k,\epsilon)} \\
\leq \frac{A_{\delta_{D},p}k^{p}C_{2.1}|\epsilon|}{\Gamma(\delta_{D}-p)C_{P}(r_{Q,R})^{\frac{1}{(\delta_{D}-1)k}}}
||w(\tau,m)||_{(\nu,\beta,\mu,k,\epsilon)} \\
\leq
\frac{A_{\delta_{D},p}k^{p}C_{2.1}\epsilon_{0}}{\Gamma(\delta_{D}-p)C_{P}(r_{Q,R})^{\frac{1}{(\delta_{D}-1)k}}}\varpi.
\label{fix_point_norm_estim_2}
\end{multline}

With the help of Proposition 2 {\bf ii)} and due to the assumptions of (\ref{constraints_k_Borel_equation}) we also get
that
\begin{multline}
||\frac{R_{l}(im)}{P_{m}(\tau)}\frac{ \epsilon^{\Delta_{l}-d_{l}+\delta_{l}-1}}{\Gamma( \frac{d_{l,k}}{k} )}
\int_{0}^{\tau^k} (\tau^{k}-s)^{\frac{d_{l,k}}{k}-1}(k^{\delta_l}s^{\delta_l}w(s^{1/k},m)) \frac{ds}{s}
||_{(\nu,\beta,\mu,k,\epsilon)}\\
\leq \frac{ k^{\delta_l}C_{2.2}}{\Gamma(\frac{d_{l,k}}{k}) C_{P}(r_{Q,R_{D}})^{\frac{1}{(\delta_{D}-1)k}}}
|\epsilon|^{\Delta_{l}-d_{l}+\delta_{l} + k(\delta_{l} - \delta_{D}) + d_{l,k}}\sup_{m \in \mathbb{R}} |\frac{R_{l}(im)}{R_{D}(im)}|
||w(\tau,m)||_{(\nu,\beta,\mu,k,\epsilon)}\\
\leq \frac{ k^{\delta_l}C_{2.2}}{\Gamma(\frac{d_{l,k}}{k}) C_{P}(r_{Q,R_{D}})^{\frac{1}{(\delta_{D}-1)k}}}
\epsilon_{0}^{\Delta_{l}-d_{l}+\delta_{l} + k(\delta_{l} - \delta_{D}) + d_{l,k}}
\sup_{m \in \mathbb{R}} |\frac{R_{l}(im)}{R_{D}(im)}|\varpi.
\label{fix_point_norm_estim_3}
\end{multline}
and that
\begin{multline}
|| \frac{R_{l}(im)}{P_{m}(\tau)}\frac{A_{\delta_{l},p}
\epsilon^{\Delta_{l}-d_{l}+\delta_{l}-1}}{\Gamma( \frac{d_{l,k}}{k} + \delta_{l}-p)}
\int_{0}^{\tau^k}
(\tau^{k}-s)^{\frac{d_{l,k}}{k}+\delta_{l}-p-1}(k^{p}s^{p}w(s^{1/k},m)) \frac{ds}{s}||_{(\nu,\beta,\mu,k,\epsilon)}\\
\leq \frac{|A_{\delta_{l},p}| k^{p}C_{2.2}}{\Gamma( \frac{d_{l,k}}{k} + \delta_{l}-p)
C_{P}(r_{Q,R_{D}})^{\frac{1}{(\delta_{D}-1)k}}} |\epsilon|^{\Delta_{l} - d_{l} + \delta_{l} + k(\delta_{l}-\delta_{D}) + d_{l,k}}
\sup_{m \in \mathbb{R}} |\frac{R_{l}(im)}{R_{D}(im)}|||w(\tau,m)||_{(\nu,\beta,\mu,k,\epsilon)}\\
\leq \frac{|A_{\delta_{l},p}| k^{p}C_{2.2}}{\Gamma( \frac{d_{l,k}}{k} + \delta_{l}-p)
C_{P}(r_{Q,R_{D}})^{\frac{1}{(\delta_{D}-1)k}}} \epsilon_{0}^{\Delta_{l} - d_{l} + \delta_{l} + k(\delta_{l}-\delta_{D}) + d_{l,k}}
\sup_{m \in \mathbb{R}} |\frac{R_{l}(im)}{R_{D}(im)}| \varpi. \label{fix_point_norm_estim_4}
\end{multline}

Using Lemma 1 and Proposition 3 again with the lower bound estimates (\ref{low_bounds_P_m}) we get that
\begin{multline}
||\frac{\epsilon^{-1}}{P_{m}(\tau) \Gamma(1 + \frac{1}{k})} \int_{0}^{\tau^k}(\tau^{k}-s)^{1/k} \\
\times \left( \frac{1}{(2\pi)^{1/2}} s\int_{0}^{s} \int_{-\infty}^{+\infty} \right.
 \varphi_{k}((s-x)^{1/k},m-m_{1},\epsilon) \\
\times \left. R_{0}(im_{1}) w(x^{1/k},m_{1}) \frac{1}{(s-x)x}
dxdm_{1} \right) \frac{ds}{s}||_{(\nu,\beta,\mu,k,\epsilon)}\\
\leq \frac{1}{\Gamma(1 + \frac{1}{k}) (2\pi)^{1/2}}
\frac{ C_{3} ||\varphi_{k}(\tau,m,\epsilon)||_{(\nu,\beta,\mu,k,\epsilon)}
||w(\tau,m)||_{(\nu,\beta,\mu,k,\epsilon)} }{ C_{P}
(r_{Q,R_{D}})^{\frac{1}{(\delta_{D}-1)k}} }\\
\leq \frac{1}{\Gamma(1 + \frac{1}{k}) (2\pi)^{1/2}}
\frac{ C_{3} \zeta_{1}\varpi }{ C_{P}
(r_{Q,R_{D}})^{\frac{1}{(\delta_{D}-1)k}} } \label{fix_point_norm_estim_6}
\end{multline}

Moreover, using Proposition 4, we also get
\begin{multline}
||\frac{\epsilon^{-1}}{P_{m}(\tau) \Gamma(1 + \frac{1}{k})} \int_{0}^{\tau^k}(\tau^{k}-s)^{1/k}
\frac{1}{(2\pi)^{1/2}}(\int_{-\infty}^{+\infty} C_{0,0}(m-m_{1},\epsilon)\\
\times R_{0}(im_{1})w(s^{1/k},m_{1}) dm_{1} ) \frac{ds}{s}
||_{(\nu,\beta,\mu,\epsilon)} \leq \frac{1}{\Gamma(1 + \frac{1}{k}) (2\pi)^{1/2}}
\frac{ C_{4} \zeta_{0}\varpi }{ C_{P}
(r_{Q,R_{D}})^{\frac{1}{(\delta_{D}-1)k}} } \label{fix_point_norm_estim_7}
\end{multline}

Finally, from Lemma 1 and Proposition 1, one gets
\begin{multline}
||\frac{\epsilon^{-1}}{P_{m}(\tau)\Gamma(1 + \frac{1}{k})}\int_{0}^{\tau^k}
(\tau^{k}-s)^{1/k} \psi_{k}(s^{1/k},m,\epsilon) \frac{ds}{s}||_{(\nu,\beta,\mu,k,\epsilon)}\\
\leq \frac{C_{1}}{\Gamma(1+\frac{1}{k})C_{P}(r_{Q,R_{D}})^{\frac{1}{(\delta_{D}-1)k}}\min_{m \in \mathbb{R}} |R_{D}(im)|}
||\psi_{k}(\tau,m,\epsilon)||_{(\nu,\beta,\mu,k,\epsilon)}\\
\leq \frac{C_{1}}{\Gamma(1+\frac{1}{k})C_{P}(r_{Q,R_{D}})^{\frac{1}{(\delta_{D}-1)k}}\min_{m \in \mathbb{R}} |R_{D}(im)|}
 \zeta_{2} \label{fix_point_norm_estim_8}
\end{multline}

Now, we choose $\varpi, \zeta_{0}, \zeta_{1},\zeta_{2} >0$ and $r_{Q,R_{D}}>0$ such that

\begin{multline}
 \frac{1}{\Gamma(1 + \frac{1}{k}) (2\pi)^{1/2}}
\frac{ C_{3} \varpi^{2} }{ C_{P}
(r_{Q,R_{D}})^{\frac{1}{(\delta_{D}-1)k}} }
+ \sum_{p=1}^{\delta_{D}-1}
\frac{|A_{\delta_{D},p}|k^{p}C_{2.1}\epsilon_{0}}{\Gamma(\delta_{D}-p)C_{P}(r_{Q,R})^{\frac{1}{(\delta_{D}-1)k}}}\varpi\\
+ \sum_{l=1}^{D-1} \frac{ k^{\delta_l}C_{2.2}}{\Gamma(\frac{d_{l,k}}{k}) C_{P}(r_{Q,R_{D}})^{\frac{1}{(\delta_{D}-1)k}}}
\epsilon_{0}^{\Delta_{l}-d_{l}+\delta_{l} + k(\delta_{l} - \delta_{D}) + d_{l,k}}
\sup_{m \in \mathbb{R}} |\frac{R_{l}(im)}{R_{D}(im)}|\varpi\\
+ \sum_{p=1}^{\delta_{l}-1} \frac{|A_{\delta_{l},p}| k^{p}C_{2.2}}{\Gamma( \frac{d_{l,k}}{k} + \delta_{l}-p)
C_{P}(r_{Q,R_{D}})^{\frac{1}{(\delta_{D}-1)k}}} \epsilon_{0}^{\Delta_{l} - d_{l} + \delta_{l} + k(\delta_{l}-\delta_{D}) + d_{l,k}}
\sup_{m \in \mathbb{R}} |\frac{R_{l}(im)}{R_{D}(im)}| \varpi\\
+  \frac{1}{\Gamma(1 + \frac{1}{k}) (2\pi)^{1/2}}
\frac{ (C_{3} \zeta_{1} + C_{4} \zeta_{0})\varpi }{C_{P}
(r_{Q,R_{D}})^{\frac{1}{(\delta_{D}-1)k}} }\\
+ \frac{C_{1}}{\Gamma(1+\frac{1}{k})C_{P}(r_{Q,R_{D}})^{\frac{1}{(\delta_{D}-1)k}}\min_{m \in \mathbb{R}} |R_{D}(im)|}
 \zeta_{2} \leq \varpi \label{fix_point_sum<varpi}
\end{multline}

Gathering all the norm estimates (\ref{fix_point_norm_estim_1}), (\ref{fix_point_norm_estim_2}),
(\ref{fix_point_norm_estim_3}), (\ref{fix_point_norm_estim_4}), (\ref{fix_point_norm_estim_6}),
(\ref{fix_point_norm_estim_7}),  (\ref{fix_point_norm_estim_8}) with the constraint
(\ref{fix_point_sum<varpi}), one gets (\ref{H_inclusion}).\medskip

Now, we check the second property (\ref{H_shrink}). Let $w_{1}(\tau,m),w_{2}(\tau,m)$ be in $F^{d}_{(\nu,\beta,\mu,k,\epsilon)}$.
We take $\varpi > 0$ such that
$$ ||w_{l}(\tau,m)||_{(\nu,\beta,\mu,k,\epsilon)} \leq \varpi,$$
for $l=1,2$, for all $\epsilon \in D(0,\epsilon_{0}) \setminus \{ 0 \}$. One can write
\begin{multline}
Q_{1}(i(m-m_{1}))w_{1}((s-x)^{1/k},m-m_{1})Q_{2}(im_{1})w_{1}(x^{1/k},m_{1})\\
-Q_{1}(i(m-m_{1}))w_{2}((s-x)^{1/k},m-m_{1})Q_{2}(im_{1})w_{2}(x^{1/k},m_{1})\\
= Q_{1}(i(m-m_{1}))\left(w_{1}((s-x)^{1/k},m-m_{1}) - w_{2}((s-x)^{1/k},m-m_{1})\right)Q_{2}(im_{1})w_{1}(x^{1/k},m_{1})\\
+ Q_{1}(i(m-m_{1}))w_{2}((s-x)^{1/k},m-m_{1})Q_{2}(im_{1})\left(w_{1}(x^{1/k},m_{1}) - w_{2}(x^{1/k},m_{1})\right)
\label{conv_product_w1_w2}
\end{multline}
and using Lemma 1 and Proposition 3 with the lower bound estimates (\ref{low_bounds_P_m}) we get that
\begin{multline}
||\frac{\epsilon^{-1}}{P_{m}(\tau) \Gamma(1 + \frac{1}{k})} \int_{0}^{\tau^k}(\tau^{k}-s)^{1/k} \\
\times \left( \frac{1}{(2\pi)^{1/2}} s\int_{0}^{s} \int_{-\infty}^{+\infty} \right.
(Q_{1}(i(m-m_{1}))w_{1}((s-x)^{1/k},m-m_{1}) Q_{2}(im_{1})w_{1}(x^{1/k},m_{1})\\
- Q_{1}(i(m-m_{1}))w_{2}((s-x)^{1/k},m-m_{1}) \\
\left. \times Q_{2}(im_{1})w_{2}(x^{1/k},m_{1})) \frac{1}{(s-x)x}
dxdm_{1} \right) \frac{ds}{s}||_{(\nu,\beta,\mu,k,\epsilon)}\\
\leq \frac{1}{\Gamma(1 + \frac{1}{k}) (2\pi)^{1/2}}
\frac{ C_{3} }{ C_{P}
(r_{Q,R_{D}})^{\frac{1}{(\delta_{D}-1)k}} }\\
\times ||w_{1}(\tau,m) - w_{2}(\tau,m)||_{(\nu,\beta,\mu,k,\epsilon)}(||w_{1}(\tau,m)||_{(\nu,\beta,\mu,k,\epsilon)} +
||w_{2}(\tau,m)||_{(\nu,\beta,\mu,k,\epsilon)})
\\
\leq \frac{1}{\Gamma(1 + \frac{1}{k}) (2\pi)^{1/2}}
\frac{ C_{3} 2 \varpi }{ C_{P}
(r_{Q,R_{D}})^{\frac{1}{(\delta_{D}-1)k}} }||w_{1}(\tau,m) - w_{2}(\tau,m)||_{(\nu,\beta,\mu,k,\epsilon)}
\label{fix_point_norm_estim_1_shrink}
\end{multline}

From the estimates (\ref{fix_point_norm_estim_2}),
(\ref{fix_point_norm_estim_3}), (\ref{fix_point_norm_estim_4}),
(\ref{fix_point_norm_estim_6}), (\ref{fix_point_norm_estim_7}), (\ref{fix_point_norm_estim_8}) and under the constraints
(\ref{constraints_k_Borel_equation}), we deduce that
\begin{multline}
||\frac{R_{D}(im)}{P_{m}(\tau)} \frac{A_{\delta_{D},p}}{\Gamma(\delta_{D}-p)}\int_{0}^{\tau^k} (\tau^{k}-s)^{\delta_{D}-p-1}
(k^{p} s^{p} (w_{1}(s^{1/k},m) - w_{2}(s^{1/k},m))) \frac{ds}{s}||_{(\nu,\beta,\mu,k,\epsilon)} \\
\leq \frac{|A_{\delta_{D},p}|k^{p}C_{2.1}|\epsilon|}{\Gamma(\delta_{D}-p)C_{P}(r_{Q,R_{D}})^{\frac{1}{(\delta_{D}-1)k}}}
||w_{1}(\tau,m) - w_{2}(\tau,m)||_{(\nu,\beta,\mu,k,\epsilon)} \\
\leq
\frac{|A_{\delta_{D},p}|k^{p}C_{2.1}\epsilon_{0}}{\Gamma(\delta_{D}-p)C_{P}(r_{Q,R_{D}})^{\frac{1}{(\delta_{D}-1)k}}}
||w_{1}(\tau,m) - w_{2}(\tau,m)||_{(\nu,\beta,\mu,k,\epsilon)}
\label{fix_point_norm_estim_2_shrink}
\end{multline}
and
\begin{multline}
||\frac{R_{l}(im)}{P_{m}(\tau)}\frac{ \epsilon^{\Delta_{l}-d_{l}+\delta_{l}-1}}{\Gamma( \frac{d_{l,k}}{k} )}
\int_{0}^{\tau^k} (\tau^{k}-s)^{\frac{d_{l,k}}{k}-1}(k^{\delta_l}s^{\delta_l}(w_{1}(s^{1/k},m)-w_{2}(s^{1/k},m)) \frac{ds}{s}
||_{(\nu,\beta,\mu,k,\epsilon)}\\
\leq \frac{ k^{\delta_l}C_{2.2}}{\Gamma(\frac{d_{l,k}}{k}) C_{P}(r_{Q,R_{D}})^{\frac{1}{(\delta_{D}-1)k}}}
|\epsilon|^{\Delta_{l}-d_{l}+\delta_{l} + k(\delta_{l} - \delta_{D}) + d_{l,k}}
\sup_{m \in \mathbb{R}} |\frac{R_{l}(im)}{R_{D}(im)}|
||w_{1}(\tau,m) - w_{2}(\tau,m)||_{(\nu,\beta,\mu,k,\epsilon)}\\
\leq \frac{ k^{\delta_l}C_{2.2}}{\Gamma(\frac{d_{l,k}}{k}) C_{P}(r_{Q,R_{D}})^{\frac{1}{(\delta_{D}-1)k}}}
\epsilon_{0}^{\Delta_{l}-d_{l}+\delta_{l} + k(\delta_{l} - \delta_{D}) + d_{l,k}}
\sup_{m \in \mathbb{R}} |\frac{R_{l}(im)}{R_{D}(im)}|
||w_{1}(\tau,m) - w_{2}(\tau,m)||_{(\nu,\beta,\mu,k,\epsilon)}
\label{fix_point_norm_estim_3_shrink}
\end{multline}
and that
\begin{multline}
|| \frac{R_{l}(im)}{P_{m}(\tau)}\frac{A_{\delta_{l},p}
\epsilon^{\Delta_{l}-d_{l}+\delta_{l}-1}}{\Gamma( \frac{d_{l,k}}{k} + \delta_{l}-p)}
\int_{0}^{\tau^k}(\tau^{k}-s)^{\frac{d_{l,k}}{k}+\delta_{l}-p-1}\\
\times (k^{p}s^{p}(w_{1}(s^{1/k},m) - w_{2}(s^{1/k},m))
\frac{ds}{s}||_{(\nu,\beta,\mu,k,\epsilon)}\\
\leq \frac{|A_{\delta_{l},p}| k^{p}C_{2.2}}{\Gamma( \frac{d_{l,k}}{k} + \delta_{l}-p)
C_{P}(r_{Q,R_{D}})^{\frac{1}{(\delta_{D}-1)k}}} |\epsilon|^{\Delta_{l} - d_{l} + \delta_{l} + k(\delta_{l}-\delta_{D}) + d_{l,k}}\\
\times
\sup_{m \in \mathbb{R}} |\frac{R_{l}(im)}{R_{D}(im)}|
||w_{1}(\tau,m) - w_{2}(\tau,m)||_{(\nu,\beta,\mu,k,\epsilon)}\\
\leq \frac{|A_{\delta_{l},p}| k^{p}C_{2.2}}{\Gamma( \frac{d_{l,k}}{k} + \delta_{l}-p)
C_{P}(r_{Q,R_{D}})^{\frac{1}{(\delta_{D}-1)k}}} \epsilon_{0}^{\Delta_{l} - d_{l} + \delta_{l} + k(\delta_{l}-\delta_{D}) + d_{l,k}}\\
\times
\sup_{m \in \mathbb{R}} |\frac{R_{l}(im)}{R_{D}(im)}|
||w_{1}(\tau,m) - w_{2}(\tau,m)||_{(\nu,\beta,\mu,k,\epsilon)}
 \label{fix_point_norm_estim_4_shrink}
\end{multline}
and that
\begin{multline}
||\frac{\epsilon^{-1}}{P_{m}(\tau) \Gamma(1 + \frac{1}{k})} \int_{0}^{\tau^k}(\tau^{k}-s)^{1/k}
\left( \frac{1}{(2\pi)^{1/2}} s\int_{0}^{s} \int_{-\infty}^{+\infty}
 \varphi_{k}((s-x)^{1/k},m-m_{1},\epsilon) \right. \\
\left. R_{0}(im_{1})(w_{1}(x^{1/k},m_{1})-w_{2}(x^{1/k},m_{1})) \frac{1}{(s-x)x}
dxdm_{1} \right) \frac{ds}{s}||_{(\nu,\beta,\mu,k,\epsilon)}\\
\leq \frac{1}{\Gamma(1 + \frac{1}{k}) (2\pi)^{1/2}}
\frac{ C_{3} ||\varphi_{k}(\tau,m,\epsilon)||_{(\nu,\beta,\mu,k,\epsilon)}
||w_{1}(\tau,m)-w_{2}(\tau,m)||_{(\nu,\beta,\mu,k,\epsilon)} }{ C_{P}
(r_{Q,R_{D}})^{\frac{1}{(\delta_{D}-1)k}} }\\
\leq \frac{1}{\Gamma(1 + \frac{1}{k}) (2\pi)^{1/2}}
\frac{ C_{3} \zeta_{1} ||w_{1}(\tau,m)-w_{2}(\tau,m)||_{(\nu,\beta,\mu,k,\epsilon)} }{ C_{P}
(r_{Q,R_{D}})^{\frac{1}{(\delta_{D}-1)k}} } \label{fix_point_norm_estim_6_shrink}
\end{multline}
together with
\begin{multline}
||\frac{\epsilon^{-1}}{P_{m}(\tau) \Gamma(1 + \frac{1}{k})} \int_{0}^{\tau^k}(\tau^{k}-s)^{1/k}
\frac{1}{(2\pi)^{1/2}}(\int_{-\infty}^{+\infty}C_{0,0}(m-m_{1},\epsilon)R_{0}(im_{1})\\
\times (w_{1}(s^{1/k},m_{1}) - w_{2}(s^{1/k},m_{1})) dm_{1} ) \frac{ds}{s} ||_{(\nu,\beta,\mu,k,\epsilon)}\\
\leq \frac{1}{\Gamma(1 + \frac{1}{k}) (2\pi)^{1/2}}
\frac{ C_{4} \zeta_{0} ||w_{1}(\tau,m)-w_{2}(\tau,m)||_{(\nu,\beta,\mu,k,\epsilon)} }{ C_{P}
(r_{Q,R_{D}})^{\frac{1}{(\delta_{D}-1)k}} } \label{fix_point_norm_estim_7_shrink}
\end{multline}

Now, we take $\varpi$ and $r_{Q,R_{D}}$ such that
\begin{multline}
 \frac{1}{\Gamma(1 + \frac{1}{k}) (2\pi)^{1/2}}
\frac{ C_{3} 2 \varpi }{ C_{P}
(r_{Q,R_{D}})^{\frac{1}{(\delta_{D}-1)k}} }\\
+ \sum_{1 \leq p \leq \delta_{D}-1}
\frac{|A_{\delta_{D},p}|k^{p}C_{2.1}\epsilon_{0}}{\Gamma(\delta_{D}-p)C_{P}(r_{Q,R_{D}})^{\frac{1}{(\delta_{D}-1)k}}}\\
+ \sum_{1 \leq l \leq D-1}  \frac{ k^{\delta_l}C_{2.2}}{\Gamma(\frac{d_{l,k}}{k}) C_{P}(r_{Q,R_{D}})^{\frac{1}{(\delta_{D}-1)k}}}
\epsilon_{0}^{\Delta_{l}-d_{l}+\delta_{l} + k(\delta_{l} - \delta_{D}) + d_{l,k}}
\sup_{m \in \mathbb{R}} |\frac{R_{l}(im)}{R_{D}(im)}|\\
+ \sum_{1 \leq p \leq \delta_{l}-1} \frac{|A_{\delta_{l},p}| k^{p}C_{2.2}}{\Gamma( \frac{d_{l,k}}{k} + \delta_{l}-p)
C_{P}(r_{Q,R_{D}})^{\frac{1}{(\delta_{D}-1)k}}} \epsilon_{0}^{\Delta_{l} - d_{l} + \delta_{l} + k(\delta_{l}-\delta_{D}) + d_{l,k}}
\sup_{m \in \mathbb{R}} |\frac{R_{l}(im)}{R_{D}(im)}|\\
+ \frac{1}{\Gamma(1 + \frac{1}{k}) (2\pi)^{1/2}}
\frac{ C_{3} \zeta_{1} + C_{4} \zeta_{0} }{ C_{P}
(r_{Q,R_{D}})^{\frac{1}{(\delta_{D}-1)k}} } \leq \frac{1}{2} \label{fix_point_sum_shrink<1/2}
\end{multline}

Bearing in mind the estimates (\ref{fix_point_norm_estim_1_shrink}), (\ref{fix_point_norm_estim_2_shrink}),
(\ref{fix_point_norm_estim_3_shrink}), (\ref{fix_point_norm_estim_4_shrink}),
(\ref{fix_point_norm_estim_6_shrink}), (\ref{fix_point_norm_estim_7_shrink})
with the constraint (\ref{fix_point_sum_shrink<1/2}), one gets (\ref{H_shrink}).

Finally, we choose $\varpi$ and $r_{Q,R_{D}}$ such that both
(\ref{fix_point_sum<varpi}) and (\ref{fix_point_sum_shrink<1/2}) are satisfied. This yields our lemma.
\end{proof}
We consider the ball $\bar{B}(0,\varpi) \subset F_{(\nu,\beta,\mu,k,\epsilon)}^{d}$ constructed in Lemma 2 which is a complete
metric space for the norm $||.||_{(\nu,\beta,\mu,k,\epsilon)}$. From the lemma above, we get that $\mathcal{H}_{\epsilon}$ is a
contractive map from $\bar{B}(0,\varpi)$ into itself. Due to the classical contractive mapping theorem, we deduce that
the map $\mathcal{H}_{\epsilon}$ has a unique fixed point denoted by $\omega_{k}(\tau,m,\epsilon)$ (i.e
$\mathcal{H}_{\epsilon}(\omega_{k}(\tau,m,\epsilon))= \omega_{k}(\tau,m,\epsilon)$) in
$\bar{B}(0,\varpi)$, for all $\epsilon \in D(0,\epsilon_{0}) \setminus \{ 0 \}$. Moreover, the function
$\omega_{k}(\tau,m,\epsilon)$ depends holomorphically on $\epsilon$ in $D(0,\epsilon_{0}) \setminus \{ 0 \}$. By construction,
$\omega_{k}(\tau,m,\epsilon)$ defines a solution of the equation (\ref{k_Borel_equation}). This yields the proposition.
\end{proof}

In the next proposition, we construct analytic solutions of the equation (\ref{SCP}).
\begin{prop} Let the assumption (\ref{constraints_k_Borel_equation}) hold. We also choose the sectors $S_{d}$ and $S_{Q,R_{D}}$
in such a way that (\ref{root_cond_1}) and (\ref{root_cond_2}) hold. We take the radius $r_{Q,R_{D}}$ as prescribed in Proposition 9.
We also assume that the inequalities
(\ref{norm_F_varphi_k_psi_k_small}) hold for $\zeta_{0},\zeta_{1},\zeta_{2}$ constructed in Proposition 9. Notice that
the inequalities for $\zeta_{1},\zeta_{2}$ can be satisfied if
$\epsilon_{0}$ is small enough due to the estimates (\ref{norm_F_varphi_k_psi_k_epsilon_0})).

Let $S_{d,\theta,h'|\epsilon|}$ be a bounded sector with aperture $\pi/k < \theta < \pi/k + 2\delta$
(where $2\delta$ is the small aperture of the unbounded sector $S_{d}$), with direction $d$ and radius $h'|\epsilon|$ for some
$h'>0$ independent of $\epsilon$. We choose $0 < \beta' < \beta$.

Then, the equation 
(\ref{SCP}) with initial condition $U(0,m,\epsilon) \equiv 0$ has a solution $(T,m) \mapsto U(T,m,\epsilon)$ defined on
$S_{d,\theta,h'|\epsilon|} \times \mathbb{R}$ for some real number $h'>0$ for all
$\epsilon \in D(0,\epsilon_{0}) \setminus \{ 0 \}$. Let $\epsilon \in D(0,\epsilon_{0}) \setminus \{ 0 \}$, then
for each $T \in S_{d,\theta,h'|\epsilon|}$, the function
$m \mapsto U(T,m,\epsilon)$ belongs to the space $E_{(\beta',\mu)}$ and for each $m \in \mathbb{R}$, the function
$T \mapsto U(T,m,\epsilon)$ is bounded and holomorphic on $S_{d,\theta,h'|\epsilon|}$. Moreover, the function
$U(T,m,\epsilon)$ can be written as a Laplace transform of order $k$ in the direction $d$,
\begin{equation}
U(T,m,\epsilon) = k \int_{L_{\gamma}} \omega_{k}^{d}(u,m,\epsilon) e^{-(\frac{u}{T})^{k}} \frac{du}{u} \label{int_repres_U}
\end{equation} 
along a halfline $L_{\gamma}=\mathbb{R}_{+}e^{i\gamma} \in S_{d} \cup \{ 0 \}$
(the direction $\gamma$ may depend on $T$), where $\omega_{k}^{d}(\tau,m,\epsilon)$ defines a continuous function on
$(\bar{D}(0,\rho) \cup S_{d}) \times \mathbb{R} \times D(0,\epsilon_{0}) \setminus \{ 0 \}$ which is holomorphic with respect to
$(\tau,\epsilon)$ on $(\bar{D}(0,\rho) \cup S_{d}) \times D(0,\epsilon_{0}) \setminus \{ 0 \}$ and satisfies the estimates : there exists
a constant $\varpi_{d}$ (independent of $\epsilon$) such that
\begin{equation}
|\omega_{k}^{d}(\tau,m,\epsilon)| \leq \varpi_{d}(1+ |m|)^{-\mu} e^{-\beta|m|}
\frac{ |\frac{\tau}{\epsilon}|}{1 + |\frac{\tau}{\epsilon}|^{2k}} \exp( \nu |\frac{\tau}{\epsilon}|^{k}) \label{|omega_k_d|<} 
\end{equation}
for all $\tau \in D(0,\rho) \cup S_{d}$, all $m \in \mathbb{R}$, all $\epsilon \in D(0,\epsilon_{0}) \setminus \{ 0 \}$.
\end{prop}
\begin{proof} Taking into account the requirements stated above in Proposition 10, we get that all the assumptions of Proposition 9 are
fulfilled. Therefore, the formal $m_{k}-$Borel transform
$\omega_{k}(\tau,m,\epsilon) = \sum_{n \geq 1} U_{n}(m,\epsilon) \tau^{n}/\Gamma(n/k)$ of the formal series
$\hat{U}(T,m,\epsilon)$ constructed in Proposition 9 is convergent with respect to $\tau$ on $D(0,\rho)$ as series with coefficients
in the Banach space $E_{(\beta,\mu)}$. Moreover, this function
$\omega_{k}(\tau,m,\epsilon)$ can be extended as an analytic function with respect to $\tau$ on the sector $S_{d}$, denoted
$\omega_{k}^{d}(\tau,m,\epsilon)$, that belongs to the Banach space $F_{(\nu,\beta,\mu,k,\epsilon)}^{d}$ and satisfies the
bounds $||\omega_{k}^{d}(\tau,m,\epsilon)||_{(\nu,\beta,\mu,k,\epsilon)} \leq \varpi_{d}$ where $\varpi_{d}$ is a constant
independent of $\epsilon$ in $D(0,\epsilon_{0}) \setminus \{ 0 \}$. This means that (\ref{|omega_k_d|<}) must hold. As a result, we
get that the formal series $\hat{U}(T,m,\epsilon) \in TE_{(\beta,\mu)}[[T]]$ is $m_{k}-$summable in the direction $d$ (see Definition 3). 
By construction, its $m_{k}-$sum $U(T,m,\epsilon)$ in direction $d$ defines a holomorphic function on the sector
$S_{d,\theta,h'|\epsilon|}$ described above in Proposition 10 with values in $E_{(\beta,\mu)}$, for
all $\epsilon \in D(0,\epsilon_{0}) \setminus \{ 0 \}$.
On the other hand, the series $C_{0}(T,m,\epsilon), F(T,m,\epsilon) \in TE_{(\beta,\mu)}[[T]]$ are convergent. Therefore, these series
are $m_{k}$-summable in any direction $d$ and their $m_{k}-$sums satisfy
$$ \mathcal{L}_{m_k}^{d}(\varphi_{k}(\tau,m,\epsilon))(T) = C_{0}(T,m,\epsilon) \ \ , \ \
\mathcal{L}_{m_k}^{d}(\psi_{k}(\tau,m,\epsilon))(T) = F(T,m,\epsilon)
$$
for all $T \in D(0,T_{0}/2)$. Finally, using the properties for the sum, product and derivative of
$m_{k}-$sums described in (\ref{sum_prod_deriv_m_k_sum}), we deduce that the $m_{k}-$sum $U(T,m,\epsilon)$ in direction $d$
satisfies the equation (\ref{SCP_irregular}) as a function of $(T,m)$ on $S_{d,\theta,h'|\epsilon|} \times \mathbb{R}$, for all
$\epsilon \in D(0,\epsilon_{0}) \setminus \{ 0 \}$, since the formal series $\hat{U}(T,m,\epsilon)$ satisfies the equation
(\ref{SCP_irregular}). As a result, the function $U(T,m,\epsilon)$ also satisfies
the equation (\ref{SCP}) as a function of $(T,m)$ on $S_{d,\theta,h'|\epsilon|} \times \mathbb{R}$, for all
$\epsilon \in D(0,\epsilon_{0}) \setminus \{ 0 \}$. 
\end{proof}

\section{Analytic solutions of a nonlinear initial value Cauchy problem with complex parameter}

Let $k \geq 1$ and $D \geq 2$ be integers. For $1 \leq l \leq D$, let
$d_{l},\delta_{l},\Delta_{l} \geq 0$ be nonnegative integers.
We assume that 
\begin{equation}
1 = \delta_{1} \ \ , \ \ \delta_{l} < \delta_{l+1}, \label{assum_delta_l}
\end{equation}
for all $1 \leq l \leq D-1$. We make also the assumption that
\begin{equation}
d_{D} = (\delta_{D}-1)(k+1) \ \ , \ \ d_{l} > (\delta_{l}-1)(k+1) \ \ , \ \ \Delta_{D} = d_{D} - \delta_{D} + 1 \label{assum_d_delta_Delta}
\end{equation}
for all $1 \leq l \leq D-1$. Let $Q(X),Q_{1}(X),Q_{2}(X),R_{l}(X) \in \mathbb{C}[X]$, $0 \leq l \leq D$, be polynomials such that
\begin{multline}
\mathrm{deg}(Q) \geq \mathrm{deg}(R_{D}) \geq \mathrm{deg}(R_{l}) \ \ , \ \
\mathrm{deg}(R_{D}) \geq \mathrm{deg}(Q_{1}) \ \ , \ \ \mathrm{deg}(R_{D}) \geq \mathrm{deg}(Q_{2}),\\
Q(im) \neq 0 \ \ , \ \ R_{D}(im) \neq 0 \label{assum_deg_Q_R}
\end{multline}
for all $m \in \mathbb{R}$, all $0 \leq l \leq D-1$.

We consider the following nonlinear initial value problem
\begin{multline}
Q(\partial_{z})(\partial_{t}u(t,z,\epsilon)) = (Q_{1}(\partial_{z})u(t,z,\epsilon))(Q_{2}(\partial_{z})u(t,z,\epsilon))
+ \sum_{l=1}^{D} \epsilon^{\Delta_{l}}t^{d_l}\partial_{t}^{\delta_l}R_{l}(\partial_{z})u(t,z,\epsilon)\\
+ c_{0}(t,z,\epsilon)R_{0}(\partial_{z})u(t,z,\epsilon) + f(t,z,\epsilon) \label{ICP_main}
\end{multline}
for given initial data $u(0,z,\epsilon) \equiv 0$.

The coefficient $c_{0}(t,z,\epsilon)$ and the forcing term $f(t,z,\epsilon)$ are constructed as follows. We consider sequences of functions
$m \mapsto C_{0,n}(m,\epsilon)$, for $n \geq 0$ and $m \mapsto F_{n}(m,\epsilon)$, for $n \geq 1$, that belong to the Banach space
$E_{(\beta,\mu)}$ for some $\beta > 0$, $\mu > \max( \mathrm{deg}(Q_{1})+1, \mathrm{deg}(Q_{2})+1)$ and which
depend holomorphically on $\epsilon \in D(0,\epsilon_{0})$. We assume that there exist constants $K_{0},T_{0}>0$
such that (\ref{norm_beta_mu_F_n}) hold for all $n \geq 1$, for all $\epsilon \in D(0,\epsilon_{0})$.
We deduce that the functions
$$ \mathbf{C}_{0}(T,z,\epsilon) = \sum_{n \geq 0} \mathcal{F}^{-1}(m \mapsto C_{0,n}(m,\epsilon))(z) T^{n} \ \ , \ \
\mathbf{F}(T,z,\epsilon) = \sum_{n \geq 1} \mathcal{F}^{-1}(m \mapsto F_{n}(m,\epsilon))(z) T^{n} $$
represent bounded holomorphic functions on $D(0,T_{0}/2) \times H_{\beta'} \times D(0,\epsilon_{0})$ for any
$0 < \beta' < \beta$ (where
$\mathcal{F}^{-1}$ denotes the inverse Fourier transform defined in Proposition 7). We define
the coefficient $c_{0}(t,z,\epsilon)$ and the forcing term $f(t,z,\epsilon)$ as
\begin{equation}
c_{0}(t,z,\epsilon) = \mathbf{C}_{0}(\epsilon t,z,\epsilon) \ \ , \ \ f(t,z,\epsilon) = \mathbf{F}(\epsilon t , z,\epsilon).
\label{defin_c_0_f}
\end{equation}
The functions $c_{0}$ and $f$ are holomorphic and bounded on $D(0,r) \times H_{\beta'} \times D(0,\epsilon_{0})$ where
$r \epsilon_{0} < T_{0}/2$.

We make the additional assumption that there exists an unbounded sector
$$ S_{Q,R_{D}} = \{ z \in \mathbb{C} / |z| \geq r_{Q,R_{D}} \ \ , \ \ |\mathrm{arg}(z) - d_{Q,R_{D}}| \leq \eta_{Q,R_{D}} \} $$
with direction $d_{Q,R_{D}} \in \mathbb{R}$, aperture $\eta_{Q,R_{D}}>0$ for some radius $r_{Q,R_{D}}>0$ such that
\begin{equation}
\frac{Q(im)}{R_{D}(im)} \in S_{Q,R_{D}} \label{assum_Q_R_D}
\end{equation} 
for all $m \in \mathbb{R}$.

\begin{defin} Let $\varsigma \geq 2$ be an integer. For all $0 \leq p \leq \varsigma-1$, we consider open sectors
$\mathcal{E}_{p}$ centered at $0$, with radius $\epsilon_{0}$ and opening
$\frac{\pi}{k}+\kappa_{p}$, with $\kappa_{p}>0$ small enough such that
$\mathcal{E}_{p} \cap \mathcal{E}_{p+1} \neq \emptyset$, for all
$0 \leq p \leq \varsigma-1$ (with the convention that $\mathcal{E}_{\varsigma} = \mathcal{E}_{0})$. Moreover, we assume that
the intersection of any three different elements in $(\mathcal{E}_{p})_{0 \leq p \leq \varsigma}$ is empty and that
$\cup_{p=0}^{\varsigma - 1} \mathcal{E}_{p} = \mathcal{U} \setminus \{ 0 \}$,
where $\mathcal{U}$ is some neighborhood of 0 in $\mathbb{C}$. Such a set of sectors
$\{ \mathcal{E}_{p} \}_{0 \leq p \leq \varsigma - 1}$ is called a good covering in $\mathbb{C}^{\ast}$.
\end{defin}

\begin{defin} Let $\{ \mathcal{E}_{p} \}_{0 \leq p \leq \varsigma - 1}$ be a good covering in $\mathbb{C}^{\ast}$. Let
$\mathcal{T}$ be an open bounded sector centered at 0 with radius $r_{\mathcal{T}}$ and consider a family of open sectors
$$ S_{\mathfrak{d}_{p},\theta,\epsilon_{0}r_{\mathcal{T}}} =
\{ T \in \mathbb{C}^{\ast} / |T| < \epsilon_{0}r_{\mathcal{T}} \ \ , \ \ |\mathfrak{d}_{p} - \mathrm{arg}(T)| < \theta/2 \} $$
with aperture $\theta > \pi/k$ and where $\mathfrak{d}_{p} \in \mathbb{R}$, for all $0 \leq p \leq \varsigma-1$, are directions
which satisfy
the following constraints: Let $q_{l}(m)$ be the roots of the polynomials (\ref{factor_P_m}) defined by (\ref{defin_roots}) and
$S_{\mathfrak{d}_p}$, $0 \leq p \leq \varsigma -1$ be unbounded sectors centered at 0 with directions
$\mathfrak{d}_{p}$ and with small aperture. We assume
that\\
1) There exists a constant $M_{1}>0$ such that
\begin{equation}
|\tau - q_{l}(m)| \geq M_{1}(1 + |\tau|) \label{root_cond_1_in_defin}
\end{equation}
for all $0 \leq l \leq (\delta_{D}-1)k-1$, all $m \in \mathbb{R}$, all $\tau \in S_{\mathfrak{d}_p} \cup \bar{D}(0,\rho)$, for all
$0 \leq p \leq \varsigma-1$.\\
2) There exists a constant $M_{2}>0$ such that
\begin{equation}
|\tau - q_{l_0}(m)| \geq M_{2}|q_{l_0}(m)| \label{root_cond_2_in_defin}
\end{equation}
for some $l_{0} \in \{0,\ldots,(\delta_{D}-1)k-1 \}$, all $m \in \mathbb{R}$, all $\tau \in S_{\mathfrak{d}_p} \cup \bar{D}(0,\rho)$, for
all $0 \leq p \leq \varsigma - 1$.\\
3) For all $0 \leq p \leq \varsigma - 1$, for all $t \in \mathcal{T}$, all $\epsilon \in \mathcal{E}_{p}$, we have that
$\epsilon t \in S_{\mathfrak{d}_{p},\theta,\epsilon_{0}r_{\mathcal{T}}}$.\medskip

\noindent We say that the family
$\{ (S_{\mathfrak{d}_{p},\theta,\epsilon_{0}r_{\mathcal{T}}})_{0 \leq p \leq \varsigma-1},\mathcal{T} \}$
is associated to the good covering $\{ \mathcal{E}_{p} \}_{0 \leq p \leq \varsigma - 1}$.
\end{defin}

In the next first main result, we construct a family of actual holomorphic solutions to the equation (\ref{ICP_main}) for given
initial data at $t=0$ being identically equal to zero, defined on the sectors $\mathcal{E}_{p}$ with respect to the
complex parameter $\epsilon$. We can also control the difference between any two neighboring solutions
on the intersection of sectors $\mathcal{E}_{p} \cap \mathcal{E}_{p+1}$ and show that it is exponentially
flat of order at most $k$.

\begin{theo} We consider the equation (\ref{ICP_main}) and we assume that the constraints
(\ref{assum_delta_l}), (\ref{assum_d_delta_Delta}), (\ref{assum_deg_Q_R}) and (\ref{assum_Q_R_D}) hold. We also make the
additional assumption that
\begin{equation}
\delta_{D} \geq \delta_{l} + \frac{2}{k} \ \ , \ \ \Delta_{l}-d_{l}+\delta_{l} + k(\delta_{l} - \delta_{D}) + d_{l,k} \geq 0,
\label{constraints_k_Borel_equation_for_u_p}
\end{equation}
hold for all $1 \leq l \leq D-1$. Let the coefficient $c_{0}(t,z,\epsilon)$ and forcing term $f(t,z,\epsilon)$ be constructed as in
(\ref{defin_c_0_f}). Let a good covering $\{ \mathcal{E}_{p} \}_{0 \leq p \leq \varsigma - 1}$ in $\mathbb{C}^{\ast}$ be given, for
which a family of sectors $\{ (S_{\mathfrak{d}_{p},\theta,\epsilon_{0}r_{\mathcal{T}}})_{0 \leq p \leq \varsigma-1},\mathcal{T} \}$
associated to this good covering can be considered.

Then, there exist a radius $r_{Q,R_{D}}>0$ large enough, $\epsilon_{0}>0$ small enough and
a constant $\zeta_{0}>0$ small enough such that if
$$ ||C_{0,0}(m,\epsilon)||_{(\beta,\mu)} < \zeta_{0} $$
for all $\epsilon \in D(0,\epsilon_{0}) \setminus \{ 0 \}$, then for every $0 \leq p \leq \varsigma-1$,
one can construct a solution $u_{p}(t,z,\epsilon)$ of the equation (\ref{ICP_main}) with $u_{p}(0,z,\epsilon) \equiv 0$ which
defines a bounded holomorphic function on the domain $(\mathcal{T} \cap D(0,h')) \times H_{\beta'} \times
\mathcal{E}_{i}$ for any given $0< \beta'< \beta$ and for some $h'>0$. Moreover, there exist constants $0 < h'' \leq h'$,
$K_{p},M_{p}>0$ (independent of $\epsilon$)
such that
\begin{equation}
\sup_{t \in \mathcal{T} \cap D(0,h''), z \in H_{\beta'}}
|u_{p+1}(t,z,\epsilon) - u_{p}(t,z,\epsilon)| \leq K_{p}e^{-\frac{M_p}{|\epsilon|^{k}}}
\label{exp_small_difference_u_p}
\end{equation}
for all $\epsilon \in \mathcal{E}_{p+1} \cap \mathcal{E}_{p}$, for all $0 \leq p \leq \varsigma-1$ (where by convention
$u_{\varsigma}=u_{0}$).
\end{theo}
\begin{proof} Using Proposition 10, one can choose $r_{Q,R_{D}}>0$ large enough, $\epsilon_{0}>0$ small enough and
$\zeta_{0}>0$ small enough such that
$$ ||C_{0,0}(m,\epsilon)||_{(\beta,\mu)} \leq \zeta_{0} $$
for all $\epsilon \in D(0,\epsilon_{0}) \setminus \{ 0 \}$ such that
for each direction $\mathfrak{d}_{p}$ with $0 \leq p \leq \varsigma - 1$, one can construct a
function $U^{\mathfrak{d}_p}(T,m,\epsilon)$ which satisfies
$U^{\mathfrak{d}_p}(0,m,\epsilon) \equiv 0$ and solves the equation
\begin{multline}
Q(im)(\partial_{T}U(T,m,\epsilon) ) =
\epsilon^{-1}\frac{1}{(2\pi)^{1/2}}\int_{-\infty}^{+\infty}Q_{1}(i(m-m_{1}))U(T,m-m_{1},\epsilon)\\
\times Q_{2}(im_{1})U(T,m_{1},\epsilon) dm_{1} \\
+ \sum_{l=1}^{D} R_{l}(im) \epsilon^{\Delta_{l} - d_{l} + \delta_{l} - 1} T^{d_{l}} \partial_{T}^{\delta_l}U(T,m,\epsilon)\\
+ \epsilon^{-1}\frac{1}{(2\pi)^{1/2}}\int_{-\infty}^{+\infty}C_{0}(T,m-m_{1},\epsilon)U(T,m_{1},\epsilon) dm_{1}\\
+ \epsilon^{-1}\frac{1}{(2\pi)^{1/2}}\int_{-\infty}^{+\infty}C_{0,0}(m-m_{1},\epsilon)U(T,m_{1},\epsilon) dm_{1} 
+ \epsilon^{-1}F(T,m,\epsilon)
\label{SCP_in_main_result}
\end{multline}
where
$$ C_{0}(T,m,\epsilon) = \sum_{n \geq 1} C_{0,n}(m,\epsilon) T^{n} \ \ , \ \ F(T,m,\epsilon) = \sum_{n \geq 1} F_{n}(m,\epsilon) T^{n} $$
are convergent series in $D(0,T_{0}/2)$ with values in $E_{(\beta,\mu)}$, for all $\epsilon \in D(0,\epsilon_{0}) \setminus \{ 0 \}$. The
function $(T,m) \mapsto U^{\mathfrak{d}_p}(T,m,\epsilon)$ is well defined on the domain
$S_{\mathfrak{d}_{p},\theta,h'|\epsilon|} \times \mathbb{R}$
where $h'>0$ is some real number, for all $\epsilon \in D(0,\epsilon_{0}) \setminus \{ 0 \}$. Moreover,
$U^{\mathfrak{d}_p}(T,m,\epsilon)$ can be written as a Laplace transform of order $k$ in the direction $\mathfrak{d}_p$,
\begin{equation}
U^{\mathfrak{d}_p}(T,m,\epsilon) = k \int_{L_{\gamma_{p}}} \omega_{k}^{\mathfrak{d}_p}(u,m,\epsilon)
e^{-(\frac{u}{T})^{k}} \frac{du}{u} \label{int_repres_U_d_p}
\end{equation} 
along a halfline $L_{\gamma_{p}}=\mathbb{R}_{+}e^{i\gamma_{p}} \in S_{\mathfrak{d}_p} \cup \{ 0 \}$
(the direction $\gamma_{p}$ may depend on $T$), where $\omega_{k}^{\mathfrak{d}_p}(\tau,m,\epsilon)$ defines a continuous function on
$(\bar{D}(0,\rho) \cup S_{d_p}) \times \mathbb{R} \times D(0,\epsilon_{0}) \setminus \{ 0 \}$ which is holomorphic with respect to
$(\tau,\epsilon)$ on $(\bar{D}(0,\rho) \cup S_{\mathfrak{d}_p}) \times D(0,\epsilon_{0}) \setminus \{ 0 \}$ for any $m \in \mathbb{R}$
and satisfies the estimates : there exists
a constant $\varpi_{\mathfrak{d}_p}$ (independent of $\epsilon$) such that
\begin{equation}
|\omega_{k}^{\mathfrak{d}_p}(\tau,m,\epsilon)| \leq \varpi_{\mathfrak{d}_p}(1+ |m|)^{-\mu} e^{-\beta|m|}
\frac{ |\frac{\tau}{\epsilon}|}{1 + |\frac{\tau}{\epsilon}|^{2k}} \exp( \nu |\frac{\tau}{\epsilon}|^{k}) \label{|omega_k_d_p|<} 
\end{equation}
for all $\tau \in D(0,\rho) \cup S_{\mathfrak{d}_{p}}$, all $m \in \mathbb{R}$, all
$\epsilon \in D(0,\epsilon_{0}) \setminus \{ 0 \}$. It is worth
noticing that all the functions $\tau \mapsto \omega_{k}^{\mathfrak{d}_p}(\tau,m,\epsilon)$ are analytic continuation on the sectors
$S_{\mathfrak{d}_p}$ of a common function denoted by
$$ \omega_{k}(\tau,m,\epsilon) = \sum_{n \geq 1} U_{n}(m,\epsilon) \frac{\tau^n}{\Gamma(\frac{n}{k})} $$
which is a convergent series on $D(0,\rho)$ with coefficients in $E_{(\beta,\mu)}$ and
where $U_{n}(m,\epsilon) \in E_{(\beta,\mu)}$ are the coefficients of the formal series
$\hat{U}(T,m,\epsilon) = \sum_{n \geq 1} U_{n}(m,\epsilon)T^{n}$ solution of the equation (\ref{SCP_in_main_result}), for
all $\epsilon \in D(0,\epsilon_{0}) \setminus \{ 0 \}$. Using the
estimates (\ref{|omega_k_d_p|<}), we get that the function 
$$ (T,z) \mapsto \mathbf{U}^{\mathfrak{d}_p}(T,z,\epsilon) = \mathcal{F}^{-1}(m \mapsto U^{\mathfrak{d}_p}(T,m,\epsilon))(z) $$
defines a bounded holomorphic function on $S_{\mathfrak{d}_{p},\theta,h'|\epsilon|} \times H_{\beta'}$, for all
$\epsilon \in D(0,\epsilon_{0}) \setminus \{ 0 \}$ and any $0 < \beta' < \beta$. For all $0 \leq p \leq \varsigma - 1$, we define
\begin{equation}
u_{p}(t,z,\epsilon) = \mathbf{U}^{\mathfrak{d}_p}(\epsilon t,z,\epsilon) = \frac{k}{(2\pi)^{1/2}}\int_{-\infty}^{+\infty}
\int_{L_{\gamma_{p}}}
\omega_{k}^{\mathfrak{d}_p}(u,m,\epsilon) e^{-(\frac{u}{\epsilon t})^{k}} e^{izm} \frac{du}{u} dm.
\end{equation}
By construction (see 3) in Definition 5), the function $u_{p}(t,z,\epsilon)$ defines a bounded holomorphic function on
the domain $(\mathcal{T} \cap D(0,h')) \times H_{\beta'} \times \mathcal{E}_{p}$. Moreover, we have
$u_{p}(0,z,\epsilon) \equiv 0$ and using the properties of the Fourier inverse transform from Proposition 7, we deduce that
$u_{p}(t,z,\epsilon)$ solves the main equation (\ref{ICP_main}) on
$(\mathcal{T} \cap D(0,h')) \times H_{\beta'} \times \mathcal{E}_{p}$.

Now, we proceed to the proof of the estimates (\ref{exp_small_difference_u_p}). Let $p \in \{ 0,\ldots,\varsigma - 1 \}$. Using the fact that
the function $u \mapsto \omega_{k}(u,m,\epsilon) \exp( -(\frac{u}{\epsilon t})^{k} )/u$ is holomorphic on $D(0,\rho)$ for all
$(m,\epsilon) \in \mathbb{R} \times (D(0,\epsilon_{0}) \setminus \{ 0 \})$, its integral along the union of a segment starting from
0 to $(\rho/2)e^{i\gamma_{p+1}}$, an arc of circle with radius $\rho/2$ which connects
$(\rho/2)e^{i\gamma_{p+1}}$ and $(\rho/2)e^{i\gamma_{p}}$ and a segment starting from
$(\rho/2)e^{i\gamma_{p}}$ to 0, is equal to zero. Therefore, we can write the difference $u_{p+1} - u_{p}$ as a sum of three
integrals,
\begin{multline}
u_{p+1}(t,z,\epsilon) - u_{p}(t,z,\epsilon) = \frac{k}{(2\pi)^{1/2}}\int_{-\infty}^{+\infty}
\int_{L_{\rho/2,\gamma_{p+1}}}
\omega_{k}^{\mathfrak{d}_{p+1}}(u,m,\epsilon) e^{-(\frac{u}{\epsilon t})^{k}} e^{izm} \frac{du}{u} dm\\ -
\frac{k}{(2\pi)^{1/2}}\int_{-\infty}^{+\infty}
\int_{L_{\rho/2,\gamma_{p}}}
\omega_{k}^{\mathfrak{d}_p}(u,m,\epsilon) e^{-(\frac{u}{\epsilon t})^{k}} e^{izm} \frac{du}{u} dm\\
+ \frac{k}{(2\pi)^{1/2}}\int_{-\infty}^{+\infty}
\int_{C_{\rho/2,\gamma_{p},\gamma_{p+1}}}
\omega_{k}(u,m,\epsilon) e^{-(\frac{u}{\epsilon t})^{k}} e^{izm} \frac{du}{u} dm \label{difference_u_p_decomposition}
\end{multline}
where $L_{\rho/2,\gamma_{p+1}} = [\rho/2,+\infty)e^{i\gamma_{p+1}}$,
$L_{\rho/2,\gamma_{p}} = [\rho/2,+\infty)e^{i\gamma_{p}}$ and
$C_{\rho/2,\gamma_{p},\gamma_{p+1}}$ is an arc of circle with radius connecting
$(\rho/2)e^{i\gamma_{p}}$ and $(\rho/2)e^{i\gamma_{p+1}}$ with a well chosen orientation.\medskip

We give estimates for the quantity
$$ I_{1} = \left| \frac{k}{(2\pi)^{1/2}}\int_{-\infty}^{+\infty}
\int_{L_{\rho/2,\gamma_{p+1}}}
\omega_{k}^{\mathfrak{d}_{p+1}}(u,m,\epsilon) e^{-(\frac{u}{\epsilon t})^{k}} e^{izm} \frac{du}{u} dm \right|.
$$
By construction, the direction $\gamma_{p+1}$ (which depends on $\epsilon t$) is chosen in such a way that
$\cos( k( \gamma_{p+1} - \mathrm{arg}(\epsilon t) )) \geq \delta_{1}$, for all
$\epsilon \in \mathcal{E}_{p} \cap \mathcal{E}_{p+1}$, all $t \in \mathcal{T} \cap D(0,h')$, for some fixed $\delta_{1} > 0$.
From the estimates (\ref{|omega_k_d_p|<}), we get that
\begin{multline}
I_{1} \leq \frac{k}{(2\pi)^{1/2}} \int_{-\infty}^{+\infty} \int_{\rho/2}^{+\infty}
\varpi_{\mathfrak{d}_{p+1}}(1+|m|)^{-\mu} e^{-\beta|m|}
\frac{ \frac{r}{|\epsilon|}}{1 + (\frac{r}{|\epsilon|})^{2k} } \\
\times \exp( \nu (\frac{r}{|\epsilon|})^{k} )
\exp(-\frac{\cos(k(\gamma_{p+1} - \mathrm{arg}(\epsilon t)))}{|\epsilon t|^{k}}r^{k}) e^{-m\mathrm{Im}(z)} \frac{dr}{r} dm\\
\leq \frac{k\varpi_{\mathfrak{d}_{p+1}}}{(2\pi)^{1/2}} \int_{-\infty}^{+\infty} e^{-(\beta - \beta')|m|} dm
\int_{\rho/2}^{+\infty}\frac{1}{|\epsilon|} \exp( -(\frac{\delta_{1}}{|t|^{k}} - \nu)(\frac{r}{|\epsilon|})^{k} ) dr\\
\leq  \frac{2k\varpi_{\mathfrak{d}_{p+1}}}{(2\pi)^{1/2}} \int_{0}^{+\infty} e^{-(\beta - \beta')m} dm
\int_{\rho/2}^{+\infty} \frac{|\epsilon|^{k-1}}{(\frac{\delta_1}{|t|^{k}} - \nu)k(\frac{\rho}{2})^{k-1}}
\times \frac{ (\frac{\delta_1}{|t|^{k}} - \nu)k r^{k-1} }{|\epsilon|^k}
\exp( -(\frac{\delta_{1}}{|t|^{k}} - \nu)(\frac{r}{|\epsilon|})^{k} ) dr\\
\leq
\frac{2k\varpi_{\mathfrak{d}_{p+1}}}{(2\pi)^{1/2}} \frac{|\epsilon|^{k-1}}{(\beta - \beta')
(\frac{\delta_{1}}{|t|^{k}} - \nu)k(\frac{\rho}{2})^{k-1}} \exp( -(\frac{\delta_1}{|t|^{k}} - \nu) \frac{(\rho/2)^k}{|\epsilon|^k} )\\
\leq \frac{2k\varpi_{\mathfrak{d}_{p+1}}}{(2\pi)^{1/2}} \frac{|\epsilon|^{k-1}}{(\beta - \beta')
\delta_{2}k(\frac{\rho}{2})^{k-1}} \exp( -\delta_{2} \frac{(\rho/2)^k}{|\epsilon|^k} ) \label{I_1_exp_small_order_k}
\end{multline}
for all $t \in \mathcal{T} \cap D(0,h')$ and $|\mathrm{Im}(z)| \leq \beta'$ with
$|t| < (\frac{\delta_{1}}{\delta_{2} + \nu})^{1/k}$, for some $\delta_{2}>0$, for all
$\epsilon \in \mathcal{E}_{p} \cap \mathcal{E}_{p+1}$.\medskip

In the same way, we also give estimates for the integral
$$ I_{2} = \left| \frac{k}{(2\pi)^{1/2}}\int_{-\infty}^{+\infty}
\int_{L_{\rho/2,\gamma_{p}}}
\omega_{k}^{\mathfrak{d}_{p}}(u,m,\epsilon) e^{-(\frac{u}{\epsilon t})^{k}} e^{izm} \frac{du}{u} dm \right|.
$$
Namely, the direction $\gamma_{p}$ (which depends on $\epsilon t$) is chosen in such a way that
$\cos( k( \gamma_{p} - \mathrm{arg}(\epsilon t) )) \geq \delta_{1}$, for all
$\epsilon \in \mathcal{E}_{p} \cap \mathcal{E}_{p+1}$, all $t \in \mathcal{T} \cap D(0,h')$, for some fixed $\delta_{1} > 0$.
Again from the estimates (\ref{|omega_k_d_p|<}) and following the same steps as in (\ref{I_1_exp_small_order_k}), we get that
\begin{equation}
I_{2} \leq \frac{2k\varpi_{\mathfrak{d}_{p}}}{(2\pi)^{1/2}} \frac{|\epsilon|^{k-1}}{(\beta - \beta')
\delta_{2}k(\frac{\rho}{2})^{k-1}} \exp( -\delta_{2} \frac{(\rho/2)^k}{|\epsilon|^k} ) \label{I_2_exp_small_order_k}
\end{equation}
for all $t \in \mathcal{T} \cap D(0,h')$ and $|\mathrm{Im}(z)| \leq \beta'$ with
$|t| < (\frac{\delta_{1}}{\delta_{2} + \nu})^{1/k}$, for some $\delta_{2}>0$, for all
$\epsilon \in \mathcal{E}_{p} \cap \mathcal{E}_{p+1}$.\medskip

Finally, we give upper bound estimates for the integral
$$
I_{3} = \left| \frac{k}{(2\pi)^{1/2}}\int_{-\infty}^{+\infty}
\int_{C_{\rho/2,\gamma_{p},\gamma_{p+1}}}
\omega_{k}(u,m,\epsilon) e^{-(\frac{u}{\epsilon t})^{k}} e^{izm} \frac{du}{u} dm \right|.
$$
By construction, the arc of circle $C_{\rho/2,\gamma_{p},\gamma_{p+1}}$ is chosen in such a way that
$\cos(k(\theta - \mathrm{arg}(\epsilon t))) \geq \delta_{1}$, for all $\theta \in [\gamma_{p},\gamma_{p+1}]$ (if
$\gamma_{p} < \gamma_{p+1}$), $\theta \in [\gamma_{p+1},\gamma_{p}]$ (if
$\gamma_{p+1} < \gamma_{p}$), for all $t \in \mathcal{T}$, all $\epsilon \in \mathcal{E}_{p} \cap \mathcal{E}_{p+1}$, for some
fixed $\delta_{1}>0$. Bearing in mind (\ref{|omega_k_d_p|<}) and (\ref{x_m_exp_x<}), we get that
\begin{multline}
I_{3} \leq \frac{k}{(2\pi)^{1/2}} \int_{-\infty}^{+\infty}  \left| \int_{\gamma_{p}}^{\gamma_{p+1}} \right.
\max_{0 \leq p \leq \varsigma - 1} \varpi_{\mathfrak{d}_p} (1+|m|)^{-\mu} e^{-\beta|m|}
\frac{ \frac{\rho/2}{|\epsilon|}}{1 + (\frac{\rho/2}{|\epsilon|})^{2k} } \\
\times \exp( \nu (\frac{\rho/2}{|\epsilon|})^{k} )
\exp(-\frac{\cos(k(\theta - \mathrm{arg}(\epsilon t)))}{|\epsilon t|^{k}}(\frac{\rho}{2})^{k})
\left. e^{-m\mathrm{Im}(z)} d\theta \right| dm\\
\leq \frac{k( \max_{0 \leq p \leq \varsigma - 1} \varpi_{\mathfrak{d}_p})}{(2\pi)^{1/2}} \int_{-\infty}^{+\infty}
e^{-(\beta - \beta')|m|} dm \times
|\gamma_{p} - \gamma_{p+1}| \frac{\rho/2}{|\epsilon|}
\exp( -\frac{( \frac{\delta_1}{|t|^k} - \nu)}{2} (\frac{\rho/2}{|\epsilon|})^{k}) \\
 \times \exp( -\frac{( \frac{\delta_1}{|t|^k} - \nu)}{2} (\frac{\rho/2}{|\epsilon|})^{k})\\
\leq \frac{2k (\max_{0 \leq p \leq \varsigma - 1} \varpi_{\mathfrak{d}_p})|\gamma_{p} - \gamma_{p+1}|}{(2\pi)^{1/2}(\beta - \beta')}
\sup_{x \geq 0} x^{1/k}e^{-(\frac{\delta_1}{|t|^k} - \nu)x} \times
\exp( -\frac{( \frac{\delta_1}{|t|^k} - \nu)}{2} (\frac{\rho/2}{|\epsilon|})^{k})\\
\leq \frac{2k (\max_{0 \leq p \leq \varsigma-1}
\varpi_{\mathfrak{d}_p})|\gamma_{p} - \gamma_{p+1}|}{(2\pi)^{1/2}(\beta - \beta')} (\frac{1/k}{\delta_2})^{1/k}
e^{-1/k} \exp( -\frac{\delta_{2}}{2} (\frac{\rho/2}{|\epsilon|})^{k}) \label{I_3_exp_small_order_k}
\end{multline}
for all $t \in \mathcal{T} \cap D(0,h')$ and $|\mathrm{Im}(z)| \leq \beta'$ with
$|t| < (\frac{\delta_{1}}{\delta_{2} + \nu})^{1/k}$, for some $\delta_{2}>0$, for all
$\epsilon \in \mathcal{E}_{p} \cap \mathcal{E}_{p+1}$.\medskip

Finally, gathering the three above inequalities (\ref{I_1_exp_small_order_k}), (\ref{I_2_exp_small_order_k}) and
(\ref{I_3_exp_small_order_k}), we deduce from the decomposition (\ref{difference_u_p_decomposition}) that
\begin{multline*}
|u_{p+1}(t,z,\epsilon) - u_{p}(t,z,\epsilon)| \leq
 \frac{2k(\varpi_{\mathfrak{d}_{p}} + \varpi_{\mathfrak{d}_{p+1}})}{(2\pi)^{1/2}} \frac{|\epsilon|^{k-1}}{(\beta - \beta')
\delta_{2}k(\frac{\rho}{2})^{k-1}} \exp( -\delta_{2} \frac{(\rho/2)^k}{|\epsilon|^k} )\\
+  \frac{2k (\max_{0 \leq p \leq \varsigma-1}
\varpi_{\mathfrak{d}_p})|\gamma_{p} - \gamma_{p+1}|}{(2\pi)^{1/2}(\beta - \beta')} (\frac{1/k}{\delta_2})^{1/k}
e^{-1/k} \exp( -\frac{\delta_{2}}{2} (\frac{\rho/2}{|\epsilon|})^{k})
\end{multline*}
for all $t \in \mathcal{T} \cap D(0,h')$ and $|\mathrm{Im}(z)| \leq \beta'$ with
$|t| < (\frac{\delta_{1}}{\delta_{2} + \nu})^{1/k}$, for some $\delta_{2}>0$, for all
$\epsilon \in \mathcal{E}_{p} \cap \mathcal{E}_{p+1}$. Therefore, the inequality
(\ref{exp_small_difference_u_p}) holds.
\end{proof}

\section{Existence of $k-$summable formal series in the complex parameter of the initial value problem}

\subsection{$k-$Summable formal series and Ramis-Sibuya Theorem}

We recall the definition of $k-$Borel summability of formal series with coefficients in a Banach space,
see \cite{ba}.

\begin{defin} Let $k \geq 1$ be an integer. A formal series
$$\hat{X}(\epsilon) = \sum_{j=0}^{\infty}  \frac{ a_{j} }{ j! } \epsilon^{j} \in \mathbb{F}[[\epsilon]]$$
with coefficients in a Banach space $( \mathbb{F}, ||.||_{\mathbb{F}} )$ is said to be $k-$summable
with respect to $\epsilon$ in the direction $d \in \mathbb{R}$ if \medskip

{\bf i)} there exists $\rho \in \mathbb{R}_{+}$ such that the following formal series, called formal
Borel transform of $\hat{X}$ of order $k$ 
$$ \mathcal{B}_{k}(\hat{X})(\tau) = \sum_{j=0}^{\infty} \frac{ a_{j} \tau^{j}  }{ j!\Gamma(1 + \frac{j}{k}) } \in \mathbb{F}[[\tau]],$$
is absolutely convergent for $|\tau| < \rho$, \medskip

{\bf ii)} there exists $\delta > 0$ such that the series $\mathcal{B}_{k}(\hat{X})(\tau)$ can be analytically continued with
respect to $\tau$ in a sector
$S_{d,\delta} = \{ \tau \in \mathbb{C}^{\ast} : |d - \mathrm{arg}(\tau) | < \delta \} $. Moreover, there exist $C >0$, and $K >0$
such that
$$ ||\mathcal{B}(\hat{X})(\tau)||_{\mathbb{F}}
\leq C e^{ K|\tau|^{k}} $$
for all $\tau \in S_{d, \delta}$.
\end{defin}
If this is so, the vector valued Laplace transform of order $k$ of $\mathcal{B}_{k}(\hat{X})(\tau)$ in the direction $d$ is defined by
$$ \mathcal{L}^{d}_{k}(\mathcal{B}_{k}(\hat{X}))(\epsilon) = \epsilon^{-k} \int_{L_{\gamma}}
\mathcal{B}_{k}(\hat{X})(u) e^{ - ( u/\epsilon )^{k} } ku^{k-1}du,$$
along a half-line $L_{\gamma} = \mathbb{R}_{+}e^{i\gamma} \subset S_{d,\delta} \cup \{ 0 \}$, where $\gamma$ depends on
$\epsilon$ and is chosen in such a way that
$\cos(k(\gamma - \mathrm{arg}(\epsilon))) \geq \delta_{1} > 0$, for some fixed $\delta_{1}$, for all
$\epsilon$ in a sector
$$ S_{d,\theta,R^{1/k}} = \{ \epsilon \in \mathbb{C}^{\ast} : |\epsilon| < R^{1/k} \ \ , \ \ |d - \mathrm{arg}(\epsilon) |
< \theta/2 \},$$
where $\frac{\pi}{k} < \theta < \frac{\pi}{k} + 2\delta$ and $0 < R < \delta_{1}/K$. The
function $\mathcal{L}^{d}_{k}(\mathcal{B}_{k}(\hat{X}))(\epsilon)$
is called the $k-$sum of the formal series $\hat{X}(t)$ in the direction $d$. It is bounded and holomorphic on the sector
$S_{d,\theta,R^{1/k}}$ and has the formal series $\hat{X}(\epsilon)$ as Gevrey asymptotic
expansion of order $1/k$ with respect to $\epsilon$ on $S_{d,\theta,R^{1/k}}$. This means that for all
$\frac{\pi}{k} < \theta_{1} < \theta$, there exist $C,M > 0$ such that
$$ ||\mathcal{L}^{d}_{k}(\mathcal{B}_{k}(\hat{X}))(\epsilon) - \sum_{p=0}^{n-1}
\frac{a_p}{p!} \epsilon^{p}||_{\mathbb{F}} \leq CM^{n}\Gamma(1+ \frac{n}{k})|\epsilon|^{n} $$
for all $n \geq 1$, all $\epsilon \in S_{d,\theta_{1},R^{1/k}}$.\medskip

Now, we state a cohomological criterion for $k-$summability of formal series with coefficients in Banach spaces (see
\cite{ba2}, p. 121 or \cite{hssi}, Lemma XI-2-6) which is known as the Ramis-Sibuya theorem in the literature. This result
is a crucial tool in the proof of our main result (Theorem 2).\medskip

\noindent {\bf Theorem (RS)} {\it Let $(\mathbb{F},||.||_{\mathbb{F}})$ be a Banach space over $\mathbb{C}$ and
$\{ \mathcal{E}_{p} \}_{0 \leq i \leq \varsigma-1}$ be a good covering in $\mathbb{C}^{\ast}$. For all
$0 \leq p \leq \varsigma - 1$, let $G_{p}$ be a holomorphic function from $\mathcal{E}_{p}$ into
the Banach space $(\mathbb{F},||.||_{\mathbb{F}})$ and let the cocycle $\Theta_{p}(\epsilon) = G_{p+1}(\epsilon) - G_{p}(\epsilon)$
be a holomorphic function from the sector $Z_{p} = \mathcal{E}_{p+1} \cap \mathcal{E}_{p}$ into $\mathbb{E}$
(with the convention that $\mathcal{E}_{\varsigma} = \mathcal{E}_{0}$ and $G_{\varsigma} = G_{0}$).
We make the following assumptions.\medskip

\noindent {\bf 1)} The functions $G_{p}(\epsilon)$ are bounded as $\epsilon \in \mathcal{E}_{p}$ tends to the origin
in $\mathbb{C}$, for all $0 \leq p \leq \varsigma - 1$.\medskip

\noindent {\bf 2)} The functions $\Theta_{p}(\epsilon)$ are exponentially flat of order $1/k$ on $Z_{p}$, for all
$0 \leq p \leq \varsigma-1$. This means that there exist constants $C_{p},A_{p}>0$ such that
$$ ||\Theta_{p}(\epsilon)||_{\mathbb{F}} \leq C_{p}e^{-A_{p}/|\epsilon|^{k}} $$
for all $\epsilon \in Z_{p}$, all $0 \leq p \leq \varsigma-1$.\medskip

Then, for all $0 \leq p \leq \nu - 1$, the functions $G_{p}(\epsilon)$ are the $k-$sums on $\mathcal{E}_{p}$ of a
common $k-$summable formal series $\hat{G}(\epsilon) \in \mathbb{F}[[\epsilon]]$.}

\subsection{Construction of $k-$summable formal series in the complex parameter of the initial value problem}

In this subsection, we establish the second main result of our work, namely the existence of a formal power series
in the parameter $\epsilon$ whose coefficients are bounded holomorphic
functions on the product of a sector with small radius centered at 0 and a strip in $\mathbb{C}^2$, that is a solution
of the equation (\ref{ICP_main}) and which is the common Gevrey asymptotic expansion of order $1/k$ of the actual solutions
$u_{p}(t,z,\epsilon)$ of (\ref{ICP_main}) constructed in Theorem 1.\medskip

\noindent The second main result of this work can be stated as follows.

\begin{theo} Let us assume that the hypotheses of Theorem 1 hold. Then, there exists a formal power series 
$$ \hat{u}(t,z,\epsilon) = \sum_{m \geq 0} h_{m}(t,z) \epsilon^{m}/m! $$
solution of the equation (\ref{ICP_main}), whose coefficients $h_{m}(t,z)$ belong to the Banach space $\mathbb{F}$ of
bounded holomorphic functions on
$(\mathcal{T} \cap D(0,h'')) \times H_{\beta'}$ equipped with supremum norm, where $h''>0$ is constructed in Theorem 1, and
such that the functions $u_{p}(t,z,\epsilon)$ defined in Theorem 1,
seen as holomorphic functions from $\mathcal{E}_{p}$ into $\mathbb{F}$, are its $k-$sums on the sectors $\mathcal{E}_{p}$, for
all $0 \leq p \leq \varsigma-1$. In other words, for all $0 \leq p \leq \varsigma-1$, there exist two constants $C_{p},M_{p}>0$ such that
\begin{equation}
\sup_{t \in \mathcal{T} \cap D(0,h''), z \in H_{\beta'}}
|u_{p}(t,z,\epsilon) - \sum_{m=0}^{n-1} h_{m}(t,z) \frac{\epsilon^m}{m!}| \leq C_{p}M_{p}^{n}\Gamma(1+\frac{n}{k})|\epsilon|^{n}
\end{equation} 
for all $n \geq 1$, all $\epsilon \in \mathcal{E}_{p}$.
\end{theo}
\begin{proof} We consider the family of functions $u_{p}(t,z,\epsilon)$, $0 \leq p \leq \varsigma-1$ constructed in Theorem 1.
For all $0 \leq p \leq \varsigma-1$, we define $G_{p}(\epsilon) := (t,z) \mapsto u_{p}(t,z,\epsilon)$, which is by construction a
holomorphic and bounded function from $\mathcal{E}_{p}$ into the Banach space $\mathbb{F}$ of bounded holomorphic functions on
$(\mathcal{T} \cap D(0,h'')) \times H_{\beta'}$ equipped with the supremum norm, where $\mathcal{T}$ is introduced in
Definition 5, $h''>0$ is set in Theorem 1 and
$\beta'>0$ is the width of the strip $H_{\beta'}$ on which the coefficients $c_{0}$ and $f$ are defined with respect to
$z$ (see (\ref{defin_c_0_f})). Bearing in mind the estimates (\ref{exp_small_difference_u_p}), we see that the cocycle
$\Theta_{p}(\epsilon) = G_{p+1}(\epsilon) - G_{p}(\epsilon)$ is exponentially flat of order $k$ on
$Z_{p} = \mathcal{E}_{p} \cap \mathcal{E}_{p+1}$, for any $0 \leq p \leq \varsigma-1$.

From Theorem (RS) stated above, there exists a formal power series $\hat{G}(\epsilon) \in \mathbb{F}[[\epsilon]]$ such that the
functions $G_{p}(\epsilon)$ are the $k-$sums on $\mathcal{E}_{p}$ of $\hat{G}(\epsilon)$ as $\mathbb{F}-$valued functions, for all
$0 \leq p \leq \varsigma-1$. We set
$$ \hat{G}(\epsilon ) = \sum_{m \geq 0} h_{m}(t,z) \epsilon^{m}/m! =: \hat{u}(t,z,\epsilon). $$
It remains to show that the formal series $\hat{u}(t,z,\epsilon)$ satisfies the main equation (\ref{ICP_main}). Since the functions
$G_{p}(\epsilon)$ are the $k-$sums of $\hat{G}(\epsilon)$, we have in particular that
\begin{equation}
\lim_{\epsilon \rightarrow 0, \epsilon \in \mathcal{E}_{p}}
\sup_{t \in \mathcal{T} \cap D(0,h''),z \in H_{\beta'}}|\partial_{\epsilon}^{m}u_{p}(t,z,\epsilon) - h_{m}(t,z)| = 0
\label{limit_deriv_order_m_of_up_epsilon}
\end{equation}
for all $0 \leq p \leq \varsigma-1$, all $m \geq 0$. Now, we choose some $p \in \{ 0, \ldots, \varsigma-1 \}$. By construction, the function
$u_{p}(t,z,\epsilon)$ is a solution of (\ref{ICP_main}). We take the derivative of order $m \geq 0$ with respect to $\epsilon$ on the
left and right handside of the equation (\ref{ICP_main}). From the Leibniz rule, we deduce that
$\partial_{\epsilon}^{m}u_{p}(t,z,\epsilon)$ verifies the following equation
\begin{multline}
Q(\partial_{z})(\partial_{t}\partial_{\epsilon}^{m}u_{p}(t,z,\epsilon)) = \sum_{m_{1}+m_{2}=m}
\frac{m!}{m_{1}!m_{2}!} (Q_{1}(\partial_{z})\partial_{\epsilon}^{m_1}u_{p}(t,z,\epsilon))
(Q_{2}(\partial_{z})\partial_{\epsilon}^{m_2}u_{p}(t,z,\epsilon))\\
+ \sum_{l=1}^{D} \left( \sum_{m_{1}+m_{2}=m} \frac{m!}{m_{1}!m_{2}!}
\partial_{\epsilon}^{m_1}(\epsilon^{\Delta_l})t^{d_l}\partial_{t}^{\delta_l}
R_{l}(\partial_{z})\partial_{\epsilon}^{m_2}u_{p}(t,z,\epsilon) \right) \\
+ \sum_{m_{1}+m_{2}=m} \frac{m!}{m_{1}!m_{2}!} \partial_{\epsilon}^{m_1}c_{0}(t,z,\epsilon)
R_{0}(\partial_{z})\partial_{\epsilon}^{m_2}u_{p}(t,z,\epsilon) + \partial_{\epsilon}^{m}f(t,z,\epsilon)
\label{ICP_main_derivative_epsilon}
\end{multline}
for all $m \geq 0$, all $(t,z,\epsilon) \in (\mathcal{T} \cap D(0,h'')) \times H_{\beta'} \times \mathcal{E}_{p}$. If we let $\epsilon$ tend
to zero in (\ref{ICP_main_derivative_epsilon}) and if we use (\ref{limit_deriv_order_m_of_up_epsilon}), we get the recursion
\begin{multline}
Q(\partial_{z})(\partial_{t}h_{m}(t,z)) = \sum_{m_{1}+m_{2}=m} \frac{m!}{m_{1}!m_{2}!}(Q_{1}(\partial_{z})h_{m_1}(t,z))
(Q_{2}(\partial_{z})h_{m_2}(t,z))\\
+ \sum_{l=1}^{D} \frac{m!}{(m-\Delta_{l})!}t^{d_l}\partial_{t}^{\delta_l}R_{l}(\partial_{z})h_{m-\Delta_{l}}(t,z)\\
+ \sum_{m_{1}+m_{2}=m} \frac{m!}{m_{1}!m_{2}!}(\partial_{\epsilon}^{m_1}c_{0})(t,z,0)R_{0}(\partial_{z})h_{m_2}(t,z) +
(\partial_{\epsilon}^{m}f)(t,z,0) \label{rec_h_m}
\end{multline}
for all $m \geq \max_{1 \leq l \leq D} \Delta_{l}$, all $(t,z) \in (\mathcal{T} \cap D(0,h'')) \times H_{\beta'}$. Since the functions
$c_{0}(t,z,\epsilon)$ and $f(t,z,\epsilon)$ are analytic with respect to $\epsilon$ at 0, we know that
\begin{equation}
c_{0}(t,z,\epsilon) = \sum_{m \geq 0} \frac{(\partial_{\epsilon}^{m}c_{0})(t,z,0)}{m!}\epsilon^{m} \ \ , \ \
f(t,z,\epsilon) = \sum_{m \geq 0} \frac{(\partial_{\epsilon}^{m}f)(t,z,0)}{m!}\epsilon^{m} \label{Taylor_c0_f}
\end{equation}
for all $\epsilon \in D(0,\epsilon_{0})$, all $z \in H_{\beta'}$. On other hand, one can check by direct inspection from the recursion
(\ref{rec_h_m}) and the expansions (\ref{Taylor_c0_f}) that the formal series
$\hat{u}(t,z,\epsilon) = \sum_{m \geq 0} h_{m}(t,z)\epsilon^{m}/m!$ solves the equation (\ref{ICP_main}).
\end{proof}

\end{document}